\documentclass[noams]{compositio}

\usepackage{mathrsfs}
\usepackage{amsmath}
\usepackage{amssymb}

\usepackage{tikz-cd}

\newcommand{\bb}[1]{\mathbb{#1}}
\newcommand{\n}[1]{\mathcal{#1}}
\newcommand{\bbf}[1]{\mathbf{#1}}
\newcommand{\f}[1]{\mathfrak{#1}}
\newcommand{\s}[1]{\mathscr{#1}}
\newcommand{\Hom}{\mbox{\normalfont Hom}}

\DeclareMathOperator{\ES}{ES}
\DeclareMathOperator{\GL}{GL}
\DeclareMathOperator{\Gal}{Gal}
\DeclareMathOperator{\FL}{FL}
\DeclareMathOperator{\Fl}{\s{F}\!\ell}
\DeclareMathOperator{\Fil}{Fil}
\DeclareMathOperator{\gr}{gr}
\DeclareMathOperator{\HT}{HT}
\DeclareMathOperator{\Spa}{Spa}
\DeclareMathOperator{\Spec}{Spec}
\DeclareMathOperator{\Spf}{Spf}
\DeclareMathOperator{\Sym}{Sym}
\DeclareMathOperator{\St}{St}
\DeclareMathOperator{\BGG}{BGG}
\DeclareMathOperator{\dlog}{dlog}
\DeclareMathOperator{\IW}{\n{I}\!w}
\DeclareMathOperator{\Iw}{Iw}
\DeclareMathOperator{\ev}{ev}

\DeclareMathOperator{\DR}{DR}
\DeclareMathOperator{\Cor}{Cor}
\DeclareMathOperator{\Res}{Res}
\DeclareMathOperator{\GLa}{\n{G}\!\n{L}}
\DeclareMathOperator{\Isom}{Isom}
\DeclareMathOperator{\diag}{diag}
\DeclareMathOperator{\pr}{pr}
\DeclareMathOperator{\Lie}{Lie}
\DeclareMathOperator{\Tr}{Tr}
\DeclareMathOperator{\id}{id}
\DeclareMathOperator{\hocolim}{hocolim}

\DeclareMathOperator{\an}{\begin{scriptsize}
an
\end{scriptsize}}
\DeclareMathOperator{\et}{\begin{scriptsize}
\acute{e}t
\end{scriptsize}}
\DeclareMathOperator{\proet}{\begin{scriptsize}
pro\acute{e}t
\end{scriptsize}}
\DeclareMathOperator{\ket}{\begin{scriptsize}
k\acute{e}t
\end{scriptsize}}
\DeclareMathOperator{\proket}{\begin{scriptsize}
prok\acute{e}t
\end{scriptsize}}

\DeclareMathOperator{\alg}{\begin{scriptsize}
alg
\end{scriptsize}}
\DeclareMathOperator{\dR}{\begin{scriptsize}
dR
\end{scriptsize}}
\DeclareMathOperator{\cyc}{\begin{scriptsize}
cyc
\end{scriptsize}}
\DeclareMathOperator{\sm}{\begin{scriptsize}
sm
\end{scriptsize}}
\DeclareMathOperator{\mo}{\begin{scriptsize}
mod
\end{scriptsize}}
\DeclareMathOperator{\ord}{\begin{scriptsize}
ord
\end{scriptsize}}
\DeclareMathOperator{\can}{\begin{scriptsize}
can
\end{scriptsize}}
\DeclareMathOperator{\naive}{\begin{scriptsize}
naive
\end{scriptsize}}

\DeclareMathOperator{\un}{\begin{scriptsize}
un
\end{scriptsize}}

\newcommand{\OBdr}{\s{O}\!\bb{B}_{\dR,\log}}

\newcommand{\OC}{\s{O}\!\bb{C}_{\log}}

\newtheorem{theo}{Theorem}[subsection]
\newtheorem{prop}[theo]{Proposition}
\newtheorem{lem}[theo]{Lemma}
\newtheorem{cor}[theo]{Corollary}

\theoremstyle{definition}

\newtheorem{definition}[theo]{Definition}

\newtheorem{convention}[theo]{Convention}

\theoremstyle{remark}
\newtheorem{remark}[theo]{Remark}

\begin{document}

\title{$p$-adic Eichler-Shimura maps for the modular curve}

\author{Juan Esteban Rodr\'iguez Camargo}
\email{rodriguez@mpim-bonn.mpg.de}
\address{Max Planck Institute for Mathematics, Vivatsgasse 7, 53111 Bonn-Germany}

\begin{abstract}
We give a new proof of Faltings's $p$-adic Eichler-Shimura decomposition  of the modular curves via BGG methods and the Hodge-Tate period map.  The key property   is the relation between the Tate module and the Faltings extension,  which was  already used in the original proof.   Then,  we construct overconvergent  Eichler-Shimura (ES) maps   for the modular curves   providing  ``the second half'' of the overconvergent ES map of Andreatta-Iovita-Stevens.   We use higher Coleman  theory  on  the modular curve developed by Boxer-Pilloni to show that the small slope part of the ES maps interpolates the classical $p$-adic Eichler-Shimura decompositions.   Finally,  we prove that the overconvergent ES maps are compatible with Poincar\'e and Serre pairings. 
\end{abstract}

\maketitle

\section{Introduction}

Let $p$ be a prime number, $\bb{A}_{\bb{Q}}^{\infty}$  the finite ad\`eles of  $\bb{Q}$,   $\bb{A}^{\infty,p}_{\bb{Q}}$ the finite  prime-to-$p$ ad\`eles,  and $\bb{Z}_p$ the ring of $p$-adic integers.  Let $\bb{C}_p$ be the $p$-adic completion of an algebraic closure of $\bb{Q}_p$, and $G_{\bb{Q}_p}=\Gal(\bb{C}_p/\bb{Q}_p)$  the absolute Galois group.  From now on, we fix a neat compact open subgroup $K^p\subset \GL_2(\bb{A}^{\infty,p}_{\bb{Q}})$.   Given  $K_p\subset \GL_2(\bb{Q}_p)$  an open compact subgroup, we denote by  $Y_{K_p}^{\alg}$ the modular curve over $\Spec \bb{Q}_p$ of level $K^pK_p\subset \GL_2(\bb{A}^{\infty}_{\bb{Q}})=\GL_2(\bb{A}^{\infty,p}_{\bb{Q}})\times \GL_2(\bb{Q}_p)$,  and by $X_{K_p}^{\alg}$ its compactification by adding cusps.   Let $Y_{K_p}$ and $X_{K_p}$ be the rigid analytic varieties attached to the modular curves, seen as adic spaces over $\Spa(\bb{Q}_p, \bb{Z}_p)$, cf.  \cite{HuberEtaleCohomology}. Let $D= X_{K_p}\backslash Y_{K_p}$ be the cusp divisor, we endow $X_{K_p}$ with the $\log$-structure defined by $D$.  \\

Given an fs $\log$ adic space  $Z$ and $?\in \{\an,\et,\ket,\proet,\proket\}$,   we denote by $Z_{?}$ its analytic,  \'etale,  Kummer-\'etale,  pro\'etale and pro-Kummer-\'etale sites respectively,  see  \cite{ScholzeHodgeTheory2013} and \cite{DiaoLogarithmic2019}.   \\

In \cite{FaltingsHodgeModular1987},  Faltings described  the Hodge-Tate decomposition  of the  \'etale cohomology (with coefficients) of the modular curve $Y_{K_p }$.    More precisely,  let $E$ be the universal elliptic curve over $Y_{K_p}$,  it admits an  extension to a semi-abelian  adic space $E^{sm}$ over $X_{K_p}$ (cf. \cite{DeligneRappLesSchemasCourbes}).  Let $e:  X_{K_p}\rightarrow  E^{sm}$ be the unit section,   $\omega_E= e^* \Omega^1_{ E^{sm}/X}$ the modular sheaf and $T_p E =\varprojlim_{n} E[p^n]$  the Tate module over $Y_{K_p}$.   We have the following theorem
\begin{theo}[Faltings]
\label{TheoIntro1}
Let  $k\geq 0$,  there exists a Galois and Hecke equivariant isomorphism 
\begin{equation}
\label{eqESDecompositionIntro}
H^1_{\et}(Y_{K_p,\bb{C}_p},  \Sym^k T_p E ) \otimes_{\bb{Q}_p}\bb{C}_p(1)= H^0_{\an}(X_{K_p,\bb{C}_p},  \omega_E^{k+2})\oplus H^1_{\an}(X_{K_p,\bb{C}_p},  \omega_E^{-k})(k+1)
\end{equation}
called the Eichler-Shimura (ES) decomposition. 
\end{theo}  
The first result of this paper is a new proof of Faltings's  Eichler-Shimura   decomposition using BGG methods and the Hodge-Tate period map.  Our proof is the pro\'etale  analogue  of the BGG decomposition for  the de Rham cohomology of Faltings-Chai    \cite[Ch. 5 Theorem  5.5]{faltings1990degeneration}.  Let us develop  the ideas behind:   \\

 Let $X_{\infty}:= \varprojlim_{K_p} X_{K_p} $  be   Scholze's   perfectoid modular curve and   $\pi_{\HT}:  X_{\infty}\rightarrow \bb{P}^1_{\bb{Q}_p}$ the  Hodge-Tate period map   \cite{ScholzeTorsion2015}.    The morphism $\pi_{\HT}$ is $\GL_2(\bb{Q}_p)$-equivariant where we see $\bb{P}^1_{\bb{Q}_p}$ as the left quotient  of $\GL_2$ by the upper triangular Borel $\bbf{B}$.  Let $\pi_{K_p}: X_{\infty}\rightarrow X_{K_p}$ be the natural map,   we can see $X_{\infty}$ as a pro-Kummer-\'etale $K_p$-torsor over $X_{K_p}$.   We let  $\widehat{\s{O}}_{X_{K_p}}$ denote the $p$-adic completion of the structural sheaf over $X_{K_p, \proket}$.  Let $\bb{Z}_p(1)= \varprojlim_{n} \mu_{p^n}$ be the Tate twist and $\widehat{\s{O}}_{X_{K_p}}(i)$ the $i$-th twist of $\widehat{\s{O}}_{X_{K_p}}$.      By \cite[Theorem 4.6.1]{DiaoLogarithmic2019}, $T_pE$ admits a natural extension to the pro-Kummer-\'etale site of $X_{K_p}$ which we denote  in the same way.    From now on, we fix the level $K_p$ and write  $Y=Y_{K_p}$ and  $X=X_{K_p}$.  \\

 The map $\pi_{\HT}$ is defined from the Hodge-Tate exact sequence 
\begin{equation}
\label{eqHTExactIntro}
0 \rightarrow \omega_E^{-1}\otimes_{\s{O}_X} \widehat{\s{O}}_X(1) \xrightarrow{\HT^{\vee}} T_p E \otimes_{\bb{Z}_p} \widehat{\s{O}}_X \xrightarrow{\HT} \omega_E \otimes_{\s{O}_X} \widehat{\s{O}}_X \rightarrow 0,
\end{equation}
which is the variation in families of the  Hodge-Tate decomposition for elliptic curves (cf.  \cite{Tatepdivisible}). More precisely, let    $\Psi: \bb{Z}_p^2\xrightarrow{\sim} T_pE$ be the  universal trivialization over $X_{\infty}$, then \eqref{eqHTExactIntro} defines a line subbundle of $\s{O}_{X_{\infty}}^{\oplus 2}$ which induces the map $\pi_{\HT}$.   \\

The $\GL_2(\bb{Q}_p)$-equivariance of $\pi_{\HT}$  recovers  (\ref{eqHTExactIntro}) from a short exact sequence of $\GL_2$-equivariant sheaves over $\bb{P}^1_{\bb{Q}_p}$.  Indeed, the presentation $\bb{P}^1_{\bb{Q}_p}= \bbf{B} \backslash \GL_2$ induces an equivalence between algebraic $\bbf{B}$-representations and $\GL_2$-equivariant vector bundles over $\bb{P}^1_{\bb{Q}_p}$. More explicitly, let $V$ be a $\bbf{B}$-representation, then one defines the vector bundle $\n{V}= \GL_2 \times^{\bbf{B}} V = \bbf{B}\backslash (\GL_2\times V)$, where in the right  term $\bbf{B}$ acts diagonally.  Let $\s{F}$ be a $K_p$-equivariant sheaf over $\bb{P}^1_{\bb{Q}_p}$, we shall identify $\pi_{\HT}^*(\s{F})$ with the pro-Kummer-\'etale sheaf over $X_{K_p}$ obtained by descent from the $K_p$-equivariant pullback over $X_{\infty}$. \\

  Let $\bbf{T}\subset \bbf{B}$ be the diagonal torus and let $\kappa=(k_1,k_2) \in  X^*(\bbf{T})$  be a character, we denote by  $\bb{Q}_p(\kappa)$ the representation defined by $\kappa$ and consider it  as a $\bbf{B}$-representation by letting the unipotent radical act trivially. We say that $\kappa=(k_1,k_2)$ is dominant if $k_1\geq k_2$. Given $\kappa$ a dominant character we let $V_{\kappa}= \Sym^{k_1-k_2}\St\otimes \det^{k_2}$ be the irreducible representation of $\GL_2$ of highest weight $\kappa$, we denote by $V_{\kappa,\et}$ the associated pro-Kummer-\'etale sheaf obtained by descent from $X_{\infty}$.     Let $W=\{1,w_0\}$ be the Weyl group of $\GL_2$.  We denote by $\s{L}(\kappa)$ the $\GL_2$-equivariant sheaf over $\bb{P}^1_{\bb{Q}_p}$  given by $\bbf{B}\backslash( \GL_2\times w_0(\kappa))$.   The  standard representation   $\St$ has a  $\bbf{B}$-filtration 
\begin{equation}
\label{eqHTintro}
0\rightarrow \bb{Q}_p(1,0)  \rightarrow  \St \rightarrow \bb{Q}_p(0,1) \rightarrow 0.
\end{equation}
By construction, the pullback of (\ref{eqHTintro}) by $\pi_{\HT}$   is equal to the Hodge-Tate-exact sequence \eqref{eqHTExactIntro}.  In particular, $\pi_{\HT}^*(  \St \otimes \s{O}_{\bb{P}^1_{\bb{Q}_p}})=  T_p E \otimes \widehat{\s{O}}_X$ and  $\pi_{\HT}( \s{L}(\kappa))= \omega_E^{k_1-k_2}\otimes_{\s{O}_X} \widehat{\s{O}}_X(k_2)$. A natural question arises, namely, to describe  the pullbacks by $\pi_{\HT}$ of the  $\GL_2$-equivariant vector bundles over $\bb{P}^1_{\bb{Q}_p}$. We already know that the pullback of those vector bundles constructed from characters of $\bbf{T}$ are related with modular sheaves, it is left to understand the $\GL_2$-equivariant sheaves constructed from non-semisimple representations of $\bbf{B}$.    \\

Let  $\s{O}(\bbf{B})$ be the ring of algebraic functions of $\bbf{B}$ endowed with the right regular action, note that any algebraic representation of $\bbf{B}$ occurs in $\s{O}(\bbf{B})$. The presentation $\bbf{B}= \bbf{T}\ltimes \bbf{N}$ as a semidirect product  gives an isomorphism  $\s{O}(\bbf{B})= \s{O}(\bbf{T})\otimes \s{O}(\bbf{N})$, where $\bbf{B}$ acts on the first factor by the  right regular action of $\bbf{T}$, and it acts on the second factor under the formula $(n,b)\mapsto t_{b}^{-1}nt_{b} n_b$, with $(n,b)\in \bbf{N}\times \bbf{B}$ and $b=(t_b,n_b)\in \bbf{T}\ltimes \bbf{N}$.   We have the following theorem:

\begin{theo}[Theorem \ref{PropFaltingsExtension}]
\label{TheoFEIntro}
Let $\underline{\s{O}}(\bbf{N})$ be the $\GL_2$-equivariant quasi-coherent sheaf over $\bb{P}^1_{\bb{Q}_p}$ associated to $\s{O}(\bbf{N})$. Let $\OBdr$ be the geometric de Rham period sheaf of \cite{ScholzeHodgeTheory2013} and \cite{DiaoLogarithmic2019}, and let $\OC= \gr^0 \OBdr$ be the Hodge-Tate sheaf.   We have a natural isomorphism of pro-Kummer-\'etale sheaves over $X$
\[
\pi^{*}_{\HT}(\underline{\s{O}}(\bbf{N})) = \OC. 
\]
\end{theo}

As a corollary one obtains the Eichler-Shimura decompositions

\begin{theo}[Theroem \ref{TheoBGGproEtale}]
\label{TheoBGGIntro} 
Let $\kappa=(k_1,k_2)$ be a dominant character and  set $\alpha= (1,-1)\in X^*(\bbf{T})$.   Let  $\BGG(\kappa)$ be the dual BGG complex of weight $\kappa$ (Proposition \ref{PropBGGclassic})
\[
0 \rightarrow V_{\kappa} \rightarrow V(\kappa) \rightarrow V(w_0(\kappa)-\alpha)\rightarrow 0.
\]
We denote by  $\underline{\BGG(\kappa)}$ the associated $\GL_2$-equivariant complex over $\bb{P}^1_{\bb{Q}_p}$. Then,  the pullback of $\underline{\BGG(\kappa)}$ by $\pi_{\HT}$ is the  short exact sequence 
\[
0 \rightarrow \Sym^{k_1-k_2} T_pE  \otimes \widehat{\s{O}}_X(k_2) \rightarrow \omega_E^{k_2-k_1}\otimes \OC(k_1) \rightarrow \omega_E^{k_1-k_2+2}\otimes \OC(k_2-1)\rightarrow 0.
\]
  Furthermore,  let $\lambda:  X_{\bb{C}_p,\proket}\rightarrow X_{ \bb{C}_p,\an}$ be the projection of sites.  Then 
\[
R\lambda_{*} (\Sym^k T_pE \otimes \widehat{\s{O}}_X(1))= \omega_E^{-k}\otimes \bb{C}_p(k+1)[0] \oplus \omega_E^{k+2} [-1].
\]
Taking $H^1$-cohomology in the analytic site of $X_{\bb{C}_p}$ we recover the Eichler-Shimura decomposition of Theorem \ref{TheoIntro1}. 
\end{theo}

The proof of Theorem \ref{TheoFEIntro} follows from the isomorphism between  Faltings extension $\gr^1 \OBdr^+$ and the sheaf $T_pE \otimes \widehat{\s{O}}_X \otimes \omega_E$.    This  isomorphism  was already known  by Faltings,  and used in his original  proof of Theorem \ref{TheoIntro1}.   Our new proof provides a more explicit definition of the Eichler-Shimura maps in terms of cocycles and can be  generalized to \textit{any} Shimura variety.   Moreover, using this method one easily deduces  the degeneration of the spectral sequence appearing in \cite{FaltingsHodgeModular1987},   as well as its natural  splitting using simple properties of the dual BGG resolution.    It is worth to mention that the isomorphism between the twist of the Tate module and the Faltings extension was used by Lue Pan in \cite{pan2020locally} to compute the relative Sen operator of the modular curve.  \\

The second goal of this paper is the interpolation of the Eichler-Shimura decomposition (\ref{eqESDecompositionIntro}).   The $H^0$-map  of the overconvergent  Eichler-Shimura maps was previously constructed by Andreatta-Iovita-Stevens in \cite{AISOvShimura2015}.  The strategy followed in this document is close to the construction of  the Eichler-Shimura map for Shimura curves in \cite{ChojeckiOvShimuraCurves2017}.   Roughly speaking, we take pullbacks by $\pi_{\HT}$  of  certain  locally analytic sheaves over $\bb{P}^1_{\bb{Q}_p}$.   In this way,  we  interpolate all the terms appearing in the Hodge-Tate exact sequence (\ref{eqHTExactIntro}):    we get overconvergent modular sheaves whose cohomology are the object of study in  higher Coleman theory developed by Boxer-Pilloni \cite{BoxerPilloniHigher2020, BoxerPilloniHigherColeman2020}.  The interpolation of the symmetric powers of the Tate module will be given by locally analytic principal series or locally analytic distributions as in \cite{AshStevensDeformations2008}. Finally,  the Hodge-Tate maps $\HT$ and $\HT^\vee$ can be put in families,  obtaining the $\dlog$ map of \cite{AISOvShimura2015} as a particular case.  \\

Let us sketch the main steps of the construction.   Let $n\geq 1$ be an integer and  let
\[
\Iw_n:= \left( \begin{matrix}  \bb{Z}_p^{\times} & \bb{Z}_p \\ p^n \bb{Z}_p & \bb{Z}_p^\times \end{matrix} \right)
\]
be the Iwahori group modulo $p^n$. From now on,   we will take $X=X_{\Iw_n}$.   Let $\epsilon \geq \delta \geq n$  be rational numbers and    $(R,R^+)$  a uniform Tate algebra over $\bb{Q}_p$ which we may assume to be sheafy (i.e.  such that the pre-sheaf of rational functions in $\Spa(R,R^+)$ is an actual sheaf).  Let $T=\bbf{T}(\bb{Z}_p)$ denote the $\bb{Z}_p$-points of the diagonal torus and $\chi=(\chi_1,\chi_2): T\rightarrow R^{+,\times}$ a $\delta$-analytic character (cf. Lemma  \ref{PropositionCharacterLocallyAnalytic}).      We let $R\widehat{\otimes}\widehat{\s{O}}_X$ be the $p$-adically complete tensor product of $R$ and the completed sheaf $\widehat{\s{O}}_X$.   Given a character $\lambda :  \bb{Z}_p^\times \to R^{+,\times}$ we denote by $R(\lambda)$ the $G_{\bb{Q}_p}$-module induced by the composition $G_{\bb{Q}_p}\xrightarrow{\chi_{\cyc}} \bb{Z}_p^\times \xrightarrow{\lambda} R^{\times}$.  Finally,  we write $\widehat{\s{O}}_X(\lambda):= R(\lambda)\widehat{\otimes}\widehat{\s{O}}_X$.    \\

We begin with the construction of  all the  sheaves  over $\bb{P}^1_{\bb{Q}_p}$:  for $w\in W=\{1,w_0\}$ we define a family of  overconvergent neighbourhoods $\{U_w(\epsilon )\Iw_n\}_{\epsilon \geq n}$ of $w\Iw_n$  in $\bb{P}^1_{\bb{Q}_p}$.   The affinoid spaces $U_w(\epsilon)\Iw_n$ admit  sections of the quotient map $\GL_2\rightarrow  \bb{P}^1_{\bb{Q}_p}$.  In particular,  the natural  $\bbf{T}$-torsor $\bbf{N}\backslash \GL_2\rightarrow \bb{P}^1_{\bb{Q}_p}$, where $\bbf{N}$ is the unipotent radical of $\bbf{B}$,  has a trivialization over $U_{w}(\epsilon)\Iw_n$.   We define a $R\widehat{\otimes} \s{O}_{\bb{P}^1_{\bb{Q}_p}}$ line bundle $\s{L}(\chi)$ in the analytic site of $U_w(\epsilon)\Iw_n$ in the same way we have defined the line bundles $\s{L}(\kappa)$ for $\kappa\in X^*(\bbf{T})$.     Then,  we define the space of $\delta$-analytic principal series of weight $\chi$ to be the $R$-Banach space 
\[
A^{\delta}_{\chi}= \Gamma(U_{w_0}(\delta) \Iw_n,  \s{L}(\chi)).
\]
We define the $\delta$-analytic distributions  $D^{\delta}_{\chi}$ to  be the continuous dual of $A^{\delta}_{\chi}$ endowed with the weak topology.  The space $A^{\delta}_{\chi}$ has a natural action of $\Iw_n$,  so that  it defines  a constant $\Iw_n$-equivariant  sheaf on  $\bb{P}^1_{\bb{Q}_p}$.   We let  $A^{\delta}_{\chi,\et}$ and  $D_{\chi,\et}^{\delta}$ denote the pro-Kummer-\'etale sheaves over $X$ obtained by descent from the topological $K_p$-equivariant sheaves over $X_{\infty}$. \\

In Proposition \ref{PropOVBGGmaps} we construct maps 
\begin{gather*}
 R(\chi)  \xrightarrow{\iota} A^{\delta}_{\chi} \mbox{ equivariant for the action of } \bbf{B}(\bb{Z}_p)\cap \Iw_n  \\ 
 A^{\delta}_{\chi}\xrightarrow{\ev_{w_0}} R(\chi)  \mbox{ equivariant for the action of } w_0^{-1} \bbf{B}(\bb{Z}_p)w_0 \cap \Iw_n,
\end{gather*}
with $\iota$ being the highest weight vector,  and $\ev_{w_0}$ the evaluation at $w_0$.   We prove that these maps give rise morphisms of  $\Iw_n$-equivariant sheaves 
\begin{equation}
\label{eqovBGGmapsIntro}
\begin{gathered}
\s{L}(w_0(\chi))  \rightarrow A^{\delta}_{\chi}\widehat{\otimes}_{\bb{Q}_p} \s{O}_{\bb{P}^1_{\bb{Q}_p}} \mbox{ over } U_{1}(\epsilon)\Iw_n \\ 
A^{\delta}_{\chi}\widehat{\otimes}_{\bb{Q}_p}  \s{O}_{\bb{P}^1_{\bb{Q}_p}} \rightarrow   \s{L}(\chi)   \mbox{ over } U_{w_0}(\epsilon)\Iw_n. 
\end{gathered}
\end{equation}\\

The next step is to translate all the previous constructions to the modular curve $X$.  We start by defining   the strict neighbourhoods of the $w$-ordinary locus $\{X_{w}(\epsilon)\}_{\epsilon>n}$;   they are equal to $\pi_{\Iw_n}( \pi_{\HT}^{-1}(U_w(\epsilon)\Iw_n) )$.    The second object we descend to $X$ are the overconvergent modular sheaves $\omega_E^{\chi}$;  they are  $R\widehat{\otimes} \s{O}_{X}$-line bundles in the \'etale or analytic site of $X_{w}(\epsilon)$.  We refer to \cite{BoxerPilloniHigherColeman2020} for the general construction of these sheaves.    The dictionary provided by $\pi_{\HT}$ gives then (see Corollary \ref{PropLtoOmega}) 
\begin{equation}
\label{eqPullbackovLineIntro}
\pi_{\HT}^*(\s{L}(\chi))= \omega_E^{\chi}\otimes \widehat{\otimes}  \widehat{\s{O}}_X(\chi_2).
\end{equation}\\

We continue with the pullback of the $\delta$-analytic principal series and distributions,  seen as $\Iw_n$-equivariant sheaves  over  $\bb{P}^1_{\bb{Q}_p}$. By definition one has  $\pi_{\HT}^*(A_{\chi}^{\delta}\widehat{\otimes} \s{O}_{\bb{P}^1_{\bb{Q}_p}})= A^{\delta}_{\chi,\et}\widehat{\otimes} \widehat{\s{O}}_X$ (resp. for $D^{\delta}_{\chi}$).    Finally,  we pullback the maps (\ref{eqovBGGmapsIntro}) obtaining overconvergent Hodge-Tate maps of pro-Kummer-\'etale sheaves
\begin{gather*}
\omega_E^{w_0(\chi)} \widehat{\otimes}  \widehat{\s{O}}_X(\chi_1) \xrightarrow{\HT^\vee} A^{\delta}_{\chi,\et}\widehat{\otimes} \widehat{\s{O}}_X \mbox{ over } X_{1}(\epsilon) \\ 
A^{\delta}_{\chi,\et}\widehat{\otimes}\widehat{\s{O}}_X \xrightarrow{\HT} \omega_E^{\chi} \widehat{\otimes} \widehat{\s{O}}_X(\chi_2) \mbox{ over } X_{w_0}(\epsilon),
\end{gather*}
similarly for $D^{\delta}_{\chi}$.  Taking pro-Kummer-\'etale cohomology one obtains the following:
\begin{theo}[Theorem \ref{TheoMain}]
\label{TheoOVESIntro}
There are overconvergent Eichler-Shimura maps 
\begin{equation}
\label{eqESsequenceIntro}
0\rightarrow H^1_{1,c}(X_{\bb{C}_p},  \omega_E^{w_0(\chi)})_{\epsilon}(\chi_1) \xrightarrow{ES_{A}^{\vee}}  H^1_{\proket}(X_{\bb{C}_p},  A^{\delta}_{\chi,\et}\widehat{\otimes} \widehat{\s{O}}_X) \xrightarrow{ES_{A}} H^0_{w_0}(X_{\bb{C}_p},  \omega_E^{\chi+\alpha})_{\epsilon}(\chi_2-1)\rightarrow 0
\end{equation}
satisfying the following properties: 
\begin{enumerate} 

\item The composition  $ES_{A}\circ ES^{\vee}_{A}$  is zero.

\item  Assume that $\n{V}=\Spa(R,R^+)$ is an affinoid subspace of the weight space $\n{W}_T$ of $T=\bbf{T}(\bb{Z}_p)$, and let $\kappa=(k_1,k_2)\in \n{V}$ be a dominant weight of $\bbf{T}$.  Let $\alpha=(1,-1)\in X^*(\bbf{T})$ and let  $\chi=\chi^{un}_{\n{V}}$ be the universal character  of $\n{V}$. Then there is a commutative diagram
 
\begin{center}
\begin{tikzpicture}[commutative diagrams/every diagram]
 \node (P0) at (0,2) {$H^1_{1,c}(X_{\bb{C}_p},\omega_E^{w_0(\chi)})_{\epsilon}(\chi_1) $};
 \node (P1) at (0,0) {$H^1_{1,c}(X_{\bb{C}_p},\omega_E^{w_0(\kappa)})_{\epsilon}(k_1) $};
 \node (P2) at (0,-2) {$H^1_{\an}(X_{\bb{C}_p},\omega_E^{w_0(\kappa)})(k_1) $};
 
 \node (P3) at (5,2) {$ H^1_{\proket}(X_{\bb{C}_p},  A^{\delta}_{\chi,\et} \widehat{\otimes}\widehat{\s{O}}_X)$};
 \node (P4) at (5,0) {$H^1_{\proket}(X_{\bb{C}_p},  A^{\delta}_{\kappa,\et} \widehat{\otimes}\widehat{\s{O}}_X)$};
 \node (P5) at (5,-2) {$H^1_{\proket}(X_{\bb{C}_p}, V_{\kappa,\et})\otimes \bb{C}_p $}; 
 
 \node (P6) at (10,2) {$H^0_{w_0}(X_{\bb{C}_p}, \omega_E^{\chi+\alpha})_{\epsilon}(\chi_2-1)$};
 \node (P7) at (10,0) {$H^0_{w_0}(X_{\bb{C}_p},\omega_E^{\kappa+\alpha})_{\epsilon}(k_2-1)$};
 \node (P8) at (10,-2) {$ H^0_{\an}(X_{\bb{C}_p}, \omega_E^{\kappa+\alpha})(k_2-1)  $};

\path[commutative diagrams/.cd, every arrow, every label] 
   (P0) edge node {$ES^\vee_{\n{A}}$}  (P3) 
   (P3) edge node {$ES_{\n{A}}$}  (P6)
   (P1) edge node {} (P4)
   (P4) edge node {} (P7)
   (P2) edge node {$ES^\vee$} (P5) 
   (P5) edge node {$ES$} (P8)
   
   (P0) edge node {} (P1)
   (P3) edge node {} (P4)
   (P6) edge node {} (P7)
   
   (P1) edge node {$\Cor$} (P2)
   (P5) edge node {} (P4)
   (P8) edge node [swap] {$\Res$} (P7);
 
\end{tikzpicture}
\end{center}

\item The maps of (2) are Galois and $U_p^t$-equivariant with respect to the good nomalizations of the $U_p^t$-operators.   In particular,   it restricts to the finite slope part with respect to the  $U_p^t$-action. 

\item  Let $h<k_1-k_2+1$.  There exists an open affinoid $\n{V}'\subset\n{V}$ containing $\kappa$ such that the $(\leq h)$-slope part  of  the restriction of  \eqref{eqESsequenceIntro}  to $\n{V}'$ is a short exact sequence of finite free $\bb{C}_p\widehat{\otimes}_{\bb{Q}_p} \s{O}(\n{V}')$-modules.    

\item Keep  the hypothesis of (4),  and let $\chi$ be the universal character of $\n{V}'$.  Let $\widetilde{\chi}= \chi_1-\chi_2+1 : \bb{Z}_p^\times \rightarrow R^{+,\times}$, and $b= \frac{d}{dt}|_{t=1} \widetilde{\chi}(t)$.  Then  we have a  Galois-equivariant split after inverting $b$
\[
H^1_{\proket}(X_{\bb{C}_p}, A^{\delta}_{\chi,\et}\widehat{\otimes}\widehat{\s{O}}_X)^{\leq h}_{b} =[H^1_{1,c}(X_{\bb{C}_p},  \omega_E^{w_0(\chi)})^{\leq h}_{\epsilon}(\chi_1)]_b \oplus [H^0_{w_0}(X_{\bb{C}_p},   \omega_E^{\chi+\alpha})^{\leq h}_{\epsilon}(\chi_2-1)]_b.
\] 

\end{enumerate}
\end{theo}
\begin{remark}

\begin{enumerate}

\item The group $H^0_w(X_{\bb{C}_p},  -)_{\epsilon}$ is the overconvergent cohomology   and $H^1_{w,c}(X_{\bb{C}_p},-)_{\epsilon}$  the overconvergent cohomology with compact support around the $w$-ordinary locus of $X$,  see   \cite{BoxerPilloniHigher2020} and Definition \ref{DefiSupportOvcohomologies} down below.

\item A similar statement holds for the distribution sheaves $D^{\delta}_{\chi,\et}$,  in this case the overconvergent Eichler-Shimura map of \cite{AISOvShimura2015} is $ES_{D}$.  

\item We also  prove the theorem for the pro\'etale cohomology with compact support of $A_{\chi,\et}^{\delta}$ and $D_{\chi,\et}^{\delta}$.  

\item Note that if $\kappa=(k_1,k_2)$ with $k_1+1\neq k_2$, i.e. when the Hodge-Tate weights are not equal,  one can choose $\n{V}'$ small enough such that $b\neq 0$.   
\end{enumerate}
\end{remark}

We finish  with the compatibility of the oveconvergent Eichler-Shimura maps (\ref{eqESsequenceIntro}) with the Poincar\'e and Serre pairings.  One can define  a  Poincar\'e   pairing  between the overconvergent pro\'etale  cohomologies  
\begin{eqnarray}
\label{eqPairing1Intro}
\langle -,- \rangle_P & : &   H^1_{\proet,c}(Y_{\bb{C}_p},  D^{\delta}_{\chi,\et}(1)) \times H^1_{\proet}(Y_{\bb{C}_p},  A^{\delta}_{\chi,\et} ) \rightarrow  \s{O}(\n{V}')
\end{eqnarray}
where the left hand side term is the pro\'etale cohomology with compact support.   On the other hand,  one also has Serre pairings between overconvergent coherent cohomologies   
\begin{eqnarray}
\label{eqPairing2Intro}
\langle -, -\rangle_S&:& H^1_{w,c}(X, \omega_E^{-\chi}(-D))_{\epsilon}\times  H^{0}_{w}(X,\omega_E^{\chi+\alpha})_{\epsilon} \rightarrow  \s{O}(\n{V}')\\
\langle -, -\rangle_S&:& H^1_{w,c}(X, \omega_E^{w_0(\chi)})_{\epsilon}\times  H^{0}_{w}(X,\omega_E^{-w_0(\chi)+\alpha}(-D))_{\epsilon} \rightarrow  \s{O}(\n{V}'). \nonumber
\end{eqnarray}
 We have the following 
\begin{theo}[Theorem \ref{TheoPairingOVES}]
\label{TheoOVESDualIntro}
\begin{enumerate} 

\item The Poincar\'e and Serre pairings \eqref{eqPairing1Intro} and \eqref{eqPairing2Intro} are compatible with the $U_p$-operators and  the overconvergent  Eichler-Shimura maps.

 \item Let $\n{W}_T$ be the weight space of $T=\bbf{T}(\bb{Z}_p)$,  let $\n{V}\subset \n{W}_T$ be an open affinoid   and  $\chi$  the universal character of $\n{V}$.  Let $\kappa=(k_1,k_2)\in \n{V}$ be a dominant weight and   fix  $h<k_1-k_2+1$.   There exists an open affinoid  $\n{V}'\subset \n{V}$ containing $\kappa$ such that  the $(\leq h)$-part of the pairings \eqref{eqPairing1Intro} and \eqref{eqPairing2Intro}  are  perfect pairings of finite free $\bb{C}_p\widehat{\otimes} \s{O}(\n{V}')$-modules  compatible with the  Eichler-Shimura decomposition.   
\end{enumerate}
\end{theo}

The outline of the document is the following.  In Section \ref{SectionOvFlag} we develop the overconvergent theory over the flag variety.  We define the affinoid subspaces $U_w(\epsilon)\Iw_n$ and the sheaves $\s{L}(\chi)$.  We construct the $\delta$-analytic principal series $A^{\delta}_{\chi}$ and the maps (\ref{eqovBGGmapsIntro}).  We recall some facts of the BGG theory  for irreducible representations of $\GL_2$,  in particular we define the dual BGG complex $\BGG(\kappa)$. \\  

Then in Section \ref{SectionOvtheoryOverModCurves},  we  translate all the previous constructions from $\bb{P}^1_{\bb{Q}_p}$ to the modular curves via $\pi_{\HT}$. We define the strict neighbourhoods of the $w$-ordinary locus, the  overconvergent modular sheaves and the overconvergent Hodge-Tate maps.  We give the good normalizations of the Hecke operators and show that the $\HT$-maps are compatible with the normalized $U_p$-correspondence. \\

Finally, in Section \ref{SectionOverconvergentES}, we show how to obtain the classical Eichler-Shimura decomposition from the dual BGG complex, proving Theorems \ref{TheoBGGIntro} and \ref{TheoIntro1}, in the process we also prove the theorem for the cohomology with compact support.     Next, we construct the overconvergent Eichler-Shimura maps and obtain Theorem \ref{TheoOVESIntro}. Finally, we show the  compatibility of Poincar\'e and Serre  duality for the overconvergent Eichler-Shimura maps  obtaining Theorem \ref{TheoOVESDualIntro}.

\begin{acknowledgements} None of this work could ever be possible without the support of my advisor Vincent Pilloni,  I want to express my deep gratitude for the many hours he dedicated  explaining to me the theory of  overconvergent modular forms.   I want to thank professors  Adrian Iovita and Fabrizio Andreatta for the very fruitful exchanges during the spring of 2020,   and the opportunity to give a talk on this subject in the Workshop of higher Coleman theory in December of the same year.     I want to thank George Boxer,  Damien Junger,  Joaqu\'in Rodrigues Jacinto and  Ju-Feng Wu for  corrections and comments in an earlier version of this document.  Finally, I want to thank the anonymous referees for the many suggestions and corrections that improved the presentation of this paper.    This work has been written while the author was a Ph. D. student at the ENS de Lyon. 
\end{acknowledgements}

\begin{notation}

Throughout this document we fix a prime number $p$, we fix an algebraic closure of $\bb{Q}_p$ and denote by $\bb{C}_p$ its $p$-adic completion.    We will  work with adic spaces over $\Spa(\bb{Q}_p,\bb{Z}_p)$ which are either locally  topologically of finite type over a non-archimedean extension $K$ of $\bb{Q}_p$, or  perfectoid spaces.  

Let $X$ be a log adic space over $\bb{Q}_p$,  we will work with the pro\'etale and pro-Kummer-\'etale site of $X$ as introduced in \cite{ScholzeHodgeTheory2013} and \cite{DiaoLogarithmicHilbert2018, DiaoLogarithmic2019}. We denote by $X_{?}$,  with $?\in \{\an, \et,\ket,\proet,\proket\}$,  the analytic, \'etale, Kummer-\'etale, pro\'etale and pro-Kummer-\'etale sites of $X$ respectively. An space without underlying log structure will be  endowed with the trivial one.  Fiber products are always fiber products of fs log adic spaces unless otherwise specified, cf.  \cite[Proposition  2.3.27]{DiaoLogarithmic2019}. 

 Finally, we will denote by $\s{O}_{X}^{(+)}$ the uncompleted structural sheaves over $X_{\proket}$, and by $\widehat{\s{O}}_X^{(+)}$ their $p$-adic completion, we will omit the index $X$ if the space is clear in the context.  Let $V$ be a topological $\bb{Z}_p$-module, by an abuse of notation we  also denote by $V$ the pro-Kummer-\'etale sheaf over $X_{\proket}$ whose points at an object $U$  are equal to the space of continuous functions $\mathrm{Cont}(|U|, V)$, where $|U|$ is the underlying topological space attached to $U$. 

\end{notation}


\section{Overconvergent sheaves over the flag variety}
\label{SectionOvFlag}

Let $\GL_2$ be the algebraic group of  invertible $2\times 2$ matrices.  Let $\bbf{B}$ and $\bbf{T}$ be the upper triangular Borel and the diagonal torus of $\GL_2$, let $\bbf{N}\subset \bbf{B}$ be the unipotent radical consisting of upper triangular unipotent matrices. We also let $\overline{\bbf{B}}$ and $\overline{\bbf{N}}$ be the lower triangular Borel and its unipotent radical respectively.  Let $W=\{1,w_0\}$ be the Weyl group of $\GL_2$. Given $n\geq 1$ we let $\Iw_n\subset \GL_2(\bb{Z}_p)$ denote the Iwahori subgroup of level $p^n$, i.e, the subgroup of invertible matrices $\left( \begin{matrix}
a & b \\ c & d
\end{matrix} \right)$ such that $c \equiv 0 \mod p^n$. We let  $\FL=  \bbf{B}\backslash \GL_2$  be  the flag variety  and  $\Fl$ its  analytification to an  adic space over $\Spa(\bb{Q}_p, \bb{Z}_p)$.    From now on, we see all the previous schemes as living over $\bb{Q}_p$.

The goal of this section is to introduce a family of $\Iw_n$-stable  overconvergent neighbourhoods of the $\Iw_n$-orbit of $w\in W$ in $\Fl$. Then, we introduce some $\Iw_n$-equivariant line bundles   which play the role of the overconvergent modular sheaves over $\Fl$. Finally, we construct some weight vector morphisms, which will be translated in the overconvergent Hodge-Tate maps over the modular curve.

For future reference we make the following convention 

\begin{convention}
\label{DefiHsubgroup}

Let $\bbf{H}$ be an algebraic group scheme over $\Spec \bb{Z}_p$.  We denote by $\n{H}^0$ the rigid generic fiber of the $p$-adic completion of $\bbf{H}$, and by $\n{H}$ the analytification of the schematic generic fiber of $\bbf{H}$, see \cite{HuberEtaleCohomology}. Given $\delta>0$  a rational number, we let $\n{H}(\delta)\subset \n{H}^0\subset \n{H}$ denote the open subgroup whose $(R,R^+)$-points are 
\[
\n{H}(\delta)(R,R^+)= \ker ( \bbf{H}(R^+)\to \bbf{H}(R^+/p^{\delta}R^+)).
\]
We call $\n{H}(\delta)$ the $\delta$-neighbourhood of the identity in $\n{H}$.   
\end{convention}

It will be useful to introduce some particular open subgroups of $\GLa_2$.

\begin{definition}
\label{DefinitionOpenGl2}
 Let $\epsilon \geq \delta$ be  positive rational numbers.  
 \begin{enumerate}
 
\item  We denote 
\[
\GLa_2(\epsilon,\delta):= \n{N}(\delta)\times \n{T}(\delta) \times \overline{\n{N}}(\epsilon)\subset \GLa_2.
\]
 
\item Suppose that $\delta \geq n$.   The $\delta$-neighbourhood of $\Iw_n$ in $\GLa_2$ is the open subgroup 
\[
\IW_n({\delta}):= \Iw_n \GLa_2(\delta)= \GLa_2(\delta)\Iw_n. 
\] 
We refer to  $\IW_n(\delta)$ as an affinoid  Iwahori subgroup of $\GLa_2$.   

\item We let $T$, $B$, $N$, etc.  denote the $\bb{Z}_p$-points of $\bbf{T}$,  $\bbf{B}$,  $\bbf{N}$, etc.  Let $n\geq 1$,  we define the following subgroups of $T$, $N$ and $\overline{N}$
\[
T_n= \left( \begin{matrix} 1+p^n \bb{Z}_p & 0 \\ 0 & 1+p^n \bb{Z}_p  \end{matrix} \right), \;\;  N_n=\left( \begin{matrix} 1 & p^n \bb{Z}_p \\ 0 & 1  \end{matrix} \right), \;\;  \overline{N}_n = \left( \begin{matrix} 1 & 0 \\ p^n \bb{Z}_p & 1   \end{matrix} \right). 
\]

\end{enumerate}

\end{definition}

\subsection{$\GL_2$-equivariant sheaves over the flag variety}

In the following paragraph we fix some notations regarding the representation theory of $\GL_2$. Let $X^*(\bbf{T})$ be the character group of the diagonal torus, we identify $X^*(\bbf{T}) \cong \bb{Z}^2$ via the presentation 
\[
\bbf{T} = \left( \begin{matrix}
 \bb{G}_m  & 0 \\ 0 & \bb{G}_m
\end{matrix} \right).
\]
A weight $\kappa\in X^*(\bbf{T})$, written as $\kappa= (k_1,k_2)$, is said dominant if $k_1\geq k_2$. We denote by $X^*(\bbf{T})^+$ the cone of dominant weights.  Given $\kappa\in X^*(\bbf{T})^+$, we let $V_{\kappa}$ denote the irreducible algebraic representation of $\GL_2$ of highest weight $\kappa$; let $\St$ and $\det$ be the standard and the determinant representations of $\GL_2$, explicitly one has  that $V_{\kappa} \cong \Sym^{k_1-k_2} \St \otimes \det^{k_2}$.   

Since $\FL= \bbf{B}\backslash \GL_2$, there is an equivalence between the category  of algebraic representations of $\bbf{B}$, and the category of $\GL_2$-equivariant vector bundles over $\FL$.  Explicitly, given $W$ a representation of $\bbf{B}$, one forms the $\GL_2$-equivariant vector bundle
\[
\n{W} := \GL_2 \times^{\bbf{B}} W := \bbf{B}\backslash (\GL_2 \times W),
\]
where in the right hand side term the group $\bbf{B}$ acts diagonally.

\begin{definition}
\label{DefiEqLineBundleFlag}
Let $\kappa \in X^*(\bbf{T})$, we define $\s{L}(\kappa)$ to be  the $\GL_2$-equivariant line bundle given by $\GL_w \times^{\bbf{B}} w_0(\kappa)$.  
\end{definition}

\begin{remark}
\label{RemarkOtherDescriptionLineBundle}
\begin{enumerate}
\item The line bundle $\n{L}(\kappa)$ can be described in the following way:  let $\widetilde{\FL}= \bbf{N}\backslash \GL_2$ be the natural $\bbf{T}$-torsor over $\FL$ and $\pi:  \widetilde{\FL} \to  \FL$ the projection map, then $\pi_* \s{O}_{\widetilde{\FL}}$ is endowed with a left regular action of $\bbf{T}$. One can construct the line bundle $\n{L}(\kappa)$ as the following isotypic component: 
\begin{align*}
\n{L}(\kappa) & = \pi_* \s{O}_{\widetilde{\FL}}[-w_0(\kappa)] \\ 
& = \{ f \in \pi_* \s{O}_{\widetilde{\FL}} :  f(t x) = w_0(\kappa) (t) f(x) \mbox{ for } t\in\bbf{T} \}. 
\end{align*}
The previous description shows that 
\[
\widetilde{\FL}= \underline{\Isom}(\s{O}_{\FL}, \s{L}(0,-1))\times \underline{\Isom}( \s{O}_{\FL}, \s{L}(-1,0)). 
\]
\item  The convention on the weight is made such that, if $\kappa$ is dominant,  $\Gamma(\FL, \n{L}(\kappa))$ is isomorphic to $V_{\kappa}$ as a $\GL_2$-representation.

\end{enumerate}

\end{remark}

\subsection{Overconvergent line bundles over the flag variety}
\label{SubsectionOpenAffinoidFlag}

In order to define the overconvergent line bundles we first need to introduce some affinoid neighbourhoods of $w\in W$.  

\begin{definition}
Let $\epsilon >0$ be a rational number, $w \in W$ and $w\GLa_2(\epsilon)$ the $\epsilon$-neighbourhood of $w$ in $\GLa_2$.  We denote by $U_w(\epsilon)$ its image in $\Fl$. 
\end{definition}

\begin{lem}
\label{LemIwahoriDecomposition}
\begin{enumerate}
\item The collection $\{U_{w}(\epsilon)\}_{\epsilon > 0}$  is a basis of open affinoid neighbourhoods of $w\in \Fl$.  Moreover, we have a natural isomorphism 
\[
\overline{\n{N}}(\epsilon) w  \xrightarrow{\sim} U_{w}(\epsilon).
\]

\item The Iwahori subgroups admit Iwahori decompositions  
\[
\IW_n(\epsilon)= (\overline{N}_n\overline{\n{N}}(\epsilon))\times (T\n{T}(\epsilon))\times (N \n{N}(\epsilon))
\]
\item  Let $\epsilon\geq \delta \geq n \geq 1$.  We have   decompositions 
\begin{gather*}
\GLa_2(\epsilon,\delta) \Iw_n = (N\n{N}(\delta))\times(T \n{T}(\delta))\times(\overline{N}_n \overline{\n{N}}(\epsilon) ) \\
\GLa_2(\epsilon,\delta) w_0 \Iw_n= (N_n \n{N}(\delta))\times (T\n{T}(\delta)) \times( \overline{N} \overline{\n{N}}(\epsilon) )w_0. 
\end{gather*}
\end{enumerate}

\proof
The collection $\{U_w(\epsilon)\}_{\epsilon >0}$ is a basis of neighbourhoods of $w$  since  $\Fl$ is a locally spectral space and  $\bigcap_{\epsilon>0} U_w(\epsilon)=\{w\}$. The isomorphism  $U_w(\epsilon)\cong  \overline{\n{N}}(\epsilon)w$ is obvious.  Next  we prove (2), part (3) is done in a  similar way.  It suffices to show the equality at $(R,R^+)$-points, with $(R,R^+)$ a uniform affinoid $\bb{Q}_p$-algebra. By definition we have 
\[
\IW_n(\epsilon)=  \left( \begin{matrix}
\bb{Z}_p^\times(1+p^\epsilon \bb{D}^1_{\bb{Q}_p})&  \bb{Z}_p+p^\epsilon \bb{D}^1_{\bb{Q}_p}  \\ p^n\bb{Z}_p+p^\epsilon \bb{D}^1_{\bb{Q}_p} & \bb{Z}_p^\times(1+p^\epsilon \bb{D}^1_{\bb{Q}_p})
\end{matrix} \right),
\]
where $\bb{D}^1_{\bb{Q}_p}= \Spa(\bb{Q}_p\langle T \rangle, \bb{Z}_p \langle T \rangle )$ is the closed affinoid disc. 
Then 
\[
\IW_n(\epsilon)(R,R^+)= \left( \begin{matrix}
\bb{Z}_p^\times(1+p^\epsilon R^+)&  \bb{Z}_p+p^\epsilon R^+  \\ p^n\bb{Z}_p+p^\epsilon R^+& \bb{Z}_p^\times(1+p^\epsilon R^+)
\end{matrix} \right) 
\]
Let $g\in \IW_n(\epsilon)(R,R^+)$, writing 
\[
g= \left( \begin{matrix}
1 & x_2 \\ 0 & 1
\end{matrix} \right)  \left( \begin{matrix}
x_1 & 0 \\ 0 & x_4 
\end{matrix} \right) \left( \begin{matrix}
1 & 0 \\ x_3 & 1
\end{matrix} \right)
\]
and solving the equations one finds $x_3\in p^n \bb{Z}_p+ p^{\epsilon}R^{+}$, $x_1$ and $x_4\in \bb{Z}_p^\times (1+p^{\epsilon }R^{+})$,  and $x_2\in \bb{Z}_p+ p^\epsilon R^+$ which gives (2).  \endproof
\end{lem}

\begin{remark}
Let us identify $\Fl \cong \bb{P}^1_{\bb{Q}_p}$ by taking $[0:1]\in \bb{P}^1_{\bb{Q}_p}$ as marked point, and where $\GL_2$ acts by 
\[
[x:y] \left( \begin{matrix} a & b \\ c & d \end{matrix} \right) = [ax+cy:  bx+dy]. 
\]
Let $T= \frac{x}{y}$ be the canonical coordinate and $\epsilon = p^{-n}$.  In the notation of \cite[\S 4.2]{AIOverconvergentES2} we have 
\begin{align*}
U_1(\epsilon) &= U_{0,0}^{(n)} = \{[x:y]\in \bb{P}^1_{\bb{Q}_p} :  |T/ p^n|\leq 1\}  \\ 
U_{w_0}(\epsilon) & = U_{\infty}^{(n)} = \{[x:y]\in \bb{P}^1_{\bb{Q}_p} : |1/(p^nT)|\leq 1\}. 
\end{align*}
\end{remark}

Lemma \ref{Lemvarpidynamics} describes the dynamics of the element $\varpi= \diag (1,p)$ over $\Fl$.  This action has only two fixed points represented by the elements of $W$, and expands or shrinks neighbourhoods of $1$ and $w_0$ respectively.   

\begin{lem}
\label{Lemvarpidynamics}
Let $\varpi= \diag(1,p)$.  The following holds
\begin{enumerate}

\item $U_1(\epsilon)\varpi= U_1(\epsilon-1)$ and $U_{w_0}(\epsilon)\varpi= U_{w_0}(\epsilon+1)$.

\item Let $\epsilon \geq n \geq 1$,  then  $U_1(\epsilon)\Iw_n \varpi = U_1(\epsilon-1)\Iw_{n-1}$ and $U_{w_0}(\epsilon)\Iw_n \varpi =U_{w_0}(\epsilon+1) N_1$. 

\end{enumerate}
\proof
It follows from Lemma  \ref{LemIwahoriDecomposition} and the computation 
\[
\left( \begin{matrix} 
1 & 0 \\ 0 & p^{-1} 
\end{matrix}\right) \left( \begin{matrix} 
  a & b\\  c & d
\end{matrix}\right) \left( \begin{matrix} 
1  & 0  \\ 0  &  p
\end{matrix}\right)= \left( \begin{matrix} 
 a & pb \\ p^{-1}c  & d
\end{matrix}\right).
\]
\endproof
\end{lem}

Let $\Gamma$ be a finite $\bb{Z}_p$-module. Abstractly, $\Gamma$ is isomorphic to $\bb{Z}_p^s\bigoplus \Gamma_{tor}$ where $s\in \bb{N}$ and $\Gamma_{tor}\subset \Gamma$ is the torsion subgroup, we call such an isomorphism a chart of $\Gamma$.  Let  $\n{V}=\Spa(R,R^+)$ be an affinoid adic space with $R$ an uniform Tate $\bb{Q}_p$-algebra, and $\chi: \Gamma \to R^{+,\times}$ a continuous character. 

\begin{lem}[{\cite[Lemma 3.4.6]{UrbanEigenvar}}]
\label{PropositionCharacterLocallyAnalytic}
Let $\psi: \Gamma \cong \bb{Z}_p^s\times \Gamma_{tor}$ be a chart.  There exists  $\delta>0$  such that $\chi$ extends to a character 
\begin{equation}
\label{ExtensionAnalyticCharacter}
\chi:  (\bb{Z}_p+p^{\delta} \bb{D}^1_{\bb{Q}_p})^s \times \Gamma_{tor}  \times \n{V} \rightarrow \bb{G}_m.
\end{equation}
We say  that $\chi$ is a $\delta$-analytic character of $\Gamma$ with respect to the chart $\psi$. 
\end{lem}

\begin{remark}
In the following we will take $\Gamma= T$, and we say that $\chi$ is $\delta$-analytic if it extends to a character of $T\n{T}(\delta)$.  Let $\f{W}_T = \Spf \bb{Z}_p[[T]]$ be the weight space of $T$ and $\n{W}_T$ its rigid generic fiber, in practice we will take $\n{V}= \Spa(R,R^+)\subset \n{W}_T$ an affinoid subspace and $\chi= \chi_{univ}$ the universal character over $\n{V}$.  
\end{remark}

\begin{definition}
\label{DefinitionReductionTorsorFlag}
Let $\epsilon \geq \delta \geq n \geq 1$,  $w\in W$ and $\widetilde{\Fl}= \n{N}\backslash \GLa_2$ the natural $\n{T}$-torsor over $\Fl$. 

\begin{enumerate}

\item We define the following neighbourhood of $w$ in $\widetilde{\Fl}$
\[
\widetilde{U}_{w}(\epsilon,\delta)= \n{N}(\delta) \backslash  \GLa_2(\epsilon,\delta) w.
\]
Let $\pr: \widetilde{U}_w(\epsilon,\delta)\Iw_n \to U_w(\epsilon)\Iw_n$ denote the natural projection, Lemma \ref{LemIwahoriDecomposition} (3) implies that $\pr$ is a trivial $T\n{T(\delta)}$-torsor. 

\item Let $\chi: T \to R^\times$ be a $\delta$-analytic character. We define the $\s{O}_{\Fl}\widehat{\otimes} R$-line bundle 
\[
\n{L}(\chi)= \pr_* \s{O}_{\widetilde{U}_w(\epsilon, \delta) }\widehat{\otimes} R [-w_0(\chi)].
\]
In other words, $\n{L}(\chi)$ is the line bundle whose sections over $U\subset U_w(\epsilon)\Iw_n$ are 
\[
\n{L}(\chi)(U)= \{f \in  \s{O}_{\Fl}(\pr^{-1}(U))\widehat{\otimes }R : f(tx) = w_0(\chi)(t)f(x) \mbox{ for } t\in T\n{T}(\delta)\}. 
\]
\end{enumerate}

\end{definition}

\begin{remark}
\label{RemarkUpActionFlag}
Let $\varpi= \diag(1,p)$,  $c= \diag(p,p)$ and $\Lambda= \langle \varpi, c\rangle\subset \bbf{T}(\bb{Q}_p)$. The natural map $\widetilde{\Fl}\to \Lambda\backslash \widetilde{\Fl}$ identifies $\widetilde{U}_w(\epsilon,\delta)$ with an open affinoid of the quotient, namely, the $\Lambda$ orbit of $\widetilde{U}_w(\epsilon,\delta)$ in $\widetilde{\Fl}$ is the disjoint union 
\[
\bigsqcup_{\gamma \in \Lambda} \widetilde{U}_{w}(\epsilon, \delta) \gamma. 
\]
\end{remark}

By construction, the sheaves $\n{L}(\chi)$ are independent of $\delta$ and $\epsilon$.  They fit in  a $U_p$-correspondence as follows: let  $\varpi= \diag(1,p)$ and  consider the double coset $\Iw_n \varpi \Iw_n$, one has that 
\[
\Iw_n \varpi \Iw_n=\bigsqcup_{a=0}^{p-1} \Iw_n \left( \begin{matrix}  1  & -a \\ 0 & p \end{matrix} \right),
\]
let us denote $U_{p,a}=\left( \begin{matrix}  1  & -a \\ 0 & p \end{matrix} \right)$.

\begin{definition}
\label{DefinitionUpCorrespondenceFlag}
We define the $U_p$-correspondence of $\widetilde{\Fl}$ (resp. the normalized $U_p$-correspondence of $\Lambda \backslash \widetilde{\Fl}$) to be the diagram
\begin{equation}
\label{eqCorrespondenceFlag}
\begin{tikzcd}
 & \bigsqcup_{a=0}^{p-1} \widetilde{\Fl} \ar[rd, "p_2"] \ar[ld, "p_1"'] &   \\ 
  \widetilde{\Fl} & & \widetilde{\Fl},
\end{tikzcd}
\end{equation}
(resp. for $\Lambda \backslash \widetilde{\Fl}$),  where $\widetilde{\Fl}_a= \widetilde{\Fl}$,  $p_1|_{ \widetilde{\Fl}_a}= \id_{ \Fl}$ and $p_2 |_{\Fl_a} = R_{U_{p,a}^{-1}}$ is the right multiplication by $U_{p,a}^{-1}$.  

\end{definition}

\begin{remark}
\label{RemarkCorrQuotientFlag}
The correspondence \eqref{eqCorrespondenceFlag}  is equivariant for the natural action of $\Iw_n$, namely, the group acts by right multiplication on the bottom spaces, and it acts on the disjoint union $\bigsqcup_{a=0}^{p-1}\widetilde{\Fl}$ as follows: Let $\gamma\in \Iw_n$ and  $a,a'\in\{0,\ldots, p-1\}$ such that 
\[
U_{p,a} \gamma \in \Iw_n U_{p,a'}.
\]
Then, given $x\in \Fl_a$, we define $x\cdot \gamma$ to be $x\gamma\in \Fl_{a'}$. This action satisfy the following properties: 
\begin{enumerate}
\item The map $p_1$ is $\Iw_n$-equivariant.
\item The map $p_2$ preserves $\Iw_n$-orbits, i.e. the composition $\bigsqcup_{a=0}^{p-1}\widetilde{\Fl}\xrightarrow{p_2}\widetilde{\Fl}/\Iw_n$ onto the quotient stack\footnote{By considering the quotient as a $v$-stack, see \cite{ScholzeEtaleDiamonds}.} factors through $(\bigsqcup_{a=0}^{p-1} \widetilde{\Fl})/\Iw_n$. 
\end{enumerate} 
The previous shows that we have a correspondence of stacks 
\begin{equation}
\label{eqCorrespondenceStack}
\begin{tikzcd}
 & (\bigsqcup_{a=0}^{p-1} \widetilde{\Fl})/\Iw_n \ar[rd, "p_2"] \ar[ld, "p_1"'] &   \\ 
  \widetilde{\Fl}/\Iw_n & & \widetilde{\Fl}/\Iw_n,
\end{tikzcd}
\end{equation}
in \S \ref{SectionOvtheoryOverModCurves} we will relate this diagram to the $U_p$-correspondence of modular curves. 
\end{remark}

The following lemma describes the dynamics of the correspondence on  neighbourhoods of $1$ and $w_0$.
\begin{lem}
\label{LemCorrespondenceFlag}
The following hold. 
\begin{enumerate}
\item $p_2 (p_1^{-1}( \widetilde{U}_{w_0}(\epsilon,\delta)\Iw_n)) \supset   \widetilde{U}_{w_0}(\epsilon-1,\delta)\Iw_n$.

\item $p_2 (p_1^{-1}(\widetilde{U}_{1}(\epsilon,\delta)\Iw_n )) \subset \widetilde{U}_1(\epsilon+1, \delta)\Iw_n$.

\item $p_1 (p_2^{-1}(\widetilde{U}_{w_0}(\epsilon, \delta) \Iw_n ))\subset \widetilde{U}_{w_0}(\epsilon+1,\delta)\Iw_n$.

\item $p_1 (p_2^{-1}(\widetilde{U}_{1}(\epsilon, \delta) \Iw_n)) \supset \widetilde{U}_{1}(\epsilon-1, \delta)\Iw_n$.

\end{enumerate}
\end{lem}
\begin{proof}
This follows from the definition of the correspondence and Lemma \ref{Lemvarpidynamics}.
\end{proof}

\begin{definition}
\label{DefUpSheavesoverFlag}
\begin{enumerate}

\item Let $\kappa\in X^*(\bbf{T})$,  we define the $U_{p,\kappa}$ and $U_{p,\kappa}^t$-correspondences of $\s{L}(\kappa)$
\begin{equation*}
\label{eqCorrespondenceLineClassical}
U_{p,\kappa}: p_{2}^* \s{L}(\kappa)  \to p_1^* \s{L}(\chi) \mbox{ and } U_{p,\kappa}^t:  p_{1}^* \s{L}(\kappa) \to  p_2^*\s{L}(\chi) 
\end{equation*}
to be the maps constructed by taking $(-w_0(\kappa))$-isotypic components of the structural sheaves of the diagram \eqref{eqCorrespondenceFlag}. 

\item Let $\chi$ be a $\delta$-analytic character of $T$, we define the normalized  $U_p$ and $U_p^t$-correspondences of $\s{L}(\chi)$ 
\begin{equation*}
\label{eqCorrespondenceLineBundle}
U_{p}: p_{2}^* \s{L}(\chi)  \to p_1^* \s{L}(\chi) \mbox{ and } U_{p}^t:  p_{1}^* \s{L}(\chi) \to  p_2^*\s{L}(\chi) 
\end{equation*}
to be the maps constructed by taking  $(-w_0(\chi))$-isotypic components of the structural sheaves of the diagrams 
\[
\begin{tikzcd}[column sep = 0.1 cm]
 & p_2^{-1}(\widetilde{U}_{w_0}(\epsilon+1,\delta)\Iw_n)  \ar[rd, "p_2"] \ar[ld, "p_1"'] &       \\ 
 \widetilde{U}_{w_0}(\epsilon+1, \delta) \Iw_n&  &  \widetilde{U}_{w_0}(\epsilon, \delta)\Iw_n       \\
   & p_1^{-1}(\widetilde{U}_{1}(\epsilon-1,\delta)\Iw_n)  \ar[rd,"p_2"] \ar[ld,  "p_1"']&  \\ 
    \widetilde{U}_{1}(\epsilon-1,\delta)\Iw_n  & &  \widetilde{U}_{1}(\epsilon, \delta)\Iw_n
\end{tikzcd}
\]
obtained from the normalized diagram \eqref{eqCorrespondenceFlag} and Lemma \ref{LemCorrespondenceFlag}.  
\end{enumerate}
\end{definition}

\begin{remark}
\label{RemakNormalizationUp}
Let $\kappa$ be a classical weight, the relation between the classical and the normalized $U_p$-operators for $\s{L}(\kappa)$ is given by the formulas
\begin{eqnarray*}
U_p= \frac{1}{w_0(\kappa)(\varpi^{-1})}U_{p,\kappa} \mbox{ and } U_p^{t}=  \frac{1}{w_0(\kappa)(\varpi)} U_{p,\kappa}^t & \mbox{ over } U_1(\epsilon)\Iw_n, \\
U_p= \frac{1}{\kappa (\varpi^{-1})} U_{p,\kappa} \mbox{ and } U_p^{t}=  \frac{1}{\kappa(\varpi)} U_{p,\kappa}^t & \mbox{ over } U_{w_0}(\epsilon)\Iw_n.
\end{eqnarray*}
\end{remark}

\subsection{Analytic principal series and distributions}

In the next paragraph we introduce some non-archimedean  spaces interpolating the finite dimensional representations of $\GL_2$.  Let $\n{V}= \Spa(R,R^+)$ be a uniform affinoid adic space over $\bb{Q}_p$ and $\chi : T \to R^{+,\times}$ a $\delta$-analytic character.  

\begin{definition}
\label{DefPrincipalSeriesDist}
Let $\delta \geq n$. 

\begin{enumerate}

\item We define the $\delta$-analytic principal series of weight $\chi$ to be the $\Iw_n$-module
\[
A_{\chi}^{\delta}= \Gamma(U_{w_0}(\delta)\Iw_n, \s{L}(\chi)),
\]
seen as a Banach space over $R$.  Equivalently, we have that 
\[
A_{\chi}^{\delta}=\{f: w_0 \IW_n(\epsilon,\delta)\to \bb{A}^1_{R} | f(bw_0 x)= w_0(\chi)(b) f(w_0x) \mbox{ for } b\in \n{B}\cap w_0 \IW_n(\epsilon,\delta) w_0^{-1} \}.
\]

\item We define the $\delta$-analytic distributions of weight $\chi$ to be the continuous weak dual  
\[
D_{\chi}^{\delta}= \Hom^0_{R}(A^{\delta}_{\chi}, R).
\]

\end{enumerate}
\end{definition} 

\begin{remark}
\label{RemarkBasisA}
It will come in handy to define lattices in the principal series and distributions, namely, let $\s{L}^+(\chi)= f_* \s{O}^+_{\widetilde{U}_{w}(\epsilon,\delta)\Iw_n}[-w_0(\chi)]$,  $A_{\chi}^{\delta,+}= \Gamma(U_{w_0}(\epsilon)\Iw_n,  \s{L}^+(\chi))$ and $D_{\chi}^{\delta,+}= \Hom_{R^+}(A^{\delta,+}_{\chi},R^+)$.  The space $N\n{N}(\delta)$ is a disjoint union of closed discs, in particular $\s{O}^+(N\n{N}(\delta))$ is an orthonormalizable $\bb{Z}_p$-algebra.  Let $\{e_i\}_{i\in I }$ be an orthonormalizable  basis of $\s{O}^+(N\n{N}(\delta))$,  using the Iwahori decomposition one has isomorphisms of $R^+$-modules 
\[
A^{\delta,+}_{\chi} \cong  \underset{i\in I}{\widehat{\bigoplus}} R^+ e_i \mbox{ and }
D^{\delta,+}_{\chi} \cong \prod_{i\in I} R^+ e_{i}^{\vee}. 
\]
\end{remark}

\begin{remark}
\label{RemarkAnalyticFunct}
It is easy to compare the $\delta$-analytic principal series and distributions defined above with those used in \cite{AISOvShimura2015}.  Let $\chi: T\rightarrow R^{+,\times}$ be a  $\delta$-analytic character written  as $\chi=(\chi_1,\chi_2)$.   Consider the set $\bb{Z}_p^\times\times \bb{Z}_p$ endowed with the right  multiplication by $\Iw_n$ and the left multiplication by $\bb{Z}_p^\times$.  We let $A^{\delta,+}_{\chi_1-\chi_2}$ be the space of functions $f: \bb{Z}_p^\times \times   \bb{Z}_p\rightarrow R^{+,\times}$ satisfying the following conditions 
\begin{itemize}
\item[i.] $f|_{1\times \bb{Z}_p}$ extends to an analytic function of $\bb{Z}_p+ p^\delta \bb{D}^1_{\bb{Q}_p}$, 

\item[ii.] $f(tx)=(\chi_1-\chi_2)(t)f(x)$ for $t\in \bb{Z}_p^\times$ and $x\in \bb{Z}_p^\times \times \bb{Z}_p$.
\end{itemize}
Note that $\bb{Z}_p^\times \times \bb{Z}_p$ endowed with the action of $\Iw_n$ and $\bb{Z}_p^\times$ is isomorphic to the quotient 
\[
\bb{Z}_p^\times \times \bb{Z}_p=  \left( \begin{matrix}
1 & 0  \\ p^n\bb{Z}_p & \bb{Z}_p^\times
\end{matrix} \right) \backslash  \left( \begin{matrix}
\bb{Z}_p^\times & \bb{Z}_p  \\ p^n\bb{Z}_p & \bb{Z}_p^\times
\end{matrix} \right)= \left( \begin{matrix}
1 & 0  \\ p^n\bb{Z}_p & \bb{Z}_p^\times
\end{matrix} \right) \backslash  \Iw_n.
\]
Thus, we have an isomorphism 
\[
A^{\delta,+}_{\chi}= A^{\delta,+}_{\chi_1-\chi_2} \otimes (\det)^{\chi_2}. 
\]

\end{remark} 

\begin{remark}
\label{RemarkBasisTopEspaces}
Let $\delta'>\delta \geq n$,  the inclusion $U_{w_0}(\delta')\Iw_n \subset U_{w_0}(\delta)\Iw_n$ induces maps $A_{\chi}^{\delta,+}\to A_{\chi}^{\delta',+}$ and $D_{\chi}^{\delta',+}\to D_{\chi}^{\delta,+}$. Furthermore, let $\widehat{A}^{\delta,+}_{\chi}$ be the completion of $A^{\delta,+}_{\chi}$ in $A^{\delta',+}_{\chi}$. Since the inclusion of affinoids above is strict, one has that 
\[
\widehat{A}_{\chi}^{\delta,+}= \prod_{i\in I} R^+ e_i,
\] 
endowed with the product topology.  Let $D_{\chi}^{\delta,b,+}$ be the continuous $R^+$-dual of $\widehat{A}_{\chi}^{\delta,+}$ endowed with the $p$-adic topology,  then the arrow $D_{\chi}^{\delta',+} \to D_{\chi}^{\delta,+}$  factors through $D_{\chi}^{\delta,b,+}$.  Moreover, we have that 
\[
D^{\delta,b,+}_{\chi} \cong   \underset{i\in I}{\widehat{\bigoplus}} R^+ e_i^{\vee}.  
\]
The  previous discussion shows that the directed (resp. inverse) systems $\{A_{\chi}^{\delta,+}\}_{\delta\geq n}$ and  $\{\widehat{A}_{\chi}^{\delta,+}\}_{\delta\geq n}$ (resp. $\{D_{\chi}^{\delta,+}\}_{\delta\geq n}$ and $\{D_{\chi}^{\delta,b,+}\}_{\delta\geq n}$) are isomorphic as systems of topological $R^+$-modules. 
\end{remark}

For future reference we will prove a devisage of  $A_{\chi}^{\delta}$ and $D^{\delta}_{\chi}$ in terms of finite $\Iw_n$-modules, we need a lemma:  
 
\begin{lem}
\label{LemmaUniformContGroupAction}
Let $(F,\n{O}_F)$ be a non archimedean field.  Let $\n{H}=\Spa(A,A^+)$ be an affinoid adic analytic group over $F$,  and  $Z=\Spa(R,R^+)$  an affinoid adic space topologically of finite type over $\Spa(F,\n{O}_F)$. Let $\Theta: \n{H}\times Z \rightarrow  Z$ be an action of $\n{H}$ over $Z$. Then for all $N>0$ there exists a neighbourhood $1\in U \subset \n{H}$ such that for all $g\in U $, $z\in Z$  and $f\in \s{O}^+(Z)$,  we have  $|f(z)-f(gz)|\leq |p|^{N}$. 
\proof
As $\s{O}^+(Z)=R^+$ is topologically of finite type over $\n{O}_F$, it suffices to prove the proposition for a single $f\in R^+$. Let $\Theta^*: R^+\rightarrow (A^+\widehat{\otimes}_{\n{O}_K} R^+)^+$ be the pullback of the multiplication map. Let $V\subset \n{H}\times Z$ be the open affinoid subspace defined by the equation 
\[
|1\otimes f- \Theta^*(f)|\leq |p|^{N}.
\] 
As $V$ contains $1\times Z$ and this is a quasi-compact closed subset of $\n{H}\times Z$, there exists $1\in U_f\subset U$ such that $U_f\times Z \subset V$. Therefore, for all $g\in U_f$ and $z\in Z$ we have $|f(z)-f(gz)|\leq |p|^{N}$.
\endproof
\end{lem} 
\begin{cor}
\label{corDevisageSheavesAetD}
Let $s\geq 1$.  There exists an open subgroup $H\subset \Iw_n$,  depending on $s$, which acts trivially on  $A^{\delta,+}_{\chi}/p^s$. In particular, we can write $A^{\delta,+}_{\chi}/p^s$ as a colimit of finite $R/p^s[\Iw_n]$-modules (dually, we can write $D^{\delta,+}_{\chi}/p^s$ as a projective limit of finite $R/p^s[\Iw_n]$-modules).   
\end{cor}
\begin{proof}
The affinoid group $\IW_n(\delta)$ acts on the space $\widetilde{U}_{w_0}(\delta,\delta)$ by right multiplication.  The corollary follows by  Lemma \ref{LemmaUniformContGroupAction} since $A^{\delta,+}_{\chi}$ is by definition an isotypic component of the global functions of  $\widetilde{U}_{w_0}(\delta,\delta)$. Dually, consider the map $D_{\chi}^{\delta,+}\to D_{\chi}^{\delta-1,+}$ and define $\Fil^s D_{\chi}^{\delta,+}= D_{\chi}^{\delta,+}\cap p^{s} D_{\chi}^{\delta-1,+}$.  Then the quotients $D_{\chi}^{\delta,+}/ \Fil^s D_{\chi}^{\delta,+}$ are finite $R^+/p^s$-modules and the weak topology of $D_{\chi}^{\delta,+}$ is the same as the inverse limit topology  (see \cite[Proposition 3.10]{AISOvShimura2015})
\[
D_{\chi}^{\delta,+}= \varprojlim_s D_{\chi}^{\delta,+}/\Fil^s D_{\chi}^{\delta,s}. 
\]
The corollary follows. 
\end{proof}

\begin{definition}
\label{DefinitionTensor}
Let $V$ be a Banach space over $\bb{Q}_p$ and $V^+\subset V$ a lattice. We define the following completed tensor products 
\[
A^{\delta}_{\chi}\widehat{\otimes} V =(\varprojlim_s A^{\delta,+}_{\chi}/p^s \otimes V^+/p^s) [\frac{1}{p}] \mbox{ and  } D^{\delta}_{\chi}\widehat{\otimes} V = (\varprojlim_{s} D^{\delta,+}_{\chi}/ \Fil^s D^{\delta,+}_{\chi} \otimes V^+/p^s) [\frac{1}{p}]. 
\]
\end{definition}

Next, let us explain how $A^{\delta}_{\chi}$ and $D^{\delta}_{\chi}$ interpolate the finite dimensional representations of $\GL_2$.

\begin{prop}
\label{PropClassicandFamilies}

Let $\kappa\in X^*(\bbf{T})$ be a dominant weight and $\delta \geq n$. There is a natural  $\Iw_n$-equivariant inclusion $V_\kappa \rightarrow A_\kappa^{\delta}$.   Dually. there is a natural  $\Iw_n$-equivariant surjective  map $ D^{\delta}_{\kappa}\rightarrow V_{\kappa}^{\vee}=V_{-w_0(\kappa)}$.

\proof
The second map is just the dual of the first one, the arrow $V_{\kappa } \to A^{\delta}_{\kappa}$ arises from the global sections functor applied to the inclusion $U_{w_0}(\delta)\Iw_n \subset \Fl$ and to the line bundle $\s{L}(\kappa)$. 
\endproof

\end{prop}

Finally, let us briefly discuss the $U_p$-correspondence of the principal series and distributions. We let $A_{\chi}^{\delta} \widehat{\otimes} \s{O}_{\Fl}$ and $D_{\chi}^{\delta} \widehat{\otimes} \s{O}_{\Fl}$ be the constant  $\Iw_n$-equivariant quasi-coherent sheaves over $\Fl$ induced by $A_{\chi}^{\delta}$ and $D_{\chi}^{\delta}$.  We have isomorphisms of $\s{O}_{\Fl}$-sheaves 
\begin{align*}
A^{\delta}_{\chi}\widehat{\otimes} \s{O}_{\Fl} =  \underset{i\in I}{\widehat{\bigoplus}} (R\widehat{\otimes } \s{O}_{\Fl}) e_i  \mbox{ and }
D^{\delta}_{\chi}\widehat{\otimes} \s{O}_{\Fl} = \prod_{i\in I} (R^+\widehat{\otimes } \s{O}_{\Fl}^+) e_i^{\vee}[\frac{1}{p}]. 
\end{align*}
By Remark \ref{RemarkUpActionFlag} one can endow $A^{\delta}_{\chi}$ with an action of $\varpi$ commuting with $\Iw_n$ (dually, one can endow $D^{\delta}_{\chi}$ with an action of $\varpi^{-1}$). Moreover, the multiplication by $\varpi$ on $A^{\delta}_{\chi}$ (resp. $D^{\delta}_{\chi}$) factors through  $A^{\delta-1}_{\chi}$ (resp. $D^{\delta+1}_{\chi}$) in accordance with Lemma \ref{Lemvarpidynamics}.    Using this action and diagram \eqref{eqCorrespondenceFlag}, one defines maps
\begin{equation}
\label{eqUpForAetD}
U_p^t:  p_{1}^* (A^{\delta}_{\chi}\widehat{\otimes} \s{O}_{\Fl}) \to p_2^*( A^{\delta-1}_{\chi} \widehat{\otimes} \s{O}_{\Fl} )\mbox{ and } U_p: p_{2}^* (D^{\delta}_{\chi}\widehat{\otimes} \s{O}_{\Fl}) \to p_1^* ( D^{\delta+1}_{\chi} \widehat{\otimes} \s{O}_{\Fl}),
\end{equation}
 see \cite[\S 3.5.2 and \S 4.5]{AIOverconvergentES2}  for more details.   We highlight that the previous $U_p$-correspondence is compatible with the maps of Proposition \ref{PropClassicandFamilies}, after normalizing the action of $\varpi$ and $c$ on $V_{\kappa}$, see Remark  \ref{RemakNormalizationUp} and the first Remark of \cite[\S 3.11]{AshStevensDeformations2008}.

\subsection{The dual BGG complex and the weight vector maps}

  Let $W=\{1,w_0\}$ be the Weyl group of $\GL_2$ and  $\bbf{B} w_0 \bbf{N}\subset \GL_2$  the big  cell.    We have a commutative diagram of torsors 
\begin{equation}
\label{eqTorsorsFlagBigcell}
\begin{tikzcd}
\bbf{B} w_0 \bbf{N} \ar[r] \ar[d] & \GL_2 \ar[d] \\ 
\bbf{B}\backslash   \bbf{B} w_0  \bbf{N} \ar[r] & \FL.
\end{tikzcd}
\end{equation}

Let $\kappa \in X^*(\bbf{T})$ be a character and $\s{L}(\kappa)$ the associated $\GL_2$-equivariant line  bundle over $\FL$, see Definition \ref{DefiEqLineBundleFlag}. Recall that, if $\kappa$ is dominant, the global sections of $\s{L}(\kappa)$ are isomorphic to $V_{\kappa}$. 

\begin{definition} We define the $(\f{g}, \bbf{B})$-representation  $ V(\kappa):=  \Gamma(\bbf{B}\backslash \bbf{B}w_0 \bbf{N}, \s{L}(\kappa))$, where the action of $(\f{g},\bbf{B})$ is induced by the right regular action on the big cell.  

\end{definition}

 As $\bbf{B}$-module, $V(\kappa)$ is a twist of the algebra of regular functions of $\bbf{N}$. Indeed, there is an isomorphism of affine schemes  $\bbf{B} \backslash \bbf{B} w_0 \bbf{N} \cong  \bbf{N}$, and one has 
\begin{equation}
\label{eqBGGresolutionfirststep}
V(\kappa) \cong  \kappa \otimes V(1)\cong  \kappa\otimes  \s{O}(\bbf{N}),
\end{equation}
where the action of $\bbf{B}=\bbf{T} \ltimes \bbf{N}$  on $\s{O}(\bbf{N})$ is induced from the map   $(n,b)\mapsto t_{b}^{-1}n t_{b} n_b$ for $(n,b)\in \bbf{N}\times \bbf{B}$ and $b= (t_b,n_b)\in\bbf{T} \ltimes \bbf{N} $. 
\begin{remark}
The $(\f{g}, \bbf{B})$-module $V(\kappa)$  is in fact the admissible dual of the Verma module of highest  weight $-w_0(\kappa)$,  see \S  3.10 of  \cite{AshStevensDeformations2008}. 
\end{remark}

 Let $\kappa$ be a dominant weight.    Taking the global sections of $\s{L}(\kappa)$ in the bottom arrow of diagram (\ref{eqTorsorsFlagBigcell}),  one obtains  a map 
 \[
 V_{\kappa} \rightarrow V(\kappa).
 \]
Writing $\kappa=(k_1,k_2)\in \bb{Z}^2$,  and  $\bb{G}_a\cong  \bbf{N} $ via $X \mapsto \left( \begin{matrix}
1 & X \\ 0 & 1
\end{matrix} \right)$,  the map of $V_{\kappa}$ in $V(\kappa)$ is identified in  (\ref{eqBGGresolutionfirststep}) with the inclusion  $\bb{Q}_p[X]_{k_1-k_2}\subset \bb{Q}_p[X]\cong \s{O}(\bbf{N})$  of   polynomials of degree $\leq k_1-k_2$.  We have the following:
\begin{prop}
\label{PropBGGclassic}
Let $\alpha=(1,-1)\in \bb{Z}^2 \cong X^*(\bbf{T})$ and let  $\kappa=(k_1,k_2)$ be  a dominant weight.  There is a short exact sequence of $(\f{g}, \bbf{B})$-representations  
\[
\BGG(\kappa): \;\; [ 0\rightarrow V_{\kappa} \rightarrow V(\kappa) \rightarrow V(w_0(\kappa)-\alpha)\rightarrow 0]
\]
called the dual BGG  complex of weight $\kappa$.   As $\bbf{B}$-representations it is identified with  the short exact sequence
\begin{equation}
\label{eqBGGdX}
0\rightarrow \kappa \otimes \bb{Q}_p[X]_{\leq k_1-k_2}\rightarrow \kappa \otimes \bb{Q}_p[X]\xrightarrow{(\frac{d}{dX})^{k_1-k_2+1}} (w_0(\kappa)-\alpha)\otimes \bb{Q}_p[X]\rightarrow 0,
\end{equation}
where $\bb{Q}_p[X] = \s{O}(\bbf{N})$. 
\proof
We have a weight decomposition of $V(\kappa)$ with respect to $\bbf{T}$
\[
V(\kappa)=  \bigoplus_{n\geq 0} (\kappa-n\alpha) \bb{Q}_p,
\]
where  $(\kappa-n\alpha)\bb{Q}_p$ is identified with $\kappa\otimes\bb{Q}_p  X^{n} $ under the isomorphism (\ref{eqBGGresolutionfirststep}).  As $V_{\kappa}$ is the irreducible representation of highest weight $\kappa$, it has a weight decomposition $V_{\kappa}\cong  \bigoplus_{0\leq n \leq k_1-k_2} (\kappa-n\alpha) \bb{Q}_p$. This shows that $V_\kappa \subset V(\kappa)$ is   identified with the inclusion $\kappa\otimes \bb{Q}_p[X]_{\leq k_1-k_2}\subset \kappa \otimes \bb{Q}_p[X]$.  As $\kappa \otimes X^{k_1-k_2+1}$ has weight $(w_0(\kappa)-\alpha)$,  the isomorphism of $\BGG(\kappa)$ with (\ref{eqBGGdX}) as $\bbf{B}$-representations is clear. 
\endproof
\end{prop}

The representation $V_{\kappa}$ has a $\bbf{B}$-filtration whose highest and lowest weight vectors are   $\bb{Q}_p(\kappa) \to V_{\kappa}$ and $V_{\kappa} \to \bb{Q}_p(w_0(\kappa))$ respectively.   Taking the associated $\GL_2$-equivariant vector bundles over $\FL$ one finds morphisms $\Psi_{-w_0(\kappa)}^{\vee}: \s{L}(w_0(\kappa)) \to \s{O}_{\FL} \otimes V_{\kappa}$ and $\Psi_{\kappa}: \s{O}_{\Fl} \otimes V_{\kappa} \to \s{L}(\kappa)$.  In Propositions \ref{PropOVBGGmaps} and \ref{CorAcompatibleV} we  interpolate these maps on neighbourhoods of $w\in W$. 
 
 \begin{prop}
\label{PropOVBGGmaps}
Let  $\epsilon \geq \delta \geq n$.  Let $(R,R^+)$ be a uniform  Tate $\bb{Q}_p$-algebra  and $\chi: T=\bbf{T}(\bb{Z}_p)\rightarrow R^{\times,+}$ a $\delta$-analytic character. 
\begin{enumerate}
\item There is a $\n{B}\cap \IW_n(\delta)$-equivariant map $\iota:  R(\chi)  \rightarrow A^{\delta}_{\chi}$ (the highest weight vector map).   It  induces a morphism of $\Iw_n$-equivariant sheaves over $U_{1}(\epsilon)\Iw_n$
\[
\Psi^{A,\vee}_{-w_0(\chi)}: \s{L}(w_0(\chi)) \rightarrow A_{\chi}^{\delta}\widehat{\otimes} \s{O}_{U_{1}(\epsilon)\Iw_n}.
\]  Dually, we have equivariant maps $D^{\delta}_{\chi}\rightarrow R(-\chi)$ and $\Psi^{D}_{-w_0(\chi)} :D_{\chi}^{\delta}\widehat{\otimes} \s{O}_{U_{1}(\epsilon)\Iw_n} \rightarrow \s{L}(-w_0(\chi))$.

\item There is a $\overline{\n{B}}\cap \IW_n(\delta)$-equivariant map $\ev_{w_0}: A^{\delta}_{\chi}\rightarrow  R(\chi)$ (the lowest weight vector quotient). Moreover,  it induces  a morphism of $\Iw_n$-equivariant sheaves over $U_{w_0}(\epsilon)\Iw_n$
\[
\Psi^{A}_{\chi} :A^{\delta}_{\chi} \widehat{\otimes} \s{O}_{U_{w_0}(\epsilon)\Iw_n} \rightarrow  \s{L}(\chi).
\] Dually, we have equivariant maps $R(-\chi ) \rightarrow D^{\delta}_{\chi}$ and  $\Psi^{D,\vee}_{\chi}:\s{L}(-\chi) \rightarrow D^{\delta}_{\chi} \widehat{\otimes} \s{O}_{U_{w_0}(\epsilon)\Iw_n}$.
\end{enumerate}

Moreover,  the morphisms of sheaves above are compatible with the $U_p$-correspondence \eqref{eqCorrespondenceFlag}.
\end{prop}

\begin{proof}
It is enough to prove the statements for the principal series. Recall that $A^{\delta}_{\chi}= \s{O}_{\widetilde{U}_{w_0}(\epsilon, \delta)}[-w_0(\chi)]$ where one takes isotypic components with respect to the left regular action of $T\n{T}(\delta)$.

\begin{enumerate}
\item   By Lemma \ref{LemIwahoriDecomposition} (3) we have an isomorphism of $R$-modules $A_{\chi}^{\delta} \cong  \s{O}( \overline{N} \overline{\n{N}}(\delta) w_0)$,  this shows that the $N\n{N}(\delta)$-invariants of  $A_{\chi}^{\delta}$ are  isomorphic to the $\n{B}\cap \IW_{n}(\delta)$-module $R(\chi)$, this provides the first arrow.  The map $\Psi^{A,\vee}_{-w_0(\chi)}: \s{L}(w_0(\chi)) \to A^{\delta}_{\chi}\widehat{\otimes} \s{O}_{U_{1}(\epsilon)\Iw_n} $ is constructed from the arrow $R(\chi)\to A^{\delta}_{\chi}$ via the presentation $U_{w_0}(\epsilon)\Iw_n= \n{B}\cap \IW_{n}(\delta)   \backslash \GLa_{2}(\epsilon,\delta) \Iw_n$.   Indeed, one has that 
\[
\s{L}(w_0(\chi))= \GLa_{2}(\epsilon,\delta)w_0\Iw_n\times^{\n{B}} R(\chi) \mbox{ and }  A_{\chi}^{\delta}\widehat{\otimes} \s{O}_{U_{1}(\epsilon)\Iw_n}= \GLa_{2}(\epsilon,\delta)w_0\Iw_n \times^{\n{B}} A_{\chi}^{\delta}.
\]

\item By Lemma \ref{LemIwahoriDecomposition} we have the presentation
\[
\widetilde{U}_{w_0}(\epsilon, \delta)\Iw_n =  N_n\n{N}(\delta) \backslash ( N_n \n{N}(\delta) \times T \n{T}(\delta) \times  \overline{N} \overline{\n{N}}(\epsilon) w_0). 
\]
Then, a straightforward  computation shows that the  evaluation map $\ev_{w_0}: A_{\chi}^{\delta} \to R(\chi)$ is  $\overline{\n{B}} \cap \IW_n(\delta)$-equivariant. The arrow  $\Psi^{A}_{\chi}: A_{\chi}^{\delta} \widehat{\otimes }\s{O}_{U_{w_0}(\epsilon)\Iw_n} \to \s{L}(\chi)$  is just the natural map induced by the global sections since $A_{\chi}^{\delta}= \Gamma(U_{w_0}(\epsilon)\Iw_n, \s{L}(\chi))$. Notice that it factors through the co-invariants of $A_{\chi}^{\delta}$ by the action of $\overline{N}_n \overline{\n{N}}(\delta)$.  
\end{enumerate}

By definition of the $U_p$-correspondence, it is enough to show that the maps $\Psi^{A,\vee}_{-w_0(\chi)}$ and $\Psi^{A}_{\chi}$ are equivariant for the action of $\varpi= \diag(1,p)$, this is clear since the multiplication by $\varpi$ on  $\s{L}(\chi)$ and $A_{\chi}^{\delta}$ is induced by the right multiplication of $\varpi$ on $\Lambda \backslash \widetilde{\Fl}$ (see Remark \ref{RemarkUpActionFlag} and Definition \ref{DefUpSheavesoverFlag}). 
\end{proof}

\begin{prop}
\label{CorAcompatibleV}
Let $\kappa \in X^*(\bbf{T})$ be a dominant character.    There are commutative diagrams of $\Iw_n$-equivariant sheaves 
  \[
   \begin{tikzcd} 
   A^{\delta}_{\kappa}\widehat{\otimes} \s{O}_{\Fl} \ar[r, "\Psi_{\kappa}^{A}"]  &  \s{L}(\kappa)  & & \s{L}(w_0(\kappa)) \ar[r, "\Psi^{A, \vee}_{-w_0(\kappa)}"] \ar[rd, "\Psi^{\vee}_{-w_0(\kappa)} "'] &  A^{\delta}_{\kappa}\widehat{\otimes} \s{O}_{\Fl}\\
   V_{\kappa}\otimes \s{O}_{\Fl} \ar[u] \ar[ur, "\Psi_{\kappa}"'] & & &  & V_{\kappa}\otimes \s{O}_{\Fl}    \ar[u]
   \end{tikzcd}
  \]
  A dual statement holds for $D^{\delta}_{\kappa}$.  Moreover, these diagrams  are compatible with the (normalized) $U_p$-correspondence  (Definition \ref{DefinitionUpCorrespondenceFlag}). 
  \proof
The commutativity of the diagrams is obvious from the definition of the maps $\Psi^{A}_{\kappa}$ and $\Psi^{A,\vee}_{-w_0(\kappa)}$ of Proposition \ref{PropOVBGGmaps},  and the fact that $\Psi^{\vee}_{-w_0(\kappa)}$  and $\Psi_{\kappa}$ are induced by the  invariants and co-invariants for the action of $\bbf{N}$ on $V_{\kappa}$ respectively. The compatibility with the (normalized) $U_p$-correspondence follows from the fact that the (normalized) action of $\varpi$ on any of the sheaves involved is induced by the right multiplication on $\Lambda \backslash \widetilde{\Fl}$ (see Remark \ref{RemarkUpActionFlag} and Definition \ref{DefUpSheavesoverFlag}).
  \endproof
\end{prop}

\begin{remark}
In  \cite[\S 4.7 and 4.8]{AIOverconvergentES2} the authors define some filtrations attached to the sheaves of modular symbols over the modular curves, it turns out that these can be constructed directly from the flag variety. Indeed, AI use the formalism of vector bundles with marked sections to define the filtrations, see \cite[Corollary 2.6]{TripleProdAI}, and the data of a vector bundle with  marked section lives over affinoid neighbourhoods of $w\in W$.  For example, let $\St$ be the standard representation of $\GL_2$ and $\St^+\subset \St$ its natural lattice, let us denote by $e_1,e_2$ the canonical basis of $\St^+$.    Over $U_1(\epsilon)$ the natural map $\s{L}^+(0,1)\to \St^+ \otimes \s{O}_{U_1(\epsilon)}$ induces an isomorphism $\s{L}^+(0,1)/p^\epsilon \cong e_1\otimes \s{O}^+_{U_1(\epsilon)}$.   Therefore, we have the data of vector bundle with marked sections $(\St^+\otimes \s{O}_{U_1(\epsilon)},  \s{L}(0,1), e_1)$.  The previous discussion shows in particular  that the map of \cite[Proposition 4.15 ii.]{AIOverconvergentES2} is $\Psi^{D}_{\chi}$. 

Following the work of Pan \cite{pan2020locally},  there is a better way to study the  filtrations above.  Let $\f{gl}_2^0= \f{gl}_2 \otimes \s{O}_{\Fl}$ be the constant Lie algebroid over $\Fl$,  and  let $\f{n}^0 \subset \f{gl}_2^0$ be the $\GL_2$-equivariant line bundle whose fiber at a point $x\in \Fl$ is $\f{n}^0(x)= \Lie x \bbf{N} x^{-1}$, i.e., the Lie algebra of the unipotent group fixing $x$.  Then $\f{n}^0$ is an ideal of $\f{gl}_2^0$, and the natural action of $\f{n}^0$ on $A^{\delta}_{\chi} \widehat{\otimes} \s{O}_{\Fl}$ induced by  derivations of $\f{gl}_2^0$ defines the filtrations of AI. 
\end{remark}


\section{Overconvergent sheaves over the modular curve}
\label{SectionOvtheoryOverModCurves}

Let  $\bb{A}^{\infty}_\bb{Q}$ and $\bb{A}^{\infty,p}_{\bb{Q}_p}$ be the rings of finite and finite prime-to-$p$ ad\`eles of $\bb{Q}$ respectively.    From now on, we fix a neat compact open subgroup $K^p\subset \GL_2(\bb{A}^{\infty,p}_{\bb{Q}})$.   Let $n\geq 0$,  we denote by $\Gamma(p^n)$, $\Gamma_1(p^n)$ and $\Gamma_0(p^n)$ the principal congruence subgroups 
\begin{gather*}
\Gamma(p^n)=\{g\in \GL_2(\bb{Z}_p) : g\equiv 1 \mod p^n\} \\ 
\Gamma_1(p^n)=\{g\in \GL_2(\bb{Z}_p) : g\equiv \left( \begin{matrix}   1 & * \\ 0 & 1 \end{matrix} \right)  \mod p^n\} \\
\Gamma_0(p^n)=\{g\in \GL_2(\bb{Z}_p) : g\equiv \left( \begin{matrix}   * & * \\ 0 & * \end{matrix} \right)  \mod p^n\}.
\end{gather*}
Let $K_p\subset \GL_2(\bb{Q}_p)$ be a compact open subgroup, we denote by $Y_{K_p}^{\alg}$  and $X_{K_p}^{\alg}$ the  affine and compact   modular curves of level $K^pK_p$ over $\Spec \bb{Q}_p$, cf.   \cite{DeligneRappLesSchemasCourbes}.   We let $Y^{\alg}(p^n)$, $Y^{\alg}_1(p^n)$ and $Y^{\alg}_0(p^n)$ be the modular curves of level $K^p\Gamma(p^n)$,  $K^p\Gamma_1(p^n)$ and $K^p\Gamma_0(p^n)$ respectively (similarly for the compact modular curves).   We let $Y_{K_p}$ and $X_{K_p}$ denote their $p$-adic analytification to adic spaces over $\Spa(\bb{Q}_p, \bb{Z}_p)$, see   \cite{HuberEtaleCohomology}.  We endow $X_{K_p}$ with the log structure defined by the cusp divisor $D= X_{K_p}\backslash Y_{K_p}$.

 Let $E^{\alg}/Y_{K_p}^{\alg}$ be the universal elliptic curve and $E^{\alg,\sm}/X_{K_p}^{\alg}$   its extension to a semi-abelian scheme.  Let $e: X_{K_p}^{\alg}\rightarrow E^{\alg, \sm}$ be the unit section and $\omega_E=e^* \Omega^1_{E^{\alg, \sm}/X^{\alg}}$ the modular sheaf.    Given an integer $k\in \bb{Z}$ we denote by  $\omega_E^{k}=\omega_E^{\otimes k}$ the  modular sheaf of weight $k$. We define the modular torsor to be the $\n{T}$-torsor over $X_{K_p}$ defined by 
 \begin{equation}
 \label{eqModTorsor}
 \n{T}_{\mo}= \underline{\Isom}( \s{O}_X,  \omega_E) \times \underline{\Isom}(\s{O}_{X} , \omega_E^{-1}).
 \end{equation}

Let $E[p^n]/Y_{K_p}$  be the  \'etale local system of $p^n$-torsion points of the universal elliptic curve.  By \cite[Theorem 4.6.1]{DiaoLogarithmic2019},  the sheaf $E[p^n]$ has a natural extension to a Kummer-\'etale  local system over $X_{K_p}$,  which by an abuse of notation we also  write as $E[p^n]$.  We let $T_pE = \varprojlim_{n} E[p^n]$ be the Tate module of $E$, seen as a pro-Kummer-\'etale local system over $X_{K_p}$.  Given $\kappa \in X^*(\bbf{T})^+$ a dominant weight we let $V_{\kappa, \et}$ denote the pro-Kummer-\'etale local system over $X_{K_p}$ attached  to the $K_p$-representation $V_{\kappa}$.  We let $R\Gamma_{\proet}(Y_{\bb{C}_p}, V_{\kappa,\et})$ and $R\Gamma_{\proet,c}( Y_{\bb{C}_p}, V_{\kappa,\et})$ be the \'etale cohomology and  the \'etale cohomology with compact support of $V_{\kappa,\et}$ respectively.

\subsection{The Hodge-Tate period map}

Let $\bb{Q}_p^{\cyc}$ be the $p$-adic completion of the $p$-adic cyclotomic field $\bb{Q}_p(\mu_{p^\infty})$.    Scholze proved in \cite{ScholzeTorsion2015} that the inverse limit $X(p^\infty)=\mbox{``} \varprojlim_n X(p^n)\mbox{''}$ has a natural interpretation as  perfectoid space.  Furthermore, he constructed a Hodge-Tate period  map $\pi_{\HT}: X(p^\infty)\rightarrow \bb{P}^1_{\bb{Q}_p}$ parametrizing the Hodge-Tate filtration of elliptic curves at  geometric points, let us briefly recall how it is constructed.  The Hodge-Tate exact sequence is the following short exact sequence of pro-Kummer-\'etale sheaves over $X_{K_p}$ 
\begin{equation}
\label{eqHTExactseq}
0 \to \omega_E^{-1}\otimes_{\s{O}}\widehat{\s{O}}(1) \to T_pE \otimes_{\bb{Z}_p} \widehat{\s{O}} \to \omega_E \otimes_{\s{O}} \widehat{\s{O}} \to 0. 
\end{equation}
On the other hand, the perfectoid space $X(p^{\infty})$ parametrizes trivializations of $T_pE$. If $\psi: \bb{Z}_p^2 \cong T_pE$ denote the  universal trivialization of the Tate module,   the pullback of \eqref{eqHTExactseq}  by $\psi$  gives rise a line subbundle of $\s{O}_{X(p^{\infty})}^{\oplus 2}$ which defines the map $\pi_{\HT}: X(p^{\infty}) \to \Fl= \bb{P}^1_{\bb{Q}_p}$.   Let us summarize the previous discussion  and some properties of $\pi_{\HT}$ in the following theorem.

\begin{theo}[Theorem III.3.18  \cite{ScholzeTorsion2015}]
\label{TheoScholzeAffineHodgeTate}
There exists a perfectoid space $X(p^\infty)$ over $\bb{Q}_p^{\cyc}$ satisfying the tilde limit property \cite[Definition 2.4.1]{scholze2013moduli}
\[
X(p^\infty)\sim  \varprojlim_n X(p^n). 
\]
Moreover, let $[x:y] \in \Fl= \bb{P}^1_{\bb{Q}_p}$ denote the projective coordinates of the projective space.  There is a $\GL_2(\bb{Q}_p)$-equivariant Hodge-Tate period map 
\[
\pi_{\HT}: X(p^\infty)\rightarrow \Fl
\]
such that for any open rational subset $U$ of $U_1= \{[x:y] | |x/y|\leq 1\}$ or $U_{\infty}= \{[x:y] | |y/x|\leq 1\}$ of $\Fl$,  the inverse image $\pi^{-1}_{\HT}(U)\subset X(p^\infty)$ is an affinoid perfectoid subspace, and there is $n>>0$ and an open affinoid $V_n\subset X(p^n)$ whose inverse image to $X(p^\infty)$ is equal to $\pi_{\HT}^{-1}(U)$. 
\end{theo}

Another feature of the Hodge-Tate period map is that it encodes the modular sheaves in terms of $\GL_2$-equivariant line bundles over $\Fl$. More precisely, let $\St$ be the standard representation, we have an exact sequence of $\bbf{B}$-modules
\[
0\to \bb{Q}_p(1,0)\to \St \to \bb{Q}_p(0,1)\to 0.
\]
Taking the associated $\GL_2$-equivariant vector bundles over $\Fl$ we have a short exact sequence 
\begin{equation}
\label{eqshortsequence1}
0\to  \s{L}(0,1)\to  \St \otimes \s{O}_{\Fl} \to  \s{L}(1,0) \to  0.
\end{equation}
Now, the map $\pi_{\HT}$ induces a $\GL_2(\bb{Q}_p)$-equivariant  morphism of ringed sites 
\[
\pi_{\HT}:  (X(p^{\infty})_{\proket}, \widehat{\s{O}}_{X} ) \to (\Fl, \s{O}_{\Fl}).
\]
In particular, we can take pullbacks of $\GL_2$-equivariant  $\s{O}_{\Fl}$-vector bundles over $\Fl$ to $\GL_2(\bb{Q}_p)$-equivariant  $\widehat{\s{O}}_X$-vector bundles over $X(p^{\infty})_{\proket}$.

\begin{convention}
Given a $K_p$-equivariant sheaf $\s{F}$ over $\Fl$, we will identify $\pi_{\HT}^*(\s{F})$ with the pro-Kummer-\'etale sheaf over $X_{K_p}$ defined by descent from the $K_p$-equivariant sheaf over $X(p^{\infty})_{\proket}$.  
\end{convention}

\begin{remark}
\label{RemarkConventionPiHT}
 Essentially by definition, the pullback of \eqref{eqshortsequence1} by  $\pi_{\HT}$ is the Hodge-Tate exact sequence \eqref{eqHTExactseq}. Our convention differs from the one of \cite[\S 4.3]{AIOverconvergentES2}, where the pullback of the standard representation by $\pi_{\HT}$ is identified with the dual of $T_pE$.  
\end{remark}

Let $\kappa\in X^*(\bbf{T})^+$ be a dominant weight and $V_{\kappa}$ the irreducible representation of $\GL_2$ of highest weight $\kappa$. Let us see $V_{\kappa}$ as a constant $\GL_2(\bb{Q}_p)$-equivariant  sheaf over $\Fl$, one has that $\pi_{\HT}^*(V_{\kappa})= V_{\kappa, \et}$ as pro-Kummer-\'etale sheaves. This implies that the pullback by $\pi_{\HT}$ of the $\GL_2$-equivariant vector bundle $V_{\kappa}\otimes \s{O}_{\Fl}$ is equal to $V_{\kappa, \et}\otimes \widehat{\s{O}}_X$.   Concerning  the $\GL_2$-equivariant line bundles over $\Fl$ one has the following result, cf.  \cite[Proposition 2.3.9]{CaraianiScholze2017}

\begin{prop}
\label{PropPullbackModtorsor}
Let $\kappa=(k_1,k_2) \in X^*(\bbf{T})$ be an algebraic weight and $\s{L}(\kappa)$ the $\GL_2$-equivariant line bundle over $\Fl$ of weight $\kappa$. There is a natural isomorphism of pro-Kummer-\'etale sheaves over $X_{K_p}$
\[
\pi_{\HT}^* (\s{L}(\kappa)) = \omega_E^{k_1-k_2}  \otimes_{\s{O}} \widehat{\s{O}} (k_2).
\]
Equivalently, let $\widetilde{\Fl}= \n{N} \backslash \GLa_2$ and $\pi_{K_p}:X(p^{\infty}) \to X_{K_p}$. We have an isomorphism of $\GL_2(\bb{Q}_p)$-equivariant $\n{T}$-torsors over $X(p^{\infty})$ 
\[
\pi_{\HT}^*(\widetilde{\Fl})= \pi_{K_p}^*(\n{T}_{\mo})(-1,0),
\]
where $\n{T}_{\mo}(-1,0)$ is a Tate twist   in the first component of the modular torsor  \eqref{eqModTorsor}.
\end{prop}
\begin{proof}
We only need to show the isomorphism of torsors, this follows from   Remark  \ref{RemarkOtherDescriptionLineBundle} (1) and the definition of the modular torsor.   In fact, since the pullback of \eqref{eqshortsequence1} is the Hodge-Tate exact sequence, one has that $\pi_{\HT}^*(\s{L}(1,0))= \omega_E^1\otimes \widehat{\s{O}}_X$ and $\pi_{\HT}^*(\s{L}(0,1))= \omega_E^{-1}\otimes \widehat{\s{O}}(1)$, the proposition follows since $ \s{L}(\kappa)= \s{L}(1,0)^{\otimes k_1}\otimes \s{L}(0,1)^{\otimes k_2}$. 
\end{proof}

\subsection{Overconvergent modular forms}

Throughout the rest of this section we will fix $n\geq 1$, we  write $X_{\infty}=X(p^{\infty})$ and $X= X_0(p^n)$.  Let $\pi_{\Iw_n}: X_{\infty}\to X$ be the natural projection and $\pi_{\HT}: X_{\infty}\to \Fl \cong \bb{P}^1_{\bb{Q}_p}$ the Hodge-Tate period map. Let $\overline{X}^{\ord}_{\infty}= \pi_{\HT}^{-1}(\Fl(\bb{Q}_p))$ be the closure of the  ordinary locus at infinite level (\cite[Lemma III.3.6.]{ScholzeTorsion2015}), and $\overline{X}^{\ord} = \pi_{\Iw_n}(\overline{X}^{\ord}_{\infty})$ the  closure of the ordinary locus of $X$. 

  Let $C^{\can}_n\subset E[p^n]$ be the canonical subgroup over $\overline{X}^{\ord}$ and $w\in W=\{1,w_0\}$.    We let $\overline{X}^{\ord}_{w}\subset \overline{X}^{\ord}$ denote the $w$-ordinary locus, i.e. the ordinary locus where $C_n^{\can}$ has relative position $w$ with  respect to the universal subgroup $H_n\subset E[p^n]$.  In other words,   $\overline{X}_{1}^{\ord}$ is the ordinary locus where $C_n^{\can}= H_n$ and $\overline{X}_{w_0}^{\ord}$ the locus where $C_n^{\can}\cap H_n=0$.  We can also  write  $\overline{X}^{\ord}_{w}=\pi_{\Iw_n}(\pi_{\HT}^{-1}(w\Iw_n))$.

Let $\epsilon \geq n \geq 1$ and $w\in W=\{1,w_0\}$.   In \S \ref{SubsectionOpenAffinoidFlag} we have defined  affinoid neighbourhoods $ \{U_w(\epsilon)\Iw_n\}_{\epsilon\geq n} $ of $w\in \Fl$.   By Theorem \ref{TheoScholzeAffineHodgeTate} their pullback to $X_{\infty}$ are affinoid perfectoids arising from some finite level modular curve $X_{K_p}$. Furthermore, as $\pi_{\HT}^{-1}(U_w(\epsilon)\Iw_n)$ is $\Iw_n$-stable, we can take $K_p=\Iw_n$. The previous discussion  leads to the following definition: 

\begin{definition}
\label{DefOpensFlag}
There exists a unique open affinoid subspace $X_{w}(\epsilon)\subset X$ such that  $\pi_{\Iw_n}^{-1}(X_{w}(\epsilon))= \pi_{\HT}^{-1}(U_w(\epsilon)\Iw_n)$. 
\end{definition}

\begin{remark}
The following properties are deduced form Lemma \ref{LemIwahoriDecomposition}. 
\begin{enumerate}

\item  $X_{w}(\epsilon') \subset X_{w}(\epsilon)$ is a strict immersion for $\epsilon'>\epsilon$.

\item $\{X_{w}(\epsilon)\}_{\epsilon\geq n}$ is a basis of strict neighbourhoods of $\overline{X}^{\ord}_{w}$, namely, $\bigcap_{\epsilon\geq n} X_{w}(\epsilon)= \overline{X}^{\ord}_w$.   
\end{enumerate}
\end{remark}

The affine modular curve $Y\subset X$ parametrizes triples $(E,H_n, \psi_N)$ where $E$ is an elliptic curve $E$, $\psi_N$ is some prime-to-$p$ level structure,  and $H_n\subset E[p^n]$ is a cyclic subgroup of order $p^n$.  
Let $\varpi=\diag(1,p)$.  In the following we study the dinamycs of the $U_p$-correspondence, cf. \cite[\S 5.3]{BoxerPilloniHigher2020}. 

\begin{definition}   The $U_p$-correspondence of $X$ is the finite    flat correspondence $C$ fitting into the diagram
\begin{equation}
\label{eqcorrespondance}
\begin{tikzcd}
 & C  \ar[ld, "p_1"'] \ar[rd, "p_2"] &  \\ X  & & X
\end{tikzcd}
\end{equation}
and parametrizing tuples $(E,H_n, \psi_{N},H')$  where $(E,H_n,\psi_N)$ defines a point in $X$,    and $H'\subset E[p]$ is a cyclic subgroup of order $p$ such that $H_n\cap H'=0$. We define  $p_1(E, H_n,\psi_N,H')=(E, H_n,  \psi_N)$ and $p_2(E, H_n ,\psi_N,H')= (E/H',\overline{H}_n ,\overline{\psi}_N)$, where $\overline{\psi}_N$ and $\overline{H}_n$ are the images of $\psi_N$ and $H_n$  in the quotient $ E/H'$.  Let  $\pi: p_1^*E\rightarrow p_2^*E$  be the universal  isogeny over $C$  and $\pi^\vee: p_2^* E\rightarrow p_1^* E$ its dual.   For a subspace $Z\subset X$ let us denote  $U_p(Z)= p_1(p^{-1}_2(Z))$ and $U_p^{t}(Z)= p_2(p_1^{-1}(Z))$.

\end{definition}
\label{SubsectionOVneighbOrdLocus}

The following Lemma uses the same strategy  of \cite[\S 4.5]{AIOverconvergentES2}, notice however that the convention on $\pi_{\HT}$ differs, see Remark \ref{RemarkConventionPiHT}.

\begin{lem}
\label{LemmaDynamicsUp}
Let $\epsilon \geq n$, the following holds
\begin{enumerate}

\item  $U_p^t(X_{1}(\epsilon)) \subset X_{1}(\epsilon+1)$ and $U_p(X_{1}(\epsilon)) \supset X_{1}(\epsilon-1)$ if $\epsilon \geq n+1$.

\item  $U_p(X_{w_0}(\epsilon)) \subset X_{w_0}(\epsilon+1)$  and $U_p^t( X_{w_0}(\epsilon)) \supset X_{w_0}(\epsilon-1)$ if  $\epsilon\geq  n+1$.

\end{enumerate}

\proof

The perfectoid modular curve $X_{\infty}$ parametrizes triples  $(E,\psi_N, (e_1,e_2))$ where $E$ is an elliptic curve, $\psi_N$ a  prime-to-$p$ level structure,   and $(e_1,e_2)$ is a basis of $T_pE$.    Let   $C_{\infty}=X_{\infty}\times_{X,p_1} C$.  The perfectoid curve  $C_{\infty}$  parametrizes  $(E, \psi_N,(e_1,e_2), H' )$ where $(E,\psi_N, (e_1,e_2))\in X_\infty$  and $H'\subset E[p^n]$ is a cyclic subgroup of oder $p$ such that $\langle   \overline{e}_1 \rangle\cap H'=0 \mod p $.  Write $C_{\infty}= \bigsqcup_{a\in \bb{F}_p}  C_{\infty,a}$  with $C_{\infty,a}$ the locus where  $H'=\langle \overline{ e}_2+a \overline{e}_1\rangle$.  Note that the map $p_1:C_{\infty,a}\rightarrow X_{\infty}$ is an isomorphism for all $a$.  We have a diagram 
\[
\begin{tikzcd}
 & \bigsqcup_{a=0}^{p-1}  C_{\infty,a} \ar[ld, "p_1"'] \ar[rd, "p_2"] & \\ X_{\infty} & & X_{\infty} 
\end{tikzcd}
\] 
with $p_1(E,\psi_N, (e_1,e_2), H')=(E,\psi_N, (e_1,e_2))$ and $p_2(E,\psi_N, (e_1,e_2), H')= (E/H', \overline{\psi}_N, (\pi(e_1), \tilde{e}_2))$,   such that the restriction of $p_2$ to $C_{\infty,a}$ is given by  $\tilde{e}_2= \frac{1}{p}(\pi(e_2)+a\pi(e_1))$  for $0 \leq a <p$ lifting $a$. Let $U_{p,a}:=\left( \begin{matrix}
1 & -a \\ 0 & p
\end{matrix} \right)$.  Composing with the Hodge-Tate period map $\pi_{\HT}: X_{\infty}\rightarrow \Fl$ we have  a commutative diagram  of correspondences (see diagram \eqref{eqCorrespondenceFlag})
\begin{equation}
\label{eqMapsCorrInfinite}
\begin{tikzcd}[row sep = 5pt]
 & \bigsqcup_a  C_{\infty,a} \ar[ld, "p_1"'] \ar[rd, "p_2"] \ar[dd] & \\ X_{\infty}  \ar[dd, "\pi_{\HT}"'] & & X_{\infty}  \ar[dd,  "\pi_{\HT}"] \\ 
  & \bigsqcup_{a=0}^{p-1} \Fl_a \ar[rd, "p_2"] \ar[ld, "p_1"'] &  \\
 \Fl & & \Fl
\end{tikzcd}
\end{equation}
where the map $p_{2,a}: \Fl_a \to \Fl$ is the  right multiplication by $U_{p,a}^{-1}$. Indeed, if $e_1,e_2$ is the canonical basis of the standard representation of $\GL_2$ we have that $p_{2,a}^*(e_1)=e_1$ and $p_{2,a}^*(e_2)= \frac{1}{p}(e_2+ae_1)$.  The Lemma follows by Lemma \ref{LemCorrespondenceFlag} and the definition of the affinoids $X_w(\epsilon)$.\endproof
\end{lem}

\begin{remark}
\label{RemarkCorresFiniteLevel}
Taking $\Iw_n$-quotients in \eqref{eqMapsCorrInfinite} one obtains the morphism of correspondences 
\[
\begin{tikzcd}[row sep = 5pt]
 & C \ar[ld, "p_1"'] \ar[rd, "p_2"] \ar[dd] & \\ X  \ar[dd, "\pi_{\HT}"'] & & X  \ar[dd,  "\pi_{\HT}"] \\ 
  & (\bigsqcup_{a=0}^{p-1} \Fl_a)/\Iw_n \ar[rd, "p_2"] \ar[ld, "p_1"'] &  \\
 \Fl/\Iw_n & & \Fl/\Iw_n,
\end{tikzcd}
\]
where the bottom correspondence is the one of \eqref{eqCorrespondenceStack}. Therefore, the $U_p$-correspondence of $X$ is simply the pullback of the $U_p$-correspondence of the stack $\Fl/\Iw_n$. 
\end{remark}

Let $\n{T}^{0} \subset \n{T}$ be the affinoid bounded torus given by the generic fiber of the $p$-adic completion of $\bbf{T}$.  Let $\n{T}_{\mo,\et}$ be the base change of $\n{T}_{\mo}$ to a $\n{T}$-torsor over the \'etale site of $X$.  In order to construct overconvergent modular sheaves we have to find refinements of the torsor $\n{T}_{\mo}$.   It turns out that the torsor  $\n{T}_{\mo}$  admits an  integral reduction to an  \'etale torsor:  
\begin{theo}[{\cite[\S 4.6]{BoxerPilloniHigher2020}} ]
\label{TheointTorsor}
Let $\widetilde{\Fl}^0= \n{N}^0 \backslash \GLa_2^{0}$ be the natural $\GLa_2^0$-equivariant  $\n{T}^0$-torsor over $\Fl$.  There exists an \'etale $\n{T}^0$-torsor  $\n{T}^0_{\mo,\et}$ over  $X$ such that 
\[
\n{T}_{\mo,\et} = \n{T} \times^{\n{T}^0} \n{T}^0_{\mo,\et} \mbox{ and }
\pi_{\HT}^*(\widetilde{\Fl}^0)= \pi_{\Iw_n}^*(\n{T}_{\mo, \et}^0) (-1,0).
\] 
\end{theo}
\begin{remark}
The existence of the integral torsor holds in greater generality for any Shimura variety. The  theorem follows from the fact that $\pi_{\HT}^*(\widetilde{\Fl}^0)$ is a $\GL(\bb{Z}_p)$-equivariant open subspace of the  twisted torsor $\pi_{\Iw_n}(\n{T}_{\mo}) (-1,0)$ over $X_{\infty}$, it follows that this open suspace descends to some finite level providing, locally \'etale on $X$, an integral trivialization of $\n{T}_{\mo}$. 
\end{remark}
 
The following definition is justified by the proof of Theorem \ref{TheointTorsor}, see \cite[Proposition 4.6.12]{BoxerPilloniHigherColeman2020} or \cite[Proposition 5.15]{BoxerPilloniHigher2020}. 

\begin{definition}
\label{PropTdeltatorsor}
Let $\epsilon\geq \delta \geq n \geq 1$ and $w\in W=\{1,w_0\}$.  Consider the $\Iw_n$-equivariant  $T\n{T}(\delta)$-torsor of Definition \ref{DefinitionReductionTorsorFlag}
\[
\widetilde{U}_w(\epsilon,\delta)\Iw_n \rightarrow U_{w}(\epsilon)\Iw_n. 
\]
The restriction of $\n{T}_{\mo,\et}^0$ to $X_{w}(\epsilon)$ admits a  reduction to  an \'etale $T\n{T}(\delta)$-torsor $\n{T}_{\mo}(\delta)$  determined by the equality 
\begin{equation}
\label{eqTorsorpullbackIW}
\pi_{\HT}^*(\widetilde{U}_{w}(\epsilon, \delta)\Iw_n) = \pi_{\Iw_n}^*( \n{T}_{\mo}(\delta))(-1,0)
\end{equation}
as open subspaces of $\pi_{\HT}^*(\widetilde{\Fl})=\pi_{\Iw_n}^*( \n{T})(-1,0)$.
\end{definition}

\begin{definition}
\label{DefOvModsheaf}
Let $(R,R^+)$ be a uniform affinoid Tate $\bb{Q}_p$-algebra, and $\chi: T=\bbf{T}(\bb{Z}_p)\rightarrow R^{+,\times}$ a $\delta$-analytic character. Let $\s{O}_{\n{T}_{\mo}(\delta)}$ be the algebra of regular functions of $\n{T}_{\mo}(\delta)$, seen as an \'etale Banach  $\s{O}_X$-algebra over $X_w(\epsilon)$. The sheaf of overconvergent modular forms of weight $\chi$ is given by 
\begin{align*}
\omega_E^{\chi} & = \s{O}_{\n{T}_{\mo}(\delta)} \widehat{\otimes }R  [-w_0(\chi)] \\
				& = \{f \in \s{O}_{\n{T}_{\mo}(\delta)} \widehat{\otimes} R :  f(tx) = w_0(\chi)(t) f(x) \mbox{ for } t\in T\n{T}(\delta)\}. 
\end{align*}
\end{definition}

\begin{remark}
\label{RemarkOvTorsorFromAnalytic}
In \cite[Proposition 4.6.15]{BoxerPilloniHigherColeman2020} it is shown that the torsor $\n{T}_{\mo}(\delta)$ is  trivial  in a   finite \'etale covering of  $X_w(\epsilon)$. This implies that the \'etale sheaf $\omega_{E}^{\chi}$ is locally in the \'etale topology an orthonormalizable $\s{O}_X$-sheaf, and equal to  the  pullback to the \'etale site of an   $\s{O}_{X_w(\epsilon)}\widehat{\otimes}R$-line bundle over the analytic site of $X_w(\epsilon)$. 
\end{remark}

From Definition \ref{PropTdeltatorsor} we deduce the following overconvergent analogue of Proposition \ref{PropPullbackModtorsor}.

\begin{cor}
\label{PropLtoOmega}
Let $(R,R^+)$ and $\delta$ be as in Definition \ref{DefOvModsheaf}, write  $\chi=(\chi_1,\chi_2)$.  Let $\chi_{\cyc}: G_{\bb{Q}_p}\rightarrow \bb{Z}_p^\times$ be the cyclotomic character and $\chi_2\circ \chi_{\cyc}: G_{\bb{Q}_p}\rightarrow  R^{+,\times}$ its composition with $\chi_2$.  We set  $\widehat{\s{O}}_{X}(\chi_2):= R( \chi_2\circ \chi_{\cyc}) \widehat{\otimes} \widehat{\s{O}}_X$.  There is a Galois equivariant isomorphism of pro-Kummer-\'etale  sheaves over $X_{w}(\epsilon)$
\[
\pi_{\HT}^*(\s{L}(\chi)) = \omega_E^{\chi} \widehat{\otimes}_{R\widehat{\otimes } \s{O}_X} \widehat{\s{O}}_X(\chi_2).  
\]
\end{cor}

We can finally define the overconvergent modular forms and the overconvergent cohomology classes appearing in  higher Coleman theory.   We refer to \cite{UrbanEigenvar} for the notion of perfect Banach complexes and compact operators of perfect Banach complexes.   See \cite[\href{https://stacks.math.columbia.edu/tag/0A39}{Tag 0A39}]{stacks-project} for the definition of cohomology with supports  in a closed subspace. 

\begin{definition}
\label{DefiSupportOvcohomologies}
 Let $w\in W=\{1,w_0\}$ and let  $\s{F}$ be a sheaf over $X_w(\epsilon)_{\an}$. Denote $X_w(>\epsilon):= \bigcup_{\epsilon'>\epsilon} X_w(\epsilon)$.   We define the cohomology complexes 
 \[
 R\Gamma_w(X, \s{F})_{\epsilon}:= R\Gamma_{\an}(X_{w}(\epsilon), \s{F}) \mbox{ and } R\Gamma_{w,c}(X, \s{F})_{\epsilon}:= R\Gamma_{\an, \overline{X_w(>\epsilon+1)}}(X_{w}(\epsilon), \s{F}). 
 \]
Set $H^0_w(X,\s{F}):=H^0( R\Gamma_w(X, \s{F})_{\epsilon})$ and  $H^1_{w,c}(X,\s{F})=H^1(R\Gamma_{w,c}(X, \s{F})_{\epsilon})$.   When $\s{F}= \omega_E^{\chi}$, we call $H^0_w(X, \omega_E^{\chi})_{\epsilon}$ and $H^1_{w,c}(X, \omega_W^{\chi})_{\epsilon}$ the space of overconvergent modular forms  and the overconvergent cohomology with compact support of weight $\chi$ respectively. 
\end{definition}

\subsubsection{Hecke operators}  We end this section with the definition of the $U_p$-operators  for the overconvergent modular forms.  First,  let us recall the definition for the classical modular sheaves.  Let   $X   \xleftarrow{p_1} C\xrightarrow{p_2} X$ be the $U_p$-correspondence. We let $\pi: p_1^*E \rightarrow p_2^* E$ be the universal isogeny over $C$ and $\pi^{\vee}:  p_2^* E \rightarrow p_1^* E$ its dual. We denote by $\pi^*: p_2^* \omega_E \rightarrow p_1^* \omega_E$ and $\pi_*: p_1^* \omega_E^{-1}\rightarrow p_2^* \omega_E^{-1}$ the pullback and pushforward maps of $\pi$ (resp. for $\pi^\vee$).      For a quasicoherent sheaf $\s{F}$ over $X$ we let $\Tr_{p_i}:  p_{i,*} p^*_i \s{F}\rightarrow \s{F}$ be the trace map of  $p_i$.    Let $\kappa= (k_1,k_2)\in X^*(\bbf{T})$ be a character  of $\bbf{T}$,  recall that we have made the convention $\omega_E^{\kappa}= \omega_E^{k_1}\otimes \omega_E^{-k_2}= \omega_E^{k_1-k_2}$.

\begin{definition}
\label{DefUpOperatorclassical}
 The Hecke operator $U_{p,\kappa}$  acting over  $R\Gamma_{\an}(X, \omega_E^{\kappa})$ is the composition 
\[
\begin{tikzcd}
R\Gamma_{\an}(X, \omega_E^\kappa) \ar[r, "p_2^*"]& R\Gamma_{\an}(C, p_2^*\omega_E^\kappa ) \xrightarrow{(\pi^{\vee,*,-1})^{\otimes k_1} \otimes (\pi_*^{-1})^{\otimes k_2} } R\Gamma_{\an}(C, p_1^*\omega_E^\kappa) \ar[r, "\Tr_{p_1}"]& R\Gamma_{\an}(X,\omega_E^{\kappa}).
\end{tikzcd}
\]
We define the $U^{t}_{p,\kappa}$ operator by shifting the roles of $p_1$ and $p_2$, and by composing with the map $( \pi^{\vee,*})^{\otimes k_1}\otimes (\pi_*)^{\otimes k_2}$. 
\end{definition}

\begin{remark}
\label{RemarkUpnaive}
The $U_{p,\kappa}$ above is equal to the operator $p^{-k_1}U_{p,k_1-k_2}^{\naive}$ of \cite{BoxerPilloniHigher2020}. Indeed,  $(\pi^{\vee,*,-1})^{\otimes k_1}= p^{-k_1}( \pi^*)^{\otimes k_1}$ and $(\pi_*^{-1})^{\otimes k_2}= (\pi^*)^{\otimes - k_2}$.   In other words,  $U_{p,k}^{\naive}=U_{p, (0,-k)}$.
\end{remark}

Let us explain the normalization of the Hecke operator $U_{p,\kappa}$ of the  previous definition, it turns out that it is induced from the $U_{p,\kappa}$-correspondence of the sheaf $\s{L}(\kappa)$ over $\Fl$, see Definition \ref{DefUpSheavesoverFlag} (1).  Recall that the Hodge-Tate exact sequence 
\[
0 \to \omega^{-1}_E\otimes \widehat{\s{O}}_X(1) \to T_p E \otimes \widehat{\s{O}} \to \omega_E \otimes \widehat{\s{O}} \to 0
\]
 is the pullback by $\pi_{\HT}$ of the exact sequence of $\GL_2$-equivariant vector bundles 
 \[
 0 \to \s{L}(0,1)\to \St \otimes \s{O}_{\Fl} \to \s{L}(1,0)\to 0.  
 \]
By the proof of the  Lemma \ref{LemmaDynamicsUp}, the $U_p$-correspondence of $X$ at level $X_{\infty}$ commutes with the $U_p$-correspondence of $\Fl$ defined in \eqref{eqCorrespondenceFlag}. Then, the natural $U_p$-correspondence of $\pi_{\HT}^* \St \otimes \s{O}_{\Fl}= T_pE \otimes \s{O}_{X_{\infty}} $ induced by the $\GL_2(\bb{Q}_p)$-equivariant structure of $\St$ is compatible with the natural $U_p$-correspondence of $\pi_{\HT}^*(\n{L}(0,1))= \omega_E^{-1}\otimes \s{O}_{X_{\infty}}(1)$ and $\pi_{\HT}^*(\n{L}(1,0))= \omega_E\otimes \s{O}_{X_\infty}$.  But the correspondence of $T_pE$ is just the natural isogeny $\pi: p_1^* (T_pE) \to p_2^*( T_pE)$ and we have a commutative diagram 
\[
\begin{tikzcd}
 0 \ar[r]& p_1^*(\omega^{-1}_E) \otimes \s{O}_{X_\infty}(1)\ar[r] \ar[d,"\pi_*"] &  p_1^*(T_p E) \otimes \s{O}_{X_\infty} \ar[r] \ar[d,"\pi"]& p_1^*(\omega_E) \otimes \s{O}_{X_\infty} \ar[r] \ar[d,"\pi^{\vee,*}"] & 0 \\
  0  \ar[r] & p_2^*(\omega^{-1}_E) \otimes \s{O}_{X_\infty}(1) \ar[r] &  p_2^*(T_p E) \otimes \s{O}_{X_\infty} \ar[r ]& p_2^*(\omega_E) \otimes \s{O}_{X_\infty} \ar[r] & 0.
\end{tikzcd}
\] Then the natural $U_p$ correspondences of $\omega_E\otimes \s{O}_{X_\infty}$ and $\omega_E^{-1}\otimes \s{O}_{X_\infty}(1)$, defined by $\s{L}(1,0)$ and $\s{L}(0,1)$ respectively, are given by the maps 
\[
\pi^{\vee,*}: p_1^* \omega_E \to p_2^* \omega_E \mbox{ and } \pi_*: p_1^* \omega_E^{-1} \otimes \s{O}_{X_\infty}(1) \to p_2^* \omega_E^{-1} \otimes \s{O}_{X_\infty}(1).
\]

With the previous explanation in mind we can now define the $U_p$-correspondence for the overconvergent modular forms.  We refer to \cite[Proposition 5.8]{BoxerPilloniHigher2020} for the details.

\begin{definition}
\label{DefUpOverconvergentModular}
Consider the $U_p$-correspondence \eqref{eqcorrespondance} and its restriction to the overconvergent neighbourhoods of $\overline{X}_w^{\ord}$ (see Lemma \ref{LemmaDynamicsUp}): 
\[
\begin{tikzcd}[column sep = 0.1 cm]
 & p_2^{-1}(X_{w_0}(\epsilon+1))  \ar[rd, "p_2"] \ar[ld, "p_1"'] &        &  & p_1^{-1}(X_{1}(\epsilon-1))  \ar[rd,"p_2"] \ar[ld,  "p_1"']&  \\ 
 X_{w_0}(\epsilon+1) &  & X_{w_0}(\epsilon)    &   X_{1}(\epsilon-1)  & &  X_{1}(\epsilon).   
\end{tikzcd}
\]
\begin{enumerate}

\item We define the maps 
\begin{eqnarray*}
U_p:p_2^* \omega_E^{\chi} \to p_1^* \omega_E^{\chi}   \mbox{ and }  
U_p^t:p_1^* \omega_E^{\chi} \to p_2^* \omega_E^{\chi} 
\end{eqnarray*}
to be the unique maps whose base change  to $\widehat{\s{O}}_X$-modules  coincide with the pullback by $\pi_{\HT}$ of the $U_p$-correspondence of Definition \ref{DefUpSheavesoverFlag} (2).  

\item The $U_p$ operator on $R\Gamma_1(X, \omega_E^{\chi})_{\epsilon}$ and $R\Gamma_{w_0,c}(X, \omega_E^{\chi})_{\epsilon}$ is the one induced by the map $U_p$  above.

\item  The $U_p^t$ operator on $R\Gamma_{w_0}(X, \omega_E^{\chi})_{\epsilon}$ and $R\Gamma_{1,c}(X, \omega_E^{\chi})_{\epsilon}$ is the one induced by the map $U_p^t$  above.

\end{enumerate}
\end{definition} 
 
In order to state the classicality result we need to normalize the $U_p$-operators.

\begin{definition}
\label{DefiGoodNormalizations}
Let $(R,R^+)$ be a uniform Tate $\bb{Q}_p$-algebra and  $\chi: T\rightarrow R^{+,\times}$ a $\delta$-analytic character. Let $\kappa\in X^*(\bbf{T})$.  We define  normalizations of $U_p$ and $U_p^t$ 
\begin{align*}
U_p^{good}= \begin{cases}
\frac{1}{p} U_{p} & \mbox{ over } R\Gamma_1(X,\omega_E^{\chi})_{\epsilon}, \\
U_{p} & \mbox{ over } R\Gamma_{w_0,c}(X,\omega_E^{\chi})_{\epsilon} \\
 p^{-\min\{ 1-k_1,-k_2\}} U_{p,\kappa}&  \mbox{ over } R\Gamma_{\an}(X, \omega^k)
\end{cases},  \\  
U_p^{t,good}= \begin{cases}
\frac{1}{p} U_{p}^t & \mbox{  over } R\Gamma_{w_0}(X,\omega_E^{\chi})_{\epsilon}, \\ 
U_{p}^t & \mbox{ over } R\Gamma_{1,c}(X,\omega_E^{\chi})_{\epsilon} \\
 p^{-\min\{1-k_1,-k_2\}}U_{p,\kappa}^{t} & \mbox{ over } R\Gamma_{\an}(X, \omega^\kappa_E).
\end{cases}
\end{align*}
\end{definition}

We shall need the following classicality theorem for overconvergent cohomologies. 

\begin{theo}[{\cite[Theorem 5.13]{BoxerPilloniHigher2020}}]
\label{TheoClassicityModforms}
Let $\kappa=(k_1,k_2)\in X^*(\bbf{T})$ be an algebraic weight.   
\begin{enumerate}

\item The $U_p^{good}$-operator has slopes $\geq 0$ on $H^0_1(X,\omega_E^{\kappa})_{\epsilon}$ and  $H^1_{w_0,c}(X,\omega_E^{\kappa})_{\epsilon}$.

\item The $U_p^{t,good}$ operator has slopes $\geq 0$ on $H^0_{w_0}(X,\omega_E^\kappa)_{\epsilon}$  and $H^1_{1,c}(X,\omega_E^\kappa)_{\epsilon}$.
\end{enumerate}
Futhermore,  we have isomorphisms of small slope  cohomologies 
\begin{eqnarray*}
H^0_1(X,\omega_E^\kappa)_{\epsilon}^{U_p^{good}<k_1-k_2-1}  = H^0_{\an}(X, \omega_E^{\kappa})^{U_p^{good}< k_1 -k_2-1}, \\  H^1_{w_0,c}(X,\omega_E^\kappa)_{\epsilon}^{U_p^{good}<1+k_2-k_1} =H^1_{\an}(X, \omega_E^{\kappa})^{U_p^{good}<1+k_2-k_1}, \\
H^0_{w_0}(X,\omega_E^\kappa)_{\epsilon}^{U_p^{t,good}<k_1-k_2-1}=  H^0_{\an}(X, \omega_E^{\kappa})^{U_p^{t,good}<k_1-k_2-1},  \\  H^1_{1,c}(X,\omega_E^\kappa)_{\epsilon}^{U_p^{t,good}<1+k_2-k_1} = H^1_{\an}(X, \omega_E^{\kappa})^{U_p^{t,good}<1+k_2-k_1} .
\end{eqnarray*}
\end{theo}

\subsection{Overconvergent modular symbols}
\label{SubsectionOVmodSym}

Let $(R,R^+)$ be a uniform affinoid Tate algebra over $\bb{Q}_p$, and $\chi: T \to R^{+,\times}$ a $\delta$-analytic character.  Let $A_{\chi}^{\delta}$ and $D_{\chi}^{\delta}=\Hom^0_{R}(A^{\delta}_{\chi}, R)$ be the principal series and distributions of Definition \ref{DefPrincipalSeriesDist}. These are topological $\bb{Q}_p$-vector spaces, $A^{\delta}_{\chi}$ being a Banach space and $D_{\chi}^{\delta}$ endowed with the weak topology. Consider the constant $\Iw_n$-equivariant  quasi-coherent sheaves $A_{\chi}^{\delta} \widehat{\otimes} \s{O}_{\Fl}$ and $D_{\chi}^{\delta}\widehat{\otimes} \s{O}_{\Fl}$ over $\Fl$, where the completed tensor products are as in Definition \ref{DefinitionTensor}.  Their pullbacks by $\pi_{\HT}$ are identified with pro-Kummer-\'etale sheaves $A_{\chi, \et}^{\delta}\widehat{\otimes} \widehat{\s{O}}_X$ and $D_{\chi,\et}^{\delta} \widehat{\otimes} \widehat{\s{O}}_X$, where $A_{\chi, \et}^{\delta}$ and $D_{\chi,\et}^{\delta}$ are the 
 sheaves over $X_{\proket}$ obtained by descent from the $\Iw_n$-equivariant constant sheaves over $X_{\infty}$ induced by the corresponding topological $\Iw_n$-modules. 

Before introducing the spaces of overconvergent modular symbols, let us define the pro\'etale cohomology with compact support. 

\begin{definition}
\label{DefinitionCohoCompactSupports}
Let $\s{F}$ be a pro\'etale sheaf over $Y=Y_0(p^n)$, and let $j_{\proket}: Y_{\proet}\to X_{\proket}$ be the natural morphism of sites. The pro\'etale cohomology with compact support of $\s{F}$ is the complex 
\[
R\Gamma_{\proet,c}(Y_{\bb{C}_p}, \s{F})= R\Gamma_{\proket}(X_{\bb{C}_p}, j_{\proket,!}\s{F}).
\]
\end{definition}

\begin{remark}

Let $j_{\proet}: Y_{\proet}\to X_{\proet}$ be the natural morphism of sites and  let $\bb{L}$ be a pro\'etale  $\bb{Z}_p$-local system over $Y$. The pro\'etale cohomology with compact support of $\bb{L}$ is usually defined as  $R\Gamma_{\proet, c}(Y_{\bb{C}_p},  \bb{L}):= R\Gamma_{\proet}(X_{\bb{C}_p},  j_{\proet,!} \bb{L})$. On the other hand,  Lemma 4.4.27 of  \cite{DiaoLogarithmic2019} implies that this cohomology  can be computed in the pro-Kummer-\'etale site, i.e.  that we have a quasi-isomorphism 
\[
R\Gamma_{\proet}(X_{\bb{C}_p}, j_{\proet, !}\bb{L})= R\Gamma_{\proket}(X_{\bb{C}_p},  j_{\proket,!} \bb{L}).  
\]
In other words, if $\s{F}=\bb{L}$ is a  pro\'etale local system over $Y$. The cohomology with compact support of Definition \ref{DefinitionCohoCompactSupports} coincides with the classical one. 
\end{remark}

\begin{definition}
We define the overconvergent modular symbols as the cohomology complexes 
\[
R\Gamma_{\proet}(Y_{\bb{C}_p}, A^{\delta}_{\chi,\et}) \mbox{ and  } R\Gamma_{\proet}(Y_{\bb{C}_p}, D^{\delta}_{\chi,\et} ). 
\]
We also define the overconvergent modular symbols  with compact support in the obvious way. 
\end{definition}

\begin{remark}
By purity on $p$-torsion local systems \cite[Theorem 6.4.1]{DiaoLogarithmic2019} and  the devisage of $A_{\chi}^{\delta,+}$  and $D_{\chi}^{\delta,+}$ of Corollary \ref{corDevisageSheavesAetD},   one has quasi-isomorphisms 
\begin{eqnarray*}
R\Gamma_{\proet}(Y_{\bb{C}_p},  A^{\delta}_{\chi,\et})  & = &  R\Gamma_{\proket}(X_{\bb{C}_p},  A^{\delta}_{\chi,\et})  \\
R\Gamma_{\proet}(Y_{\bb{C}_p},  D^{\delta}_{\chi,\et})   & = &  R\Gamma_{\proket}(X_{\bb{C}_p},  D^{\delta}_{\chi, \et}).
\end{eqnarray*}
\end{remark}

The primitive comparison theorem also applies for the modular symbols as follows: 

\begin{lem}
Let $\iota:D \subset X$ be the cusp divisor endowed with the log-structure induced by $X$. Let $\iota: D_{\proket}\to X_{\proket}$ be the natural morphism, and let $\widehat{\s{I}}^{+}= \ker( \widehat{\s{O}}_{X}^{+} \to \iota_* \widehat{\s{O}}^{+}_{D})$ be the ideal defining the cusps. Then $(j_{\proket,!}A^{\delta,+}_{\chi,\et})\widehat{\otimes} \widehat{\s{O}}^+_X= A^{\delta,+}_{\chi,\et}\widehat{\otimes }\widehat{\s{I}}^+$.   Furthermore, we have almost quasi-isomorphisms 
\begin{equation}
\label{eqPrimitiveComparisonAetD}
\begin{aligned}
R\Gamma_{\proket}(X_{\bb{C}_p}, A^{\delta,+}_{\chi,\et})\widehat{\otimes} \n{O}_{\bb{C}_p} &=^{a}& R\Gamma_{\proket}(X_{\bb{C}_p}, A^{\delta,+}_{\chi,\et} \widehat{\otimes}  \widehat{\s{O}}^+_X) \\ 
R\Gamma_{\proket}(X_{\bb{C}_p}, j_{\proket,!} A^{\delta,+}_{\chi,\et})\widehat{\otimes} \n{O}_{\bb{C}_p} &=^{a}& R\Gamma_{\proket}(X_{\bb{C}_p}, A^{\delta,+}_{\chi,\et} \widehat{\otimes}  \widehat{\s{I}}^+).
\end{aligned}
\end{equation}
 The analogous statement holds for the sheaf $D_{\chi,\et}^{\delta}$. 
\end{lem}
\begin{proof}
We only give the proof for $A_{\chi,\et}^{\delta,+}$ the other being identical. By Corollary \ref{corDevisageSheavesAetD} we can write $A_{\chi,\et}^{\delta,+}= R\varprojlim_{i}( \varinjlim_{j} \bb{L}_{i,j})$, where $\bb{L}_{i,j}$ are finite  Kummer-\'etale local systems over $X$. Then, by the primitive comparison theorem for finite local systems one has that 
\begin{eqnarray*}
R\Gamma_{\proket}(X_{\bb{C}_p}, A^{\delta,+}_{\chi,\et})\widehat{\otimes} \n{O}_{\bb{C}_p} & = & R\varprojlim_{i} \hocolim_{j} (R\Gamma_{\ket}(X_{\bb{C}_p}, \bb{L}_{i,j}) \otimes    \n{O}_{\bb{C}_p})  \\ 
& =^{a} &   R\varprojlim_{i} \hocolim_{j} (R\Gamma_{\ket}(X_{\bb{C}_p}, \bb{L}_{i,j}\otimes \s{O}^+_X))   \\ 
& = & R\Gamma_{\proket}(X_{\bb{C}_p}, A^{\delta,+}_{\chi,\et} \widehat{\otimes}  \widehat{\s{O}}^+_X). 
\end{eqnarray*} 

To prove the case of cohomology with compact support we need the following observation. Consider the exact sequence  of Kummer-\'etale sheaves over $X$
\begin{equation*}
0 \to j_{\ket,!}\bb{Z}/p^s \to \bb{Z}/p^s \to \iota_* \bb{Z}/p^s \to 0.
\end{equation*}
Tensoring with $\s{O}^+_X$ and taking projective limits on $s$, one obtains the short exact sequence 
\[
0\to \widehat{\s{I}}^+ \to  \widehat{\s{O}}_X^+ \to \iota_* \widehat{\s{O}}_D \to 0.
\]
Therefore, tensoring with $A^{\delta,+}_{\chi,\et}$, one gets a short exact sequence 
\begin{equation}
\label{eqSuiteExactCompact}
0 \to  A^{\delta,+}_{\chi,\et} \widehat{\otimes} \widehat{\s{I}}^+ \to  A^{\delta,+}_{\chi,\et} \widehat{\otimes} \widehat{\s{O}}_{X}^+ \to A^{\delta,+}_{\chi, \et}\widehat{\otimes }  \iota_* \widehat{\s{O}}_D^+ \to 0. 
\end{equation}
It is easy to see that \eqref{eqSuiteExactCompact} is the (completed) $\widehat{\s{O}}^+_X$-scalar extension of the short exact sequence
\[
0 \to  j_{\proket,!} A^{\delta,+}_{\chi} \to A_{\chi,\et}^{\delta,+}\to  \iota_* \iota^* A_{\chi,\et}^{\delta,+}\to 0.
\]
The previous shows that $j_{\proket,!}A_{\chi,\et}^{\delta,+}\widehat{\otimes} \widehat{\s{O}}_X^+= A_{\chi,\et}^{\delta,+}\widehat{\otimes }\widehat{\s{I}}^+$. Moreover,   the second almost equality of \eqref{eqPrimitiveComparisonAetD} holds after taking pro-Kummer-\'etale cohomology of the triangle  \eqref{eqSuiteExactCompact}, by applying the primitive comparison theorem for both $X$ and $D$, see \cite[Theorem 2.2.1]{LanLiuZhuRhamComparison2019}.
\end{proof}

Next, we define the $U_p$-operators for overconvergent modular symbols.  They are obtained by pulling back  the maps  of equation \eqref{eqUpForAetD}. 

\begin{definition}
\label{DefUpModularSymbols}
Consider the $U_p$-correspondence $C$ of $X$. The $U^t_p$ and $U_p$-correspondence of $A^{\delta}_{\chi,\et}$ and $D_{\chi,\et}^{\delta}$  are  the morphisms 
\[
U_p^t: p_1^*( A_{\chi,\et}^{\delta})  \to U_p:  p_2^* (A^{\delta}_{\chi,\et}) \mbox{ and } U_p:p_2^*( D_{\chi,\et}^{\delta}) \to p_1^* (D_{\chi,\et}^{\delta})
\]
defined by the pullback of \eqref{eqUpForAetD} by $\pi_{\HT}$. 
\end{definition}

We shall need the following classicality result, see \cite[Theorem 6.4.1]{AshStevensDeformations2008} and \cite[Theorem 3.16]{AISOvShimura2015}.

\begin{theo}
\label{TheoClassicityModularSymbols}
Let $\kappa=(k_1,k_2)\in X^*(\bbf{T})^+$ be a dominant weight.   The maps $D^{\delta}_{\kappa} \rightarrow V_{-w_0(\kappa)}$ and $V_\kappa \rightarrow A^\delta_\kappa$ induce  isomorphisms of the $(<k_1-k_2+1)$-slope part  for the action of the (normalized) $U_p$-operators  
\begin{align*}
H^1_{\proket}(X_{\bb{C}_p}, D^\delta_{\kappa,\et}) ^{U_p<k_1-k_2+1} \xrightarrow{\sim} H^1_{\proket}(X_{\bb{C}_p}, V_{-w_0(\kappa), \et})^{U_p<k_1-k_2+1} \\ H^1_{\proket}(X_{\bb{C}_p},  V_{\kappa,\et})^{U_p^t<k_1-k_2+1}\xrightarrow{\sim} H^1_{\proket}(X_{\bb{C}_p}, A^\delta_{\kappa,\et})^{U_p^t<k_1-k_2+1}.
\end{align*} 
A similar result holds for the cohomology with compact support. 
\end{theo}
\begin{remark}
In the references Theorem \ref{TheoClassicityModularSymbols} is proved only for the cohomology  of distributions. However, the same strategy works for the principal series and the cohomology with compact support.  Particularly,  the bounds of the classicality theorem are motivated by   \cite[Theorem 3.11.1]{AshStevensDeformations2008}.
\end{remark}

\subsection{Overconvergent Hodge-Tate  maps}

\label{SubsectionDlog}

We end this section with  the definition of overconvergent $\HT$-maps   interpolating the morphisms $\HT^k:\Sym^k T_pE \otimes \widehat{\s{O}}_X \rightarrow \omega_E^{k}\otimes_{\s{O}_X}\widehat{\s{O}}_X$ and $\HT^{k,\vee}: \omega_E^{-k}\otimes_{\s{O}_X}\widehat{\s{O}}_X(k)\rightarrow \Sym^k T_pE \otimes \widehat{\s{O}}_X$. 

\begin{definition}
\label{DefinitionOVHTModularCurve}
Let $\epsilon\geq \delta \geq n$,  $(R,R^+)$ a uniform affinoid Tate $\bb{Q}_p$-algebra and  $\chi=(\chi_1,\chi_2): T=\bbf{T}(\bb{Z}_p)\rightarrow R^{+,\times} $ a $\delta$-analytic character.  
\begin{enumerate}
\item We define the   map of pro-Kummer-\'etale sheaves  over $X_{1}(\epsilon)$
\[
\HT^{A,\vee}_{-w_0(\chi)}: \omega_E^{w_0(\chi)}\otimes_{R\widehat{\otimes} \s{O}_X} \widehat{\s{O}}_X(\chi_1) \rightarrow  A^{\delta}_{\chi,\et}\widehat{\otimes}\widehat{\s{O}}_X 
\]
as  the pullback of the highest weight vector map $\Psi^{A,\vee}_{-w_0(\chi)}:\s{L}(w_0(\chi))\rightarrow A^{\delta}_{\chi}\widehat{\otimes} \s{O}_{\Fl}$ over $U_1(\epsilon)\Iw_n\subset \Fl$ by  $\pi_{\HT}$,  cf.  Proposition \ref{PropOVBGGmaps}.    We let  $ \HT_{-w_0(\chi)}^{D}:  D^{\delta}_{\chi,\et}\widehat{\otimes}\widehat{\s{O}}_X \rightarrow \omega_E^{-w_0(\chi)}\otimes \widehat{\s{O}}_X(-\chi_1)$  be  the dual of $\HT^{A,\vee}_{-w_0(\chi)}$. 

\item We define the   map of pro-Kummer-\'etale sheaves  over $X_{w_0}(\epsilon)$
\[
\HT^{A}_{\chi}: A^{\delta}_{\chi,\et}\widehat{\otimes}\widehat{\s{O}}_X \rightarrow \omega_E^{\chi}\otimes  \widehat{\s{O}}_X(\chi_2)
\]
as the pullback by $\pi_{\HT}$ of the lowest weight vector map $ A^{\delta}_{\chi}\widehat{\otimes} \s{O}_{\Fl} \rightarrow \s{L}(\chi)$ over $U_{w_0}(\epsilon)\Iw_n\subset \Fl$.   We let  $\HT^{D, \vee}_{\chi}:  \omega_E^{-\chi}\otimes \widehat{\s{O}}_X (-\chi_2) \rightarrow   D^{\delta}_{\chi,\et}\widehat{\otimes}\widehat{\s{O}}_X$  be  the dual of $\HT^{A}_{\chi}$. 
\end{enumerate}
\end{definition}

\begin{lem}
\label{LemmaCompatibilityUp}
The overconvergent Hodge-Tate maps of Definition \ref{DefinitionOVHTModularCurve} are compatible with the $U_p$-correspondences of  Definitions \ref{DefUpOverconvergentModular} and \ref{DefUpModularSymbols}. Moreover,  they are compatible with the normalized $U_p$-correspondances of the sheaves $V_{\kappa, \et}$ for $\kappa\in X^*(\bbf{T})^+$, see Remark \ref{RemakNormalizationUp}.  
\end{lem}
\begin{proof}
The lemma is an immediate consequence of the definitions and  Proposition \ref{CorAcompatibleV}, see Remarks \ref{RemarkCorrQuotientFlag} and \ref{RemarkCorresFiniteLevel}. 
\end{proof}


\section{$p$-adic Eichler-Shimura maps}
\label{SectionOverconvergentES}

Throughout this section we fix $K^p\subset \GL_2(\bb{A}_{\bb{Q}}^{\infty,p})$ a neat compact open subgroup, and given  $K_p\subset \GL_2(\bb{Q}_p)$ we let $Y=Y_{K_p}$ and  $X=X_{K_p}$ denote the affine and compact modular curves of level $K_p$ respectively. We let  $D= X\backslash Y$ be the cusp divisor.   Let   $f: E^{sm}  \rightarrow X$ be  the semi-abelian scheme extending the universal elliptic curve over $Y$,  and $\overline{E}$ its relative compactification to a log smooth adic space over $X$, cf. \cite{DeligneRappLesSchemasCourbes}.   We denote by  $\DR_X(\overline{E})$ the relative log de Rham complex of $\overline{E}$ over $X$, and $\s{H}^1_{\dR}:= R^1 f_{\an,*}(\DR_X(\overline{E}))$ the first relative de Rham cohomology group.   The sheaf $\s{H}^{1}_{\dR}$ is endowed with a log connection 
\[
\nabla: \s{H}^1_{\dR}\rightarrow \s{H}^1_{\dR}\otimes_{\s{O}_X} \Omega_X^1(\log)
\] 
and a Hodge filtration  $0\to \omega_E \to \s{H}^1_{dR} \to \omega_E^{-1}\to 0$ with $\Fil^0 \s{H}^1_{\dR}= \s{H}^1_{\dR}$,  $\Fil^1 \s{H}^1_{\dR}=\omega_E$ and $\Fil^2\s{H}^1_{\dR}=0$,  satisfying Griffiths transversality. 
This last section is dedicated to the construction of the Eichler Shimura  decomposition for the \'etale cohomology of the modular curves.   We first provide  a new proof of Faltings's  Eichler-Shimura (ES) decomposition of the cohomology of the local systems $V_{\kappa, \et}$ (cf. \cite{FaltingsHodgeModular1987}). Our method uses   the Hodge-Tate period map  and the dual BGG resolution of Proposition \ref{PropBGGclassic}.   Next, we use the overconvergent Hodge-Tate  maps of  Definition \ref{DefinitionOVHTModularCurve} to define overconvergent Eichler-Shimura   maps.  We shall recover the results of   \cite{AISOvShimura2015} as well as a new map from the $H^1$-cohomology with compact support of overconvergent modular forms to overconvergent modular symbols.  Finally, we show that the overconvergent ES maps are  compatible with the Poincar\'e and Serre duality pairings, and that,  for small slope,  we have a perfect pairing.

\subsection{A pro\'etale Eichler Shimura  decomposition}

\label{SubsectionClassicalES}

Let $\kappa=(k_1,k_2)\in X^*(\bbf{T})^+$ be a dominant weight,  $V_{\kappa}$ the irreducible representation of highest weight $\kappa$ and  $V_{\kappa,\et}$ the pro-Kummer-\'etale  local system over $X$ defined by $V_{\kappa}$.    Let $\alpha=(1,-1)\in X^*(\bbf{T})$.  Let us recall a theorem of Faltings 

\begin{theo}[{\cite[Theorem 6]{FaltingsHodgeModular1987}}]
\label{TheoESDecompositionClassical}
 There are    Hecke and Galois equivariant isomorphisms 
\begin{eqnarray*}
H^1_{\proet}(Y_{\bb{C}_p}, V_{\kappa,\et})\otimes_{\bb{Z}_p} \bb{C}_p & = & H^1_{\an}(X_{\bb{C}_p},  \omega_E^{w_0(\kappa)})(k_1) \oplus  H^0_{\an}(X_{\bb{C}_p},  \omega_E^{\kappa+\alpha})(k_2-1) \\ 
H^1_{\proet,c}(Y_{\bb{C}_p}, V_{\kappa,\et})\otimes_{\bb{Z}_p} \bb{C}_p & = &H^{1}_{\an}( X_{\bb{C}_p}, \omega_E^{w_0(\kappa)}(-D))(k_1) \oplus  H^0_{\an}(X_{\bb{C}_p},  \omega_E^{\kappa+\alpha}(-D))(k_2-1).
\end{eqnarray*}
\end{theo}

 Let $\bb{B}_{\dR}^+$ be the de Rham period sheaf of $X_{\proket}$, $\theta: \bb{B}_{\dR}^+\to \widehat{\s{O}}_X$ the Fontaine map and $\xi\in \ker \theta$ a local generator of the kernel, we set $\bb{B}_{\dR}:= \bb{B}^+_{\dR}[\frac{1}{\xi}]$. Let $\OBdr^+$ be the geometric de Rham period sheaf and $\OBdr=\OBdr^+[\frac{1}{\xi}]$, we denote  by $\OC$ the sheaf $\gr^0(\OBdr)$. We refer to \cite{ScholzeHodgeTheory2013} and  \cite{DiaoLogarithmicHilbert2018} for the definition of the period sheaves. 
 
    The main ingredient of our  proof of Theorem  \ref{TheoESDecompositionClassical}   is an explicit relation between the Faltings extension $\gr^1 \OBdr^+$ and the Tate module  $T_pE$. This arises naturally in the study of  pullbacks of $\GL_2$-equivariant vector bundles of $\Fl$ by $\pi_{\HT}$. Recall that $\FL= \bbf{B} \backslash \GL_2$, so that we have an equivalence of categories  between $\GL_2$-equivariant vector bundles over $\Fl$ and algebraic representations of $\bbf{B}$. Let $\s{O}(\bbf{B})$ be the ring of algebraic functions of $\bbf{B}$ endowed with the right regular action, note that any finite representation of $\bbf{B}$ occurs in $\s{O}(\bbf{B})$.  Writing $\bbf{B}= \bbf{T} \ltimes \bbf{N}$ as a product of the diagonal torus and its unipotent radical, one has that $\s{O}(\bbf{B})=\s{O}(\bbf{T}) \otimes \s{O}(\bbf{N})$.  The action of $\bbf{B}$ on $\s{O}(\bbf{T})$ factors through $\bbf{T}$, so that this ring can be decomposed in terms of characters of the torus. By Proposition \ref{PropPullbackModtorsor} we already know what  the pullback by $\pi_{\HT}$ of the  quasi-coherent sheaf associated to $\s{O}(\bbf{T})$ is; it admits an explicit description in terms of modular sheaves. On the other hand, the action of $\bbf{B}$ on $\s{O}(\bbf{N})$ is determined by the right action $(n,b)\mapsto t_{b}^{-1}nt_b n_b$ where $n\in \bbf{N}$ and $b=(t_b,n_b)\in \bbf{B}=\bbf{T} \ltimes \bbf{N}$. Let $\underline{\s{O}}(\bbf{N})$ be the $\GL_2$-equivariant quasi-coherent sheaf  over $\Fl$ attached to $\s{O}(\bbf{N})$.

\begin{theo}[{\cite[Theorem  5]{FaltingsHodgeModular1987} and \cite[Theorem 4.2.2]{pan2020locally}}]
\label{PropFaltingsExtension}

There is a natural isomorphism of pro-Kummer-\'etale sheaves over $X$
\[
\pi_{\HT}^*(\underline{\s{O}}(\bbf{N}))= \OC.
\]
Furthermore, let $\s{O}(\bbf{N})^{\leq 1}\subset \s{O}(\bbf{N})$ be the  subrepresentation consisting on polynomials of degree $\leq 1$. We have an isomorphism as $\bbf{B}$-representations $\s{O}(\bbf{N})^{\leq 1} = \St \otimes \bb{Q}_p(-1,0)$, in particular   $\pi_{\HT}^*(\underline{\s{O}}(\bbf{N})^{\leq 1})= T_pE \otimes \widehat{\s{O}}_X(-1)\otimes \omega_E$. Moreover,  there is an isomorphism of extensions 
\[
\begin{tikzcd}
0 \ar[r] & \widehat{\s{O}}_X(1) \ar[r, "\HT^\vee"]  \ar[d, "\id"]& T_p E \otimes \widehat{\s{O}}_E \otimes \omega_E  \ar[r, "\HT"]  \ar[d, "\alpha"] & \omega_E^2 \otimes  \widehat{\s{O}}_X \ar[r] \ar[d, "-KS"]&  0  \\ 
0 \ar[r] & \widehat{\s{O}}_X(1) \ar[r] &  \gr^1 \OBdr^+ \ar[r] & \Omega^1_X(\log) \otimes \widehat{\s{O}}_X \ar[r]& 0 
\end{tikzcd}
\]
where $KS$ is the Kodaira-Spencer isomorphism.

\proof

It is enough to show the second statement, namely, if $\pi_{\HT}^*( \underline{\s{O}}(\bbf{N})^{\leq 1}) =T_pE\otimes \widehat{\s{O}}(-1)\otimes \omega_E= \gr^1 \OBdr^+(-1)$, one has 
\begin{eqnarray*}
\OC =\varinjlim_k \Sym^k( \gr^1\OBdr^+ (-1) ) = \varinjlim_k  \pi_{\HT}^*( \Sym^k \underline{\s{O}}(  \bbf{N})^{\leq 1}) = \pi_{\HT}^*(\underline{\s{O}}(\bbf{N})). 
\end{eqnarray*}

 Let $\s{F}$ be a sheaf   endowed with an integral log connection $\nabla$, we denote by $\DR(\s{F}, \nabla)$ the log de Rham complex of $\s{F}$.  Let $f:\overline{E}\to X$ be the compactification of the elliptic curve as a log smooth adic space over $X$.  We have a quasi-isomorphism of complexes over  $\overline{E}_{\proket}$ 
\[
T_p\bb{G}_m\otimes_{\widehat{\bb{Z}}_p}\bb{B}_{\dR,\overline{E}}\simeq  T_p\bb{G}_m\otimes_{\widehat{\bb{Z}}_p} \DR( \s{O}\!\bb{B}_{\dR,\log, \overline{E}},d  )= \DR( \s{O}\!\bb{B}_{\dR,\log, \overline{E}},d)(1).
\]
Taking $R^1 f_{\proket,*}$  one obtains by \cite[Theorem 3.2.7 (5)]{DiaoLogarithmicHilbert2018}  or \cite[Theorem 8.8]{ScholzeHodgeTheory2013}  
\begin{equation}
\label{eqDrahmcomparison}
T_pE\otimes \bb{B}_{\dR, X} \simeq T_p E \otimes_{\widehat{\bb{Z}}_p} \DR( \s{O}\!\bb{B}_{\dR,\log, X},d) \cong \DR(\s{H}^1_{dR}\otimes  \s{O}\!\bb{B}_{\dR,\log, X},  \nabla) (1) .
\end{equation}
Let $\bb{M}:= T_pE(-1)\otimes \bb{B}_{\dR,X}^+ =(T_pE(-1) \otimes \s{O}\! \bb{B}_{\dR,\log,X}^+)^{\nabla=0}$ and $\bb{M}_0 = ( \s{H}_{\dR}^1\otimes \s{O}\! \bb{B}_{\dR,\log, X}^+)^{\nabla=0}$.  Both $\bb{M}_0$ and $\bb{M}$ are $\bb{B}_{\dR,X}^+$-lattices of  $T_pE(-1)\otimes \bb{B}_{\dR,X}$.    The  Hodge Filtration of $\s{H}^{1}_{\dR}$ is concentrated in degrees $0$ and $1$, and equal to 
\[
0 \rightarrow \omega_E \rightarrow \s{H}^1_{\dR} \rightarrow  \omega_E^{-1}\rightarrow 0. 
\]
This implies that $  \xi\bb{M} \subset \bb{M}_0\subset \bb{M}$, and that $(\Fil^1( \s{H}^1_{\dR} \otimes \s{O}\! \bb{B}^+_{\dR,\log,X} ))^{\nabla=0}= \xi \bb{M}$.  Then,  Proposition 7.9 of \cite{ScholzeHodgeTheory2013} implies 
\begin{gather*}
\bb{M}_0/ \xi \bb{M}  = \gr^0 \s{H}^1_{\dR}\otimes \widehat{\s{O}}_X = \omega_{E}^{-1}\otimes \widehat{\s{O}}_{X} \\
\bb{M}/\bb{M}_0 = \gr^{1} \s{H}^1_{\dR} \otimes \widehat{\s{O}}_X(-1)= \omega_E\otimes \widehat{\s{O}}_X(-1). 
\end{gather*}
In particular,  
\[
0\rightarrow \xi \bb{M}_0/ \xi^2\bb{M} \rightarrow \xi \bb{M}/ \xi^2\bb{M}  \rightarrow  \xi \bb{M}/ \xi \bb{M}_0 \rightarrow 0
\]
is just the Hodge-Tate exact sequence of $T_pE\otimes \widehat{\s{O}}_X$ (note the multiplication by $\xi$ induced by the Tate twist in (\ref{eqDrahmcomparison})), and 
\[
0 \rightarrow  \xi \bb{M}/\xi \bb{M}_0 \rightarrow \bb{M}_0/\xi \bb{M}_0 \rightarrow \bb{M}_0 /\xi \bb{M}\rightarrow 0
\]
is the Hodge exact sequence of $\s{H}_{dR}^1\otimes \widehat{\s{O}}_X$. 

Consider the map of short exact sequences 
\begin{equation}
\label{eqMapShortExactSeq}
\begin{tikzcd}
0 \ar[r]& \bb{M} \ar[r] & \bb{M}\otimes \s{O}\!\bb{B}_{\dR,\log,X}^+ \ar[r, "d"]& \bb{M}\otimes \s{O}\!\bb{B}_{\dR,\log,X}^+ \otimes \Omega_X^1(\log) \ar[r]& 0 \\ 
0 \ar[r] & \bb{M}_0 \ar[r] \ar[u] & \bb{M}_0\otimes \s{O}\!\bb{B}_{\dR,\log,X}^+ \ar[r, "d"] \ar[u]& \bb{M}_0\otimes \s{O}\!\bb{B}_{\dR,\log,X}^+ \otimes \Omega_X^1(\log) \ar[r] \ar[u]& 0
\end{tikzcd}
\end{equation}
and let $\tilde{\theta}:  \s{O}\!\bb{B}_{\dR,\log,X}^+ \rightarrow \widehat{\s{O}}_X$ be  the Fontaine's map. 

Taking the first graded piece in the upper short exact sequence one finds 
\[
0 \rightarrow \xi\bb{M}/\xi^2\bb{M} \rightarrow \frac{\bb{M}\otimes (\ker \tilde{\theta})}{  \bb{M}\otimes (\ker \tilde{\theta})^2} \xrightarrow{\overline{\nabla}} \frac{\bb{M}\otimes \s{O}\!\bb{B}^+_{\dR,\log,X}}{\bb{M}\otimes (\ker \tilde{\theta})} \otimes \Omega^1_X(\log) \rightarrow 0.
\]
Since $\xi \bb{M}\subset \bb{M}_0$,  taking the intersection with the image of the  lower short exact sequence in (\ref{eqMapShortExactSeq}) one obtains a short exact sequence 
\begin{equation}
\label{eqShort1}
0\rightarrow \frac{\xi \bb{M}}{\xi^2\bb{M}} \rightarrow \frac{\bb{M}_0\otimes  (\ker\tilde{\theta}) + \xi \bb{M}\otimes \s{O}\!\bb{B}_{\dR,\log,X}^+ }{\bb{M}_0 \otimes (\ker \tilde{\theta})^2 + \xi \bb{M}\otimes (\ker \tilde{\theta}) } \xrightarrow{\nabla} \frac{\bb{M}_0 \otimes \s{O}\!\bb{B}_{\dR,\log,X}^+}{\bb{M}_0 \otimes (\ker \tilde{\theta})+ \xi \bb{M}\otimes \s{O}\!\bb{B}_{\dR,\log,X}^+}\otimes \Omega^1_X(\log) \rightarrow  0
\end{equation}
The right term of (\ref{eqShort1}) is equal to $\bb{M}_0/ \xi \bb{M}\otimes \Omega^1_X(\log)= \omega_E^{-1}\otimes \Omega^1_X(\log)\otimes \widehat{\s{O}}_X$.  The middle term is equal to 
\[\gr^1(\s{H}^1_{\dR}\otimes \s{O}\! \bb{B}_{\dR,\log,X}^+)=\omega_E \otimes\widehat{\s{O}}_X \oplus \omega_E^{-1}\otimes \gr^1 \s{O}\! \bb{B}^+_{\dR,\log,X}.
\]
Note that the restriction of $\overline{\nabla}$ to $\omega_E \otimes \widehat{\s{O}}_X$ is the Kodaira-Spencer map by definition.  Indeed, if $\nabla: \s{H}^1_{dR}\rightarrow \s{H}^1_{dR}\otimes \Omega^1_{X}(\log)$ is the connection, taking the first graded piece we get the map 
\[
KS : \omega_E \rightarrow \omega_E^{-1}\otimes \Omega_X^1(\log). 
\]

Therefore,   we have constructed a short exact sequence 
\[
0\rightarrow T_pE \otimes \widehat{\s{O}}_X \xrightarrow{\HT\oplus \alpha} \omega_E \otimes \widehat{\s{O}}_X \oplus \omega_E^{-1}\otimes \gr^1 \s{O}\! \bb{B}^+_{\dR,\log,X} \xrightarrow{KS\oplus \overline{\nabla}}  \omega_E^{-1}\otimes \Omega^1_{X}(\log)\otimes \widehat{\s{O}}_X \rightarrow 0. 
\]
Thus,  as $KS$ is an isomorphism so is $\alpha$, and  we have a commutative diagram
\[
\begin{tikzcd}
0\ar[r] & \omega_E^{-1}\otimes \widehat{\s{O}}_X(1) \ar[d, "\id"] \ar[r, "\HT^{\vee}"] & T_pE \otimes \widehat{\s{O}}_X  \ar[r, "\HT" ] \ar[d, "\alpha"]& \omega_E \otimes \widehat{\s{O}}_X  \ar[d, "-KS"] \ar[r]& 0 \\
0 \ar[r] & \omega_E^{-1}\otimes \widehat{\s{O}}_X(1) \ar[r]& \omega_E^{-1}\otimes \gr^1 \s{O}\!\bb{B}_{\dR,\log,X}^+ \ar[r, "\nabla"] & \omega_E^{-1}\otimes \Omega_X^1(\log)\otimes \widehat{\s{O}}_X \ar[r] & 0 
\end{tikzcd}
\]
which finishes the proof.  \endproof
\end{theo}

  \begin{remark}The previous  proposition  is the key tool  necessary  to compute the relative Sen operator for the modular curve in  Pan's locally analytic vectors,  cf.  \cite{pan2020locally}.  
  \end{remark}

We deduce the Eichler-Shimura decompositions for the local systems $V_{\kappa,\et}$.

\begin{theo}
\label{TheoBGGproEtale}
Let $\alpha=(1,-1)$. Let $\kappa=(k_1,k_2)\in X^*(\bbf{T})^+$ be a dominant weight and $\BGG(\kappa)$ the BGG complex of  Proposition \ref{PropBGGclassic}
\[
\BGG(\kappa): \;\; [0\rightarrow  V_{\kappa,\et} \rightarrow  V(\kappa)\rightarrow  V(w_0(\kappa)-\alpha)\rightarrow 0].
\]
Let $\underline{\BGG(\kappa)}$ be the $\GL_2$-equivariant complex of sheaves defined by $\BGG(\kappa)$. We have a quasi-isomorphism of complexes over $X_{\proket}$
\[
\pi_{\HT}^*(\underline{\BGG(\kappa)})=[0\rightarrow V_{\kappa}\otimes \widehat{\s{O}}_X \rightarrow \omega_E^{w_0(\kappa)} \otimes \OC(k_1) \rightarrow \omega_E^{\kappa+\alpha}\otimes \OC(k_2-1) \rightarrow 0].
\]
Moreover, let $\lambda: X_{\bb{C}_p, \proket} \to X_{\bb{C}_p, \an}$ be the projection of sites. Let $\iota: D_{\proket}\to X_{\proket}$ be the natural morphism, and $\widehat{\s{I}}= \ker(\widehat{\s{O}}_X \to  \iota_* \widehat{\s{O}}_D)$.  We have  
\begin{eqnarray*}
R\lambda_{*} (V_{\kappa,\et}\otimes \widehat{\s{O}}_{X})&= & \omega_E^{w_0(\kappa)}\otimes \bb{C}_p(k_1)[0]\oplus \omega_E^{\kappa+\alpha} \otimes \bb{C}_p(k_2-1)[-1] \\ 
R\lambda_{*} (V_{\kappa,\et}\otimes \widehat{\s{I}}_{X}) & = &  \omega_E^{w_0(\kappa)}(-D)\otimes \bb{C}(k_1)[0]\oplus \omega_E^{\kappa+\alpha}(-D) \otimes \bb{C} (k_2-1)[-1]. 
\end{eqnarray*}
Then, taking $H^1$-cohomology over $X_{\bb{C}_p,\an}$ one obtains Theorem \ref{TheoESDecompositionClassical}.   
\proof
Note that $V(\kappa)= \kappa \otimes V(0)$ as $\bbf{B}$-module, thus the first part of the theorem follows from Theorem \ref{PropFaltingsExtension} and Proposition \ref{PropPullbackModtorsor}. On the other hand, by   \cite[Lemma 3.3.2]{DiaoLogarithmicHilbert2018}  we know that $R \lambda_* \OC = \s{O}_{X_{\bb{C}_p}}$ and $R\lambda_* ( \OC\otimes \iota_* \widehat{\s{O}}_D )= \iota_* \s{O}_D$, in particular $ R\lambda_* (\OC \otimes  \widehat{\s{I}})= \s{O}(-D)$. Therefore,  
\begin{equation}
\label{eqProjectionOfSites}
\begin{gathered}
R \lambda_*(V_{\kappa,\et}\otimes \widehat{\s{O}}_X)=[\omega_E^{w_0(\kappa)}\otimes \bb{C}_p(k_1) \rightarrow  \omega_E^{\kappa+\alpha} \otimes \bb{C}_p(k_2-1)] \\ R\lambda_*(V_{\kappa,\et}\otimes \widehat{\s{I}})=[\omega_E^{w_0(\kappa)}(-D)\otimes \bb{C}_p(k_1)\to \omega_E^{\kappa+\alpha}(-D)\otimes \bb{C}_p(k_2-1)].
\end{gathered}
\end{equation}
But the arrows  of \eqref{eqProjectionOfSites} are $0$  since  the sheaf $\omega_E^{w_0(\kappa)}\otimes \bb{C}_p(k_1)$ already factors through $V_{\kappa,\et}\otimes  \widehat{\s{O}}_X$ via $\HT^{\vee}_{-w_0(\kappa)}: \omega_E^{w_0(\kappa)}\otimes \widehat{\s{O}}_X(k_1)\to V_{\kappa, \et}\otimes \widehat{\s{O}}_X$. The theorem follows. 
\endproof
\end{theo}

\begin{remark}
In \cite{LanLiuZhuRhamComparison2019},  Lan-Liu-Zhu gave another proof of the Eichler-Shimura decomposition for arbitrary Shimura varieties (also called BGG decomposition), see Theorem 6.2.3 of  \textit{loc. cit.} Their proof uses the $p$-adic Riemann-Hilbert correspondence and the BGG decomposition in terms of Verma modules, then they apply Faltings's strategy to construct a complex of $\n{D}$-modules quasi-isomorphic to the de Rham complex of the corresponding vector bundle with connection.     In our situation, we use the dual BGG decomposition and the associated $\GL_2$-equivariant sheaves over $\Fl$ instead. Our key ingredient to compute the pro\'etale cohomology of $V_{\kappa}\otimes \widehat{\s{O}}_X$  is Theorem  \ref{PropFaltingsExtension} which serves as a dictionary between vector bundles over $\Fl$ and  $\widehat{\s{O}}_X$-vector bundles over $X$. 
\end{remark}

We finish this section with the compatibility of the Eichler-Shimura decomposition with Poincar\'e and Serre duality. 

\begin{prop}[{\cite[Theorem 6.2.3]{LanLiuZhuRhamComparison2019}}]
\label{CoroPairingsClassical}
 Let $\Tr_P:  H^2_{\proet,c}(Y_{\bb{C}_p}, \bb{Q}_p(1))\to \bb{Q}_p$ and $\Tr_S: H^1_{\an}(X_{\bb{C}_p},   \Omega^1_X)\to \bb{C}_p$ be the Poincar\'e and Serre traces. Then the Poincar\'e pairing 
\[
H^1_{\proet}(Y_{\bb{C}_p},V_{\kappa,\et})(1)\times H^1_{\proet,c}(Y_{\bb{C}_p},V_{-w_0(\kappa),\et})\xrightarrow{\cup} H^2_{\proet,c}(Y_{\bb{C}_p},\bb{Q}_p(1))\xrightarrow{\Tr_P}\bb{Q}_p
\]
and the Serre pairing 
\[
H^1_{\an}(X_{\bb{C}_p},\omega^{-\kappa}_E)\times H^0_{\an}(X_{\bb{C}_p}, \omega_E^{\kappa+\alpha}(-D))\xrightarrow{KS \circ  \cup} H^1_{\an}(X_{\bb{C}_p},\Omega^1_X)\xrightarrow{\Tr_S}\bb{C}_p
\]
(resp. for  $\omega^{w_0(\kappa)}_E(-D)$ and $\omega_E^{-w_0(\kappa)}$) are compatible with the  Eichler-Shimura decomposition. 
\end{prop}

\subsection{The overconvergent Eichler-Shimura maps}

Let $n\geq 1$ be a fixed integer.  In the next two sections  we shall work with $Y=Y_0(p^n)$ and  $X=X_0(p^n)$, the modular curves of level $K^p\Iw_n$.   Let $\epsilon \geq \delta \geq n$ be rational numbers,  $(R,R^+)$ a uniform affinoid Tate $\bb{Q}_p$-algebra and $\chi=(\chi_1,\chi_2):  T\rightarrow R^{+,\times}$ a $\delta$-analytic character.  Let $w\in W=\{1,w_0\}$ be an element in the Weyl group of $\GL_2$ and $X_{w}(\epsilon)$ the $\epsilon$-neighbourhood of the $w$-ordinary locus, cf. Definition \ref{DefOpensFlag}.  Let $\omega_E^{\chi}$ be the sheaf of overconvergent modular  forms of weight $\chi$  over $X_{w}(\epsilon)$ (cf. Definition \ref{DefOvModsheaf}), and let   $A_{\chi,\et}^{\delta}$ and $D_{\chi,\et}^{\delta}$ be the the pro-Kummer-\'etale sheaves of $\delta$-analytic principal series and distributions over $X$ (cf. \S \ref{SubsectionOVmodSym}).  We can finally state the main theorem of this section, first we need a lemma:

\begin{lem}
\label{LemmaOVESMaps}
Let $\alpha=(1,-1)\in X^*(\bbf{T})$.  The overconvergent  Hodge-Tate morphisms of Definition \ref{DefinitionOVHTModularCurve} give  rise Galois and  $U_p^t$-equivariant  maps  of cohomology groups (see Definition \ref{DefiGoodNormalizations} for the good normalizations)
\begin{equation}
\label{eqESmapsA}
\begin{gathered}
H^1_{\proket}(X_{\bb{C}_p},  A^{\delta}_{\chi,\et}\widehat{\otimes}\widehat{\s{O}}_X )\xrightarrow{ES_{A}} H^0_{w_0}(X_{\bb{C}_p},  \omega_E^{\chi+\alpha})_{\epsilon}(\chi_2-1) \\
H^1_{1,c}(X_{\bb{C}_p}, \omega_E^{w_0(\chi)})_{\epsilon}(\chi_1)\xrightarrow{ES^{\vee}_{A}} H^1_{\proket}(X_{\bb{C}_p}, A^{\delta}_{\chi,\et}\widehat{\otimes} \widehat{\s{O}}_X ). 
\end{gathered}
\end{equation}
Dually, we have Galois and $U_p$-equivariant maps of cohomology groups
\begin{gather*}
H^1_{w_0,c}(X_{\bb{C}_p}, \omega_E^{-\chi})_{\epsilon}(-\chi_2)  \xrightarrow{ES^\vee_{D}}   H^1_{\proket}(X_{\bb{C}_p}, D^{\delta}_{\chi,\et}\widehat{\otimes}\widehat{\s{O}}_X)   \\ 
 H^1_{\proket}(X_{\bb{C}_p},  D^{\delta}_{\chi,\et}\widehat{\otimes}\widehat{\s{O}}_X)\xrightarrow{ES_{D}} H^0_{1}(X_{\bb{C}_p}, \omega_E^{-w_0(\chi)+ \alpha} )_{\epsilon}(-\chi_1-1). 
\end{gather*}
An analogous statement holds by exchanging $A^{\delta}_{\chi,\et}$ and $D^{\delta}_{\chi,\et }$ by  $j_{\proket,!} A^{\delta}_{\chi,et}$ and $j_{\proket,!}D^{\delta}_{\chi,\et}$, and $\omega^{\chi}_E$ by $\omega^{\chi}_{E}(-D)$. 
\proof
Let $\lambda:X_{\bb{C}_p, \proket}\rightarrow X_{\bb{C}_p, \an}$ be the natural projection.  First,  let us show that 
\begin{equation}
\label{eqprojOvmodsheaf}
\begin{gathered}
R\lambda_{*} (\omega_E^{\chi}\widehat{\otimes }\widehat{\s{O}}_X)  =   \omega_E^{\chi}[0] \oplus \omega_E^{\chi+\alpha}\otimes \bb{C}_p(-1)[-1] \\ 
R\lambda_{*} (\omega_E^{\chi}\widehat{\otimes }\widehat{\s{I}} )  =   \omega_E^{\chi}(-D)[0] \oplus \omega_E^{\chi+\alpha}(-D)\otimes\bb{C}_p (-1)[-1].
\end{gathered}
\end{equation}

 By Remark \ref{RemarkOvTorsorFromAnalytic},  the sheaf $\omega_E^{\chi,+}$ is an orthonormalizable $\s{O}^+_{X,\et}\widehat{\otimes }R^+$ sheaf locally for the \'etale topology of $X$.   Let $\nu: X_{\bb{C}_p,\proket}\rightarrow X_{\bb{C}_p,\ket}$ be the natural projection of sites. Then,  locally  \'etale, we can write $\omega_E^{\chi,+}\widehat{\otimes}\widehat{\s{O}}_X= \widehat{\bigoplus}_i \widehat{\s{O}}^+_X \widehat{\otimes} R^+ e_i$.  We get that
\begin{eqnarray*}
R\nu_{  *} (\omega_E^{\chi}\widehat{\otimes}\widehat{\s{O}}_X) & = &  R\nu_{*} (\omega_E^{\chi,+} \widehat{\otimes}\widehat{\s{O}}_X^+)   [\frac{1}{p}]  \\
							& = &R \varprojlim_s R\nu_{*} ( \bigoplus_i (\s{O}^+_X/p^s\otimes R^+) e_i) [\frac{1}{p}] \\
							& = & R\varprojlim_s \bigoplus_i( R\nu_{*} \widehat{\s{O}}_X^+/p^s\otimes R^+)e_i   [\frac{1}{p}] \\
							& = &  \omega_E^{\chi} \widehat{\otimes}_{R\widehat{\otimes} \widehat{\s{O}}_X}  R \nu_{*} (\widehat{\s{O}}_X\widehat{\otimes}R). 
\end{eqnarray*}
Since $R$ is a $\bb{Q}_p$-Banach space, it has a orthonormalizable basis over $\bb{Q}_p$, the same reasoning as before shows that $\omega_E^{\chi} \widehat{\otimes}_{R\widehat{\otimes} \widehat{\s{O}}_X}    R \nu_{*} (\widehat{\s{O}}_X\widehat{\otimes}R)= \omega_E^\chi \widehat{\otimes}_{\widehat{\s{O}}_X} R\nu_* \widehat{\s{O}}_X$.  On the other hand,    by Therorem \ref{TheoBGGproEtale} we know that $R \nu_* \widehat{\s{O}}_X= \s{O}_{X,\ket}[0] \oplus  \omega_E^{\alpha}(-1)[-1]$.  Lemma 5.5 of \cite{ScholzeHodgeTheory2013} and Lemma 6.17 of \cite {DiaoLogarithmic2019}  imply that the integral structure  obtained by $R\nu_{*}(\widehat{\s{O}}^+_X)$ defines the same topology of the one given by $\s{O}^+_{X,\ket}[0]\oplus \omega_E^{\alpha,+}(-1)[-1]$ (in fact, the cited lemmas show that both complexes differ just by bounded torsion when evaluated at affinoids).   Therefore 
\[
R\nu_{*} (\omega_E^{\chi} \widehat{\otimes} \widehat{\s{O}}_{X}) = \omega_E^{\chi} [0] \oplus \omega_E^{\chi+\alpha}(-1)[0] 
\]
over the Kummer-\'etale site of $X_{w}(\epsilon)$.   Finally,  let $\mu:X_{\bb{C}_p,\ket}\rightarrow X_{\bb{C}_p,\an}$ be the projection map.  In order to  descend to the analytic site we recall that $\omega_E^{\chi}$ is a projective Banach sheaf over  $X_{w}(\epsilon)$ (cf.  \cite[\S 5.5.2]{BoxerPilloniHigher2020}).  Thus,  it is a direct summand of an orthonormalizable Banach sheaf $\widehat{\bigoplus}_i \s{O}_X$ over $X_{w}(\epsilon)$.  But we know that the Kummer-\'etale cohomology of $\s{O}^+_{X}$ in affinoids admitting a Kummer-\'etale map to a torus $\bb{T}= \Spa(\bb{Q}_p\langle T^{\pm 1} \rangle,  \bb{Z}_p\langle  T^{\pm 1}\rangle)$ or a disc $\bb{D}=\Spa(\bb{Q}_p\langle U\rangle, \bb{Z}_p\langle U \rangle)$ has bounded torsion (by computing the cohomology via the  pullback of the perfectoid torus or disc, and using Lemma 5.5 of \cite{ScholzeHodgeTheory2013} or Lemma 6.1.7 of   \cite {DiaoLogarithmic2019}  again).    A similar argument as before  using derived limits shows that $R\mu_{*} (\widehat{\bigoplus}_i \s{O}_{X,\ket}) =\widehat{\bigoplus}_i \s{O}_{X,\an}$,  whence $R\mu_{*} \omega_E^{\chi}= \omega_E^{\chi}$.  Finally, to prove the second equality of \eqref{eqprojOvmodsheaf}, it is enough to show the analogous property for $\omega^{\chi}_E\widehat{\otimes}\iota_* \widehat{\s{O}}_D$, this follows from the previous argument applied to the log   adic space defined by the cusps (notice that even if $D$ is a disjoint union of points, the log structure is not trivial!). 

Next, we construct the overconvergent Eichler-Shimura maps; we only explain the procedure for the sheaf $A_{\chi,\et}^{\delta}$ and the pro-Kummer-\'etale cohomology, the case of $D_{\chi,\et}^{\delta}$ or the cohomology with compact support follows the same steps.  Consider the overconvergent Hodge-Tate maps  of Definition \ref{DefinitionOVHTModularCurve}
\begin{gather*}
\HT^{A,\vee}_{-w_0(\chi)}:  \omega_E^{w_0(\chi)}\widehat{\otimes } \widehat{\s{O}}_X(\chi_1) \rightarrow A^{\delta}_{\chi,\et}\widehat{\otimes}\widehat{\s{O}}_X \mbox{ over } X_{1}(\epsilon )\\
\HT^{A}_{\chi}: A^{\delta}_{\chi,\et}\widehat{\otimes}\widehat{\s{O}}_X \rightarrow    \omega_E^{\chi}\widehat{\otimes} \widehat{\s{O}}_X(\chi_2)  \mbox{ over } X_{w_0}(\epsilon).  
\end{gather*}
Taking the projection from the pro-Kummer-\'etale site to the analytic site,  one gets maps 
\begin{gather*}
\omega_E^{w_0(\chi)}\widehat{\otimes } R\lambda_{*} \widehat{\s{O}}_X(\chi_1) \rightarrow R\lambda_{*}(A^{\delta}_{\chi,\et}\widehat{\otimes} \widehat{\s{O}}_X)   \mbox{ over } X_{1}(\epsilon ) \\
R\lambda_{*}(A^{\delta}_{\chi,\et}\widehat{\otimes} \widehat{\s{O}}_X) \rightarrow \omega_E^{\chi} \widehat{\otimes} R \lambda_{*}\widehat{\s{O}}_X(\chi_2) \mbox{ over } X_{w_0}(\epsilon). 
\end{gather*}
Taking the overconvergent cohomologies of Definition \ref{DefiSupportOvcohomologies} and using \eqref{eqprojOvmodsheaf} we obtain maps
\begin{equation}
\label{eqESMaps1}
\begin{gathered}
R\Gamma_{1,c}(X_{\bb{C}_p},  \omega_E^{w_0(\chi)})_{\epsilon}(\chi_1)  \rightarrow R\Gamma_{1,c}(X_{\bb{C}_p}, R\lambda_{*}(A^{\delta}_{\chi,\et}\widehat{\otimes} \widehat{\s{O}}_X))_{\epsilon} \\
R\Gamma_{w_0}(X_{\bb{C}_p}, R\lambda_{*}(A^{\delta}_{\chi,\et}\widehat{\otimes}\widehat{\s{O}}_X))_{\epsilon} \rightarrow   R\Gamma_{w_0}(X_{\bb{C}_p},  \omega_E^{\chi+\alpha}))_{\epsilon}(\chi_2-1)[-1]. 
\end{gathered}
\end{equation}
On the other hand, we have restriction and correstriction maps 
\begin{equation}
\label{eqResCor}
 R\Gamma_{1,c}(X_{\bb{C}_p}, R\lambda_{*}(A^{\delta}_{\chi,\et}\widehat{\otimes} \widehat{\s{O}}_X))_{\epsilon} \xrightarrow{\Cor} R\Gamma_{\proket}(X_{\bb{C}_p},  A^{\delta}_{\chi,\et} \widehat{\otimes} \widehat{\s{O}}_X) \xrightarrow{\Res} R\Gamma_{w_0}(X_{\bb{C}_p}, R\lambda_{*}(A^{\delta}_{\chi,\et}\widehat{\otimes}\widehat{\s{O}}_X))_{\epsilon}.
\end{equation}
Taking $H^1$-cohomology  in \eqref{eqESMaps1}, and  composing with  the  morphisms of \eqref{eqResCor},  one obtains the maps (\ref{eqESmapsA}). The Galois equivariance is clear as the $\HT$-maps are defined over $X_{w}(\epsilon)\subset X$.   The compatibility with respect to the good normalization of the  $U_p$-operators follows from  Lemma   \ref{LemmaCompatibilityUp} and the fact that  $U_{p,\alpha}^{good}= U_{p,\alpha}$ for the correspondence associated to $\omega^{\alpha}_E$, see Definition \ref{DefiGoodNormalizations}. 
\endproof
\end{lem}

\begin{theo}
\label{TheoMain}
Let $\epsilon\geq \delta \geq n$, $(R,R^+)$ a uniform affinoid Tate $\bb{Q}_p$-algebra and $\chi: T=\bbf{T}(\bb{Z}_p)\rightarrow R^{+,\times}$ a $\delta$-analytic character. The following hold

\begin{enumerate}

\item 
The composition of the Eichler-Shimura maps $ES_{A}\circ ES^{\vee}_{A}$  is zero.  Consider the following sequence
\begin{equation}
\label{eqESsequence}
0\rightarrow H^1_{1,c}(X_{\bb{C}_p},  \omega_E^{w_0(\chi)})_{\epsilon}(\chi_1) \xrightarrow{ES_{A}^{\vee}}  H^1_{\proket}(X_{\bb{C}_p},  A^{\delta}_{\chi,\et}\widehat{\otimes} \widehat{\s{O}}_X) \xrightarrow{ES_{A}} H^0_{w_0}(X_{\bb{C}_p},  \omega_E^{\chi+\alpha})_{\epsilon}(\chi_2-1)\rightarrow 0.
\end{equation}

\item  Assume that $\n{V}=\Spa(R,R^+)$ is an affinoid subspace of the weight space $\n{W}_T$ of $T$, an let $\kappa=(k_1,k_2)\in \n{V}$ be a dominant weight of $\bbf{T}$.  Let $\alpha=(1,-1)\in X^*(\bbf{T})$ and let  $\chi=\chi^{un}_{\n{V}}$ be the universal character  of $\n{V}$. The following diagram commutes
 
\begin{center}
\begin{tikzpicture}[commutative diagrams/every diagram]
 \node (P0) at (0,2) {$H^1_{1,c}(X_{\bb{C}_p},\omega_E^{w_0(\chi)})_{\epsilon}(\chi_1) $};
 \node (P1) at (0,0) {$H^1_{1,c}(X_{\bb{C}_p},\omega_E^{w_0(\kappa)})_{\epsilon}(k_1) $};
 \node (P2) at (0,-2) {$H^1_{\an}(X_{\bb{C}_p},\omega_E^{w_0(\kappa)})(k_1) $};
 
 \node (P3) at (5,2) {$ H^1_{\proket}(X_{\bb{C}_p},  A^{\delta}_{\chi,\et} \widehat{\otimes}\widehat{\s{O}}_X)$};
 \node (P4) at (5,0) {$H^1_{\proket}(X_{\bb{C}_p},  A^{\delta}_{\kappa,\et} \widehat{\otimes}\widehat{\s{O}}_X)$};
 \node (P5) at (5,-2) {$H^1_{\proket}(X_{\bb{C}_p}, V_{\kappa,\et})\otimes \bb{C}_p $}; 
 
 \node (P6) at (10,2) {$H^0_{w_0}(X_{\bb{C}_p}, \omega_E^{\chi+\alpha})_{\epsilon}(\chi_2-1)$};
 \node (P7) at (10,0) {$H^0_{w_0}(X_{\bb{C}_p},\omega_E^{\kappa+\alpha})_{\epsilon}(k_2-1)$};
 \node (P8) at (10,-2) {$ H^0_{\an}(X_{\bb{C}_p}, \omega_E^{\kappa+\alpha})(k_2-1)  $};

\path[commutative diagrams/.cd, every arrow, every label] 
   (P0) edge node {$ES^\vee_{A}$}  (P3) 
   (P3) edge node {$ES_{A}$}  (P6)
   (P1) edge node {} (P4)
   (P4) edge node {} (P7)
   (P2) edge node {$ES^\vee$} (P5) 
   (P5) edge node {$ES$} (P8)
   
   (P0) edge node {} (P1)
   (P3) edge node {} (P4)
   (P6) edge node {} (P7)
   
   (P1) edge node {$\Cor$} (P2)
   (P5) edge node {} (P4)
   (P8) edge node [swap] {$\Res$} (P7);

\end{tikzpicture}
\end{center}

\item The maps of (2) are Galois and $U_p^t$-equivariant with respect to the good nomalizations of the Hecke operator (Definition \ref{DefiGoodNormalizations}).  In particular, the diagram above restricts to the finite slope part with respect to the action of  $U_p^t$.

\item  Let $h<k_1-k_2+1$.  There exists an open affinoid $\n{V}'\subset\n{V}$ containing $\kappa$ such that the $(\leq h)$-slope part  of  the restriction of  (\ref{eqESsequence})  to $\n{V}'$ is a short exact sequence of finite free $\bb{C}_p\widehat{\otimes}_{\bb{Q}_p} \s{O}(\n{V}')$-modules.      

\item Keep  the hypothesis of (4),  and let $\chi$ be the universal character of $\n{V}'$.  Let $\widetilde{\chi}= \chi_1-\chi_2+1 : \bb{Z}_p^\times \rightarrow R^{+,\times}$, and $b= \frac{d}{dt}|_{t=1} \widetilde{\chi}(t)$.  Then  we have a  Galois-equivariant split after inverting  $b$ 
\[
H^1_{\proket}(X_{\bb{C}_p}, \n{A}^{\delta}_{\chi}\widehat{\otimes}\widehat{\s{O}}_X)^{\leq h}_{b} =[H^1_{1,c}(X_{\bb{C}_p},  \omega_E^{w_0(\chi)})^{\leq h}_{\epsilon}(\chi_1)]_b \oplus [H^0_{w_0}(X_{\bb{C}_p},   \omega_E^{\chi+\alpha})^{\leq h}_{\epsilon}(\chi_2-1)]_b.
\] 
Analogous statements hold for the cohomology with compact support and for the sheaf $D_{\chi,\et}^{\delta}$.

\end{enumerate}
\end{theo}

\begin{proof}
Part  (1) follows from the fact that the composition of the restriction and correstriction maps  (\ref{eqResCor}) is zero.   

Parts (2)  and (3) follows from Lemma \ref{LemmaOVESMaps},   and the compatibility of the formation of $A^{\delta}_{\chi,\et}$ and $\omega_{E}^{\chi}$ with the character $\chi$.   The commutation of the lower  diagram is a direct consequence of the constructions and  Corollary \ref{CorAcompatibleV}.  

 For part (4) we follow the same arguments of \cite{AISOvShimura2015}:    the finite slope theory (cf. \cite{UrbanEigenvar,buzzard_2007}) implies that there is an affinoid open subspace $\n{V}'\subset\n{V}$ containing $\kappa$ such that the  $(\leq h)$-part of the sequence (\ref{eqESsequence}) restricted to $\n{V}'$ is a sequence of finite free $\bb{C}_p \widehat{\otimes}_{\bb{Q}_p}\s{O}(\n{V}')$-modules.    Moreover,  by the classicality Theorems  \ref{TheoClassicityModforms} and \ref{TheoClassicityModularSymbols}, and the classical Eichler-Shimura decomposition (Theorem \ref{TheoESDecompositionClassical}),  we can take $\n{V}'$ such that the $(\leq h)$-slope of the sequence (\ref{eqESsequence}) is short and exact.

Finally,    we briefly  sketch  the argument for part (5).  Let $\n{V}'$ be as   in  (4),  let $R=\s{O}(\n{V}')$,  and consider the short exact sequence of the $(\leq h)$-slope of   (\ref{eqESsequence}).   Taking basis  and tensoring with the Tate twist $R(1-\chi_2)$ we are left to prove that the localization by $b$ of $H^1(G_{\bb{Q}_p},  \bb{C}_p \widehat{\otimes} R(\chi_1-\chi_2+1))$ vanishes.  By almost \'etale descent one has 
\begin{equation}
\label{eqExtensionSen}
H^1(G_{\bb{Q}_p},  \bb{C}_p \widehat{\otimes} R(\chi_1-\chi_2+1))=H^1(\Gal(\bb{Q}_p^{\cyc}/\bb{Q}_p),  \bb{Q}_p^{\cyc}\widehat{\otimes}_{\bb{Q}_p}R(\chi_1-\chi_2+1)). 
\end{equation}
We identify $\Gal(\bb{Q}_p^{\cyc}/\bb{Q}_p)$ with $\bb{Z}_p^\times$ via $\chi_{\cyc}$.   By Sen theory,  to show that  (\ref{eqExtensionSen}) is of $b$-torsion it is enough to prove that $H^1(\Lie \bb{Z}_p^\times,  R(\chi_1-\chi_2+1))_b=0$,  but this is clear as $H^1(\Lie \bb{Z}_p^\times,  R(\chi_1-\chi_2+1))\cong  R/ b R $.
\end{proof}

\subsubsection{Previous works in the literature}

In this short paragraph we discuss some previous works and their connection with Theorem \ref{TheoMain}.  Our main result is the complement of the work of Andreatta-Iovita-Stevens; they have constructed the map $\ES_{D}$ from $H^1_{\proet}(Y_{\bb{C}_p}, D^{\delta}_{\chi,\et})\widehat{\otimes }\bb{C}_p$ to the space of overconvergent modular forms of weight $-w_0(\chi)$, cf.   \cite[Theorem  6.1]{AISOvShimura2015}. Our theorem constructs a new map  from the overconvergent $H^1$-cohomology with supports of higher Coleman theory \cite{BoxerPilloniHigher2020}, to the space of modular symbols defined by the distributions. In addition, we have also discussed the dual picture with the principal series, and with the  pro\'etale cohomology with compact support instead.

On the other hand, the first work in the subject which uses the perfectoid modular curve to construct the $\ES_{D}$ map goes back to Chojecki-Hansen-Johansson. Aditionally, they constructed the map for Shimura curves, and they have translated an analogous of Theorem  \ref{TheoMain} in terms of the eigencurve, see  \cite[Theorem 5.14]{ChojeckiOvShimuraCurves2017}.

The work of Sean Howe \cite{SeanHoweOverconvergent} studies   natural pairings between some local cohomologies attached to the flag variety $\Fl= \bb{P}^1_{\bb{Q}_p}$ and overconvergent modular forms, these  take values in the  locally analytic vectors of the completed cohomology of the modular curve.  The local cohomologies are nothing but the cohomology with supports $\varprojlim_{\epsilon} H^1_{\overline{U_w(\epsilon)}}(\bb{P}^1_{\bb{Q}_p}, \s{L}(\chi))$ or the overconvergent cohomology $\varinjlim_{\epsilon} H^1(U_w(\epsilon), \s{L}(\chi))$, see  Lemma 4.3.1 and Remark 1.2.12 of  \cite{SeanHoweOverconvergent} (in the notation of \textit{loc. cit.} one has $0=[0:1]$, which is represented by $1\in \GL_2$). It is expected that these  pairings provide a more geometric interpretation of the $\ES_{A}^{\vee}$ map of Theorem \ref{TheoMain}, namely, they are closely  related to the description the of completed cohomology of Lue Pan that we briefly explain down below. 
 
 In his recent work \cite{pan2020locally}, Lue Pan gave an exhausting description of the $\chi$-isotypoic part of the locally analytic vectors of the completed cohomology   for the action of the Borel algebra  $\mathrm{Lie} \bbf{B}$,  see \cite[Theorem 1.0.1]{pan2020locally}. His method uses  a new tool in $p$-adic Hodge theory which is the geometric Sen operator, then, using the dictionary between representation theory  over $\Fl $ and pro\'etale sheaves over the modular curve provided by $\pi_{\HT}$, he showed that the completed cohomology can be decomposed in terms of overconvergent modular forms. The intersection  between locally analytic vectors of completed cohomology and Theorem \ref{TheoMain} lies in the fact that the cohomology group $H^1_{\proet}(Y_{\bb{C}_p}, A^{\delta}_{\chi,\et})$ is a subspace of the locally analytic completed cohomology, consisting on those $\delta$-analytic cohomology  classes admitting an action of $\Iw_n$, such that $\bbf{B}(\bb{Z}_p)\cap \Iw_n$ acts via $-\chi$. Finally, the maps $\ES_A^{\vee}$ and $\ES_{A}$ are instances  of the spaces $M_{\mu,1}$ and $M_{\mu,w}$ appearing in \cite[Theorem 5.4.2]{pan2020locally}.

 \subsubsection{The pairings} 
We end this section with the compatibility of the overconvergent Eichler-Shimura maps and Poincar\'e and Serre duality.  Let $\epsilon \geq \delta \geq n$,  let  $(R,R^+)$ and $\chi$ be as in previous sections.  By definition there is a natural pairing between the $\delta$-principal series and distributions 
 \[
 A^{\delta}_{\chi} \times D^{\delta}_{\chi}\rightarrow R.
 \]
It is easy to see that it induces  a Poincar\'e  pairing 
\begin{equation*}
\label{EQYoneda}
\langle - , - \rangle_P : H^1_{\proet,c}(Y_{\bb{C}_p},  D^{\delta}_{\chi,\et}(1))\times H^1_{\proet}(Y_{\bb{C}_p},  A^{\delta}_{\chi,\et} )\rightarrow H^2_{\proet,c}(Y_{\bb{C}_p},  R(1)) \xrightarrow{\Tr_P} R
\end{equation*}
where  the first arrow is a Yoneda pairing,  and the last arrow is induced by  the Poincar\'e trace $H^1_{\proet,c}(Y_{\bb{C}_p},  \bb{Z}_p(1))\xrightarrow{\Tr_P} \bb{Z}_p$.

On the other hand, in \cite{BoxerPilloniHigher2020} the authors have  defined  overconvergent  Serre pairings in families 
\begin{equation*}
\label{SerrePairing}
\langle -, -\rangle_S: H^1_{w,c}(X_{\bb{C}_p}, \omega_E^{-\chi}(-D))_{\epsilon}\times  H^{0}_{w}(X_{\bb{C}_p},\omega_E^{\chi+\alpha})_{\epsilon} \rightarrow \bb{C}_p\widehat{\otimes} R
\end{equation*}
compatible with the classical Serre pairings.  The  pairings are constructed by taking the Yoneda's product 
\[
\cup : H^1_{w,c}(X_{\bb{C}_p},\omega_E^{-\chi}(-D))_{\epsilon} \times H^{0}_{w}(X_{\bb{C}_p},\omega_E^{\chi+\alpha})_{\epsilon}  \xrightarrow{\cup} H^1_{w,c}(X_{\bb{C}_p},  \omega_E^{\alpha}(-D)\widehat{\otimes}R)_{\epsilon}\xrightarrow{KS} H^1_{w,c}(X_{\bb{C}_p}   \Omega_X^1\widehat{\otimes } R)_{\epsilon}
\]
and composing with the  Serre trace map of $X$
\[
H^1_{w,c}(X_{\bb{C}_p},\Omega_X^1\widehat{\otimes }R )_{\epsilon}  \xrightarrow{\Cor} H^1_{\an}(X_{\bb{C}_p}, \Omega^1_{X}\widehat{\otimes}R) \xrightarrow{\Tr_S} \bb{C}_p\widehat{\otimes} R. 
\]
\begin{theo}
\label{TheoPairingOVES}
Keep the notation of Theorem \ref{TheoMain}.  The following hold

\begin{enumerate}
\item The Poincar\'e and Serre pairings of overconvergent cohomologies are compatible with the good normalizations of the $U_p$-operators (Definition \ref{DefiGoodNormalizations}).  Moreover, they are  compatible with the classical  Eichler-Shimura maps of Theorem  \ref{TheoESDecompositionClassical}.

 \item Let $\n{W}_T$ be the weight space of $T= \bbf{T}(\bb{Z}_p)$,  let  $\n{V}\subset \n{W}_T$ be an open affinoid and let  $\chi=\chi^{\un}_{\n{V}}$ be the universal character of $\n{V}$.  Let $\kappa=(k_1,k_2)\in \n{V}$ be a classical weight and   fix  $h<k_1-k_2+1$.   There exists an open affinoid  $\n{V}'\subset \n{V}$ containing $\kappa$ such that  we have perfect pairings of finite free $\bb{C}_p\widehat{\otimes} \s{O}(\n{V}')$-modules
\[
\langle -,- \rangle_P:  H^1_{\proet,c}(Y_{\bb{C}_p},  D^{\delta}_{\chi,\et}(1))^{\leq h}\times H^1_{\proet}(Y_{\bb{C}_p},   A^{\delta}_{\chi,\et} )^{\leq h} \rightarrow  \s{O}(\n{V}')
\]
and 
\begin{gather*}
\langle -, -\rangle_S: H^1_{w,c}(X_{\bb{C}_p}, \omega_E^{-\chi}(-D))^{\leq h}_{\epsilon}\times  H^{0}_{w}(X_{\bb{C}_p},\omega_E^{\chi+\alpha})^{\leq h}_{\epsilon} \rightarrow \bb{C}_p\widehat{\otimes} \s{O}(\n{V}') \\ 
\langle -, -\rangle_S: H^1_{w,c}(X_{\bb{C}_p}, \omega_E^{w_0(\chi)})^{\leq h}_{\epsilon}\times  H^{0}_{w}(X_{\bb{C}_p},\omega_E^{-w_0(\chi)+\alpha}(-D))^{\leq h}_{\epsilon} \rightarrow  \bb{C}_p\widehat{\otimes} \s{O}(\n{V}') ,
\end{gather*}
compatible with the  overconvergent Eichler-Shimura maps.
\end{enumerate}

\proof
The Hecke operators are compatible with the pairings by their definition via finite flat correspondances,    see Definitions \ref{DefUpOverconvergentModular} and \ref{DefUpModularSymbols}.  

In the following we forget the Galois action.  Let  $\lambda:  X_{\bb{C}_p, \proket} \rightarrow X_{\bb{C}_p, \an}$ be the projection of sites.    We have a commutative diagram of Yoneda's products 
\[
\begin{tikzcd}[column sep= 10pt]
H^1_{1,c}( X_{\bb{C}_p},  R\lambda_{*} (\omega_E^{w_0(\chi)} \widehat{\otimes} \widehat{\s{O}}_X) )_{\epsilon} \times H^1_{1}(X_{\bb{C}_p}, R\lambda_{*}( \omega_E^{-w_0(\kappa)} \widehat{\otimes}\widehat{\s{I}}))_{\epsilon} \ar[r] \ar[d, shift right = 1.5 cm] & H^2_{1,c}(X_{\bb{C}_p},  R\lambda_{*}  (R\widehat{\otimes}  \widehat{\s{I}}) )_{\epsilon} \ar[d, "\Cor"] \\ 
H^1_{\proket}(X_{\bb{C}_p}, A^{\delta}_{\chi,\et}\widehat{\otimes}\widehat{\s{O}}_X) \times H^1_{\proket}(X_{\bb{C}_p}, D^{\delta}_{\chi,\et}\widehat{\otimes}\widehat{\s{I}}) \ar[r]  \ar[d , shift right = 1.5 cm]  \ar[u, shift right =1.5 cm] & H^2_{\proket}(X_{\bb{C}_p},  R\widehat{\otimes}\widehat{\s{I}}) \\ 
H^1_{w_0}(X_{\bb{C}_p},  R\lambda_{*}( \omega_E^{\chi}\widehat{\otimes} \widehat{\s{O}}_X))_{\epsilon} \times H^1_{w_0,c}(X_{\bb{C}_p}, R\lambda_{*}(\omega_E^{-\chi}\widehat{\otimes} \widehat{\s{I}}))_{\epsilon} \ar[r]  \ar[u,  shift right = 1.5 cm]& H^2_{w_0,c}(X_{\bb{C}_p}, R\lambda_{*}( R\widehat{\otimes} \widehat{\s{I}}))_{\epsilon} \ar[u, "\Cor"']
\end{tikzcd}
\]
 On the other hand, we also have compatible pairings provided by the Faltings extension, cf. \cite[Corollary 6.14]{ScholzeHodgeTheory2013}
\[
\begin{tikzcd}
H^1_{w,c}(X_{\bb{C}_p}, \omega_E^{-\chi}(-D))_{\epsilon} \times H^{0}_{w}(X_{\bb{C}_p},\omega_E^{\chi+\alpha})_{\epsilon} \ar[r, "\Cor\circ \cup"]  \ar[d, shift right = 1.5 cm] & H^1_{\an}(X_{\bb{C}_p},R\widehat{\otimes}  \Omega^1_X)  \ar[d, "FE"] \\ 
H^1_{w_0,c}(X_{\bb{C}_p}, R\lambda_{*}(\omega_E^{-\chi}\widehat{\otimes} \widehat{\s{I}}))_{\epsilon} \times H^1_{w_0}(X_{\bb{C}_p},  R\lambda_{*}( \omega_E^{\chi}\widehat{\otimes} \widehat{\s{O}}_X))_{\epsilon}\ar[r, "\Cor\circ \cup"] \ar[u,  shift right = 1.5 cm]  & H^2_{\proket}(X_{\bb{C}_p}, R\widehat{\otimes } \widehat{\s{I}}).
\end{tikzcd}
\]
The compatibility of Poincar\'e and Serre traces (\cite[Theorem 4.4.1 (4)]{LanLiuZhuRhamComparison2019})  implies part (1).    Part (2)  follows the same lines of the proof of Theorem \ref{TheoMain} using the fact that  the pairings are perfect for the classical Eichler-Shimura decomposition. 
\endproof
\end{theo}


\bibliographystyle{alpha}
\bibliography{ESetale}

\end{document}